\documentclass{article}

\usepackage{PRIMEarxiv}
\usepackage[utf8]{inputenc}
\usepackage[T1]{fontenc}
\usepackage[numbers,sort&compress]{natbib}
\usepackage{hyperref}
\usepackage{url}
\usepackage{booktabs}
\usepackage{nicefrac}
\usepackage{microtype}
\usepackage{fancyhdr}
\usepackage{float}
\usepackage[font=small,labelfont=bf]{caption}
\usepackage{enumitem}

\usepackage{amsmath, amsfonts, amssymb, amsthm, mathtools}
\usepackage{bm}
\usepackage{multirow}
\usepackage{siunitx}
\usepackage{algorithm}
\usepackage{algpseudocode}
\usepackage{tikz}
\usepackage{tikz-cd}
\usepackage{pgfplots}
\pgfplotsset{compat=1.18}

\usepackage[nameinlink,noabbrev,capitalize]{cleveref}

\usepackage[acronym]{glossaries-extra}
\setabbreviationstyle[acronym]{long-short}
\glssetcategoryattribute{acronym}{nohyperfirst}{true}
\newacronym{acr:amod}{AMoD}{Autonomous Mobility-on-Demand}
\newacronym{acr:mod}{MoD}{Mobility-on-Demand}
\newacronym{acr:pt}{PT}{Public Transit}
\newacronym{acr:mcp}{MCP}{Municipality}
\newacronym{acr:od}{OD}{Origin-Destination}
\newacronym{acr:mnl}{MNL}{Multinomial Logit}
\newacronym{acr:kkt}{KKT}{Karush-Kuhn-Tucker}
\newacronym{acr:micp}{MICP}{Mixed Integer Convex Program}
\newacronym{acr:sca}{SCA}{Successive Convex Approximation}
\newacronym{acr:ibr}{IBR}{Iterative Best Response}
\newacronym{acr:osm}{OSM}{OpenStreetMap}

\newcommand{\agraph}{\mathcal{G}^a}
\newcommand{\anodes}{\mathcal{V}^a}
\newcommand{\aedges}{\mathcal{E}^a}
\newcommand{\tup}[1]{\left(#1\right)}
\newcommand{\ptl}{L}
\newcommand{\ptgraph}{\mathcal{G}^{pt}}
\newcommand{\ptnodes}{\mathcal{V}^{pt}}
\newcommand{\ptedges}{\mathcal{E}^{pt}}
\newcommand{\wgraph}{\mathcal{G}^{w}}
\newcommand{\wnodes}{\mathcal{V}^{w}}
\newcommand{\wedges}{\mathcal{E}^{w}}
\newcommand{\od}{\tilde D}

\theoremstyle{plain}
\newtheorem{theorem}{Theorem}
\newtheorem{lemma}{Lemma}
\newtheorem{proposition}{Proposition}
\newtheorem{corollary}{Corollary}
\newtheorem{assumption}{Assumption}
\theoremstyle{definition}
\newtheorem{definition}{Definition}
\newtheorem{remark}{Remark}
\renewenvironment{proof}[1]{\par\noindent\textit{#1.}\enskip\renewcommand{\qed}{}}{\hfill$\square$\par}
\newenvironment{APPENDICES}{\appendix}{}


\newcommand{\revised}[1]{#1}
\newenvironment{revisedblock}{}{}
\newcommand{\gzadd}[1]{#1}
\newcommand{\gzdel}[1]{}
\newcommand{\gzmoved}[1]{}

\title{Implementation-Based Incentive Design for Autonomous Mobility-on-Demand and Transit Systems}

\author{
  \textbf{Xinling Li}$^{1}$ \quad
  \textbf{Runyu Zhang}$^{1}$ \quad
  \textbf{Gioele Zardini}$^{1}$ \\
  $^{1}$Laboratory for Information\&Decision Systems, Massachusetts Institute of Technology, Cambridge, MA, USA \\
  \texttt{xinli831@mit.edu}, \texttt{runyuzha@mit.edu}, \texttt{gzardini@mit.edu}
}

\date{}

\begin{document}
\maketitle

\begin{abstract}
A municipality may know a socially desirable operating point for a multimodal transportation system, but achieving that point is challenging when \gls{acr:amod} and \gls{acr:pt} operators pursue selfish objectives alongside endogenous passenger choices.
Existing equilibrium-based regulation models typically search over municipal policies and then predict the induced operator equilibrium, which creates strong behavioral assumptions, equilibrium-selection issues, and difficult bilevel or multilevel optimization problems.
This paper proposes an implementation-based alternative.
Rather than asking which municipal action induces the best equilibrium, we ask: given a target operation profile, what is the minimum total payment the municipality must commit to the operators at that target so that neither operator benefits from unilateral deviation?
Using $k$-implementation theory, this payment decomposes into two unilateral deviation gains, one for the \gls{acr:amod} operator and one for the \gls{acr:pt} operator.
Calculating this payment requires computing three distinct objects: the social target, the \gls{acr:amod} best response, and the \gls{acr:pt} best response. Operationalizing these computations is nontrivial because each represents a large-scale network optimization problem complicated by endogenous mode choice and congestion. To address this complexity, we develop tailored mathematical formulations and algorithms for each of the three corresponding oracles.
For the social target oracle, we derive a decomposition and entropy-regularized mixed-integer convex formulation that balances social optimality and strategy implementability.
For the \gls{acr:amod} oracle, we derive an exact reformulation, a convex relaxation providing a global upper bound, and a sequential convex approximation procedure producing feasible lower bounds.
For the \gls{acr:pt} oracle, we develop a mixed-integer convex relaxation and characterize its exactness condition.
A Manhattan, NYC case study shows that the proposed framework computes tight implementation-payment bounds and reveals how the dominant source of incentive misalignment shifts with congestion.
\end{abstract}


\section{Introduction}
Rapid urbanization strains existing transportation infrastructure, and traffic congestion imposes substantial economic and environmental costs on cities, including lost productivity and increased emissions~\citep{schrank20112011}.\gzdel{Rapid urbanization continues to place a severe strain on existing transportation infrastructure. Traffic congestion has emerged as a pervasive challenge for urban areas globally, resulting in profound economic and environmental externalities, including lost productivity and increased emissions.}
Urban transportation agencies increasingly face a coordination problem rather than a purely technological one. 
Public transit provides high-capacity service on fixed corridors, while on-demand mobility, \revised{provided by human-driven ride-hailing fleets and increasingly by autonomous robotaxi fleets,} provides flexible door-to-door service. 
These services can be complementary when coordinated, but they can also compete for riders, duplicate service, and worsen congestion through additional vehicle miles traveled and empty rebalancing trips~\citep{oh2021impacts,nahmias2021evaluating}. 
The operational decisions that determine this interaction are made at different institutional layers: an on-demand operator chooses prices, route assignments, and rebalancing flows; a \gls{acr:pt} operator chooses service frequencies subject to capacity and operating-cost considerations; passengers respond through mode choice; and the \gls{acr:mcp} evaluates outcomes through a social objective that internalizes travel time, operating cost, and externalities. \revised{Throughout, we study on-demand mobility under centralized fleet control and refer to it as \gls{acr:amod}. An \gls{acr:amod} operator controls its vehicle fleet centrally, including routing and rebalancing, whereas in conventional ride-hailing human drivers retain routing, cruising, and participation decisions.
Robotaxi services (e.g., Waymo and Zoox) implement this control structure, but they currently operate at limited scale in a small number of cities.
Our analysis therefore addresses a forward-looking scenario in which centrally controlled fleets serve a substantial share of urban trips, which is precisely the scenario in which regulation must be designed before the market structure settles.}
%Such coordination can improve fleet-level operational efficiency, and the robotaxi paradigm (i.e., Waymo and Zoox) provides a practical means of implementing this control structure. This differs from conventional ride-hailing, in which human drivers retain routing, cruising, and participation decisions.}

The resulting regulatory problem is difficult for two reasons. 
First, the socially preferred operating point is generally not privately optimal for either operator. 
The \gls{acr:amod} operator internalizes fare revenue and fleet costs, but not the full congestion externality imposed on the road network. 
The \gls{acr:pt} operator may be more aligned with public goals, but its frequency decisions are still governed by its own operating objective and institutional constraints~\citep{gomez2011going}. \revised{Misalignment can also arise from positive service externalities. Higher transit ridership can support frequencies that reduce waiting time for all riders, while denser on-demand markets can improve matching and pickup performance~\citep{castillo2023matching,mohring1972optimization}. A profit-oriented operator does not fully internalize either benefit,
so private and social service levels diverge in both directions.}
Second, the technical objects needed to evaluate this misalignment are themselves challenging network optimization problems. 
Passenger mode choice depends on prices and service quality; \gls{acr:amod} routing and rebalancing create endogenous road congestion; and transit frequencies affect both waiting times and capacity. 
\revised{Thus, the incentive payment cannot be read off the solution of a co-optimization problem. Evaluating it requires solving each operator's own service-design problem to a certified accuracy.}

%A significant part of the complexity in managing the multimodal mobility system results from a critical misalignment between the overall system efficiency and selfish operators' incentives. 
%While multimodal transportation systems function as essential societal infrastructure to achieve optimized system-wide welfare, typically measured by minimizing total travel time, operational costs, and environmental externalities, the \gls{acr:mod} operators often act as self-interested profit maximizers. Their operational strategies, including pricing, routing, and rebalancing, are optimized for private revenue rather than public benefit. Consequently, the operation of \gls{acr:mod} has been shown to exacerbate road congestion and total vehicle miles traveled~\cite{oh2021impacts,nahmias2021evaluating}. Meanwhile, though \gls{acr:pt} operators are often regarded as more aligned with the public good, they still operate as independent agents governed by their own internal utility metrics and institutional constraints~\cite{gomez2011going}.

The dominant approach in the emerging literature is to model regulation as an equilibrium-based hierarchical game, in which the municipality selects a policy and anticipates the equilibrium response of the operators~\citep{zardini2021game,dandl2021regulating,he2025hierarchical}. This is an intuitive approach, but it creates a demanding modeling pipeline. 
Each candidate municipal action induces an operator game that may have multiple equilibria and inherits the nonconvexities of passenger choice and congestion, and operators are assumed to reason through the complete strategic environment and coordinate on the relevant equilibrium.\gzdel{The municipality must search over a regulatory action space, each candidate action induces a lower-level operator game, the induced game may have multiple equilibria, and the equilibrium constraints inherit the nonconvexities created by passenger choice and network congestion. 
The behavioral assumption is also strong: operators are assumed to reason through the complete strategic environment and coordinate on the relevant equilibrium.} 
For a regulator, however, the central question is often more direct: given a socially desirable operating point, what is the minimum intervention needed to make unilateral deviation unattractive?

This paper answers that question using $k$-implementation theory~\citep{monderer2003k}. 
In the implementation view, the municipality is an interested party that cannot redesign the game, prohibit actions, or force operator behavior, but can commit to nonnegative transfers contingent on the realized operator profile. 
For a target profile $z$, the minimum realized payment required to implement $z$ is the sum of the operators' unilateral deviation gains at that target. 
This converts the regulatory problem from an equilibrium-search problem into a target-implementation problem. 
Instead of optimizing over municipal policies and solving an induced equilibrium game, the municipality must compute a target profile and two conditional best responses. This shift organizes the paper. Specifically, we do not simply import an existing game-theoretic concept into a transportation setting. The non-obvious step, \revised{and the focus of this paper}, is to make the implementation formula operational for a realistic multimodal network, where the utilities in the deviation terms are not closed-form functions but outcomes of large-scale service-design problems with endogenous passenger choice. 
Applying $k$-implementation therefore requires solving three coupled but separable optimization tasks: a social target problem, an \gls{acr:amod} deviation problem, and a \gls{acr:pt} deviation problem. 
Each task corresponds to a central class of transportation models, and each has its own nonconvexities. 
The paper develops the formulations and algorithms needed to compute these objects with explicit solution-quality certificates.
% \gzadd{For the best responses, a feasible solution is not enough, because an underestimated deviation value leaves a profitable deviation uncompensated. Each best-response value must therefore be bounded from above.}

We develop one oracle per object. The target oracle is a joint \gls{acr:amod}--\gls{acr:pt} social-cost model with binary-logit mode choice, transit capacity, and road congestion endogenous to \gls{acr:amod} fleet flows. Because \gls{acr:amod} prices enter the social target only through mode choice, the model admits a price-recovery decomposition that removes the logit constraint. An entropy term then avoids \revised{large price spreads}, and its weight trades strict social-cost minimization against price implementability. \revised{We call a target solution \emph{price-implementable} if there exist finite \gls{acr:amod} prices that induce the mode split in this solution. Under a random-utility choice model, no finite price produces an all-or-nothing split, so a deterministic system-optimal assignment is in general not price-implementable. Price implementability concerns passengers' response to \gls{acr:amod} prices and differs from the implementation of the target in $k$-implementation, which concerns whether the payments incentivize the operators to choose the target.} 
The \gls{acr:amod} oracle computes the operator's best response over prices, passenger route assignments, and rebalancing for fixed transit frequencies. We derive an exact reformulation that isolates the nonconvexity in a single congestion--rebalancing term, whose convex relaxation gives a global upper bound on the best-response profit and whose trust-region sequential convexification gives feasible incumbents. The \gls{acr:pt} oracle computes the transit frequency design for fixed \gls{acr:amod} operations. We relax the logit equality into a convex demand envelope and characterize when the relaxation is exact. When it is not, the slack identifies a capacity-constrained operating regime. The difference between each best-response value and the operator's utility at the target is its deviation gain, and the sum of the two gains is the implementation payment.

We evaluate the framework on a case study of Manhattan, New York City. \revised{The Manhattan study is intended as an empirically grounded demonstration of the framework under a tractable behavioral and network specification, rather than as a universal representation of the multimodal system or a transferable estimate of implementation payments across cities. In the appendices, we repeat the Manhattan experiments with varying parameter setups, and show that although these changes affect the system status on the passenger experience level, the principal incentive conclusions remain robust.}
The experiments show that the proposed oracles are computationally practical at urban scale and that the implementation payment is economically informative. 
% In the tested regimes, the payment is not monotone in background traffic. 
% Instead, it is largest at low and high congestion levels, revealing two different sources of incentive misalignment. 
% Under light congestion, the main friction is market-power behavior by \gls{acr:amod} and frequency incentives by \gls{acr:pt}. 
% Under heavy congestion, the friction shifts toward road externalities and under-provision of socially valuable transit service. 
% Thus, the payment serves as a diagnostic of which operator incentive is misaligned and why.

\textit{Statement of contributions.} Our contributions are fivefold.\gzdel{Specifically, this paper makes the following contributions.}
\revised{First, we develop a target-oracle model for socially preferred multimodal operation with explicit mode choice and endogenous \gls{acr:amod} congestion. We prove a price-recovery decomposition and derive a mixed-integer convex formulation after entropy regularization. Second, we derive an \gls{acr:amod} best-response reformulation that isolates the structural nonconvexity, and supports both a convex global upper bound and a sequential convexification algorithm with a certified optimality gap. Third, we derive a \gls{acr:pt} best-response mixed-integer convex relaxation and provide an exactness condition that distinguishes demand-constrained regimes from capacity-constrained regimes. Together, these oracles give certified bounds on each operator's best-response value, which is what makes the computed implementation payment provably sufficient. Fourth, we use these oracles to operationalize $k$-implementation for multimodal \gls{acr:amod}--\gls{acr:pt} regulation, replacing the outer policy search of equilibrium-based models with one target problem and two best-response problems. Finally, we demonstrate the framework on a large-scale Manhattan instance and show that the resulting implementation payment quantifies both the magnitude and the operational source of incentive misalignment.}
%\begin{enumerate}
%\item We formulate an implementation-based regulatory framework for multimodal \gls{acr:amod}--\gls{acr:pt} systems. 
%    The framework replaces an equilibrium-based municipal policy search with a constructive target-implementation calculation based on unilateral deviation gains.
%\item We develop a target-oracle model for socially preferred multimodal operation with explicit mode choice, endogenous \gls{acr:amod} congestion, rebalancing, transit frequency design, and transit capacity. 
%We prove a price-recovery decomposition and derive a mixed-integer convex formulation after entropy regularization.
%\item We derive an \gls{acr:amod} best-response reformulation that eliminates prices, isolates the structural nonconvexity, and supports both a convex global upper bound and a sequential convexification algorithm with a certified optimality gap.
%\item We derive a \gls{acr:pt} best-response mixed-integer convex relaxation and provide an exactness condition that distinguishes demand-constrained regimes from capacity-constrained regimes.
%    \item We demonstrate the framework on a large-scale Manhattan instance and show that the resulting implementation payment quantifies both the magnitude and the operational source of incentive misalignment.
%\end{enumerate}

The remainder of this paper is organized as follows. \cref{sec:review} reviews the relevant literature on multimodal design, strategic regulation, and implementation. 
\cref{sec:pre} defines the multimodal interaction model and introduces the implementation concept. \cref{sec:Method} presents the modular framework and the three optimization oracles. 
\cref{sec:Results} reports the computational study.
\cref{sec:Conclusion} concludes and discusses extensions. The appendices contain the proofs, algorithmic details, extensions of the choice model, the case-study setup, and robustness analyses.

\section{Literature Review}\label{sec:review}
This paper lies at the intersection of three literature streams: centralized multimodal system design, strategic regulation of decentralized mobility operators, and implementation incentive design. The target oracle builds on the first stream, the operator best responses on the second, and the payment characterization on the third.
\gzdel{The first stream studies how transit and on-demand services should be configured when the relevant decisions are aligned under a common objective. 
The second stream acknowledges that operators and the municipality have different objectives and therefore models the system as a game. 
The third stream asks how an interested party can induce a desired behavior through transfers without redesigning the underlying game.
Our contribution combines all three: the municipal target is computed through a coordinated system-design model, the decentralized operator interaction is inherited from the game-theoretic literature, and the regulatory payment is derived from~$k$-implementation rather than from equilibrium search.
}

%The proliferation of \gls{acr:mod} has catalyzed significant research interest in regulating the integrated transit and \gls{acr:mod} systems to achieve a better multimodal system design. Existing research diverges primarily on the behavioral assumption governing the operators. One paradigm assumes centralized, integrated optimization with perfectly aligned objectives. The alternative paradigm acknowledges the self-interested, strategic nature of independent operators competing or cooperating within a shared network. In this section, we review the methodological contributions from both perspectives.

Two remarks on scope are in order. First, most works reviewed below assume centralized fleet control but refer to the service as \gls{acr:mod}. Human-driven fleets are subject to driver-side cruising and participation decisions~\citep{brar2026mean}, whereas autonomous fleets enable centralized control~\citep{li2026reproducibility}, so we use \gls{acr:amod} throughout to denote the centralized paradigm. Second, we use ``multimodal'' in the narrow sense of the \gls{acr:pt}--\gls{acr:amod} interaction and review congestion pricing, micromobility integration, and transit fare policy only where they inform this regulation problem.
\gzdel{Before reviewing these streams, we clarify scope and terminology. 
First, \gls{acr:mod} encompasses both human-driven fleets (e.g., Uber, Lyft) and autonomous fleets (e.g., Waymo, Zoox).
At the fleet coordination level, the key difference between these two services is that the human-driven \gls{acr:mod} services are subject to driver-side uncertainties, i.e., cruising and participation behavior~\cite{brar2026mean}, whereas the \gls{acr:amod} systems enable centralized fleet control~\cite{li2026reproducibility}.
 While all the works reviewed in the section assume centralized fleet control, many do not distinguish between human-driven and autonomous \gls{acr:mod}, referring to it as \gls{acr:mod} when it actually means \gls{acr:amod}.
To highlight this difference, we explicitly use \gls{acr:amod} to emphasize the centralized control paradigm.
Likewise, we use ``multimodal'' in the narrow sense of the interaction between \gls{acr:pt} and \gls{acr:amod}. 
We do not attempt a comprehensive review of the much broader literatures on general congestion pricing, micromobility integration, or transit fare policy except where those works directly inform the present \gls{acr:pt}--\gls{acr:amod} regulation problem.
}

\subsection{Centralized and Integrated Multimodal System Design}
The first stream assumes that transit and \gls{acr:amod} services are under aligned objectives, either because a single planner controls both services or because the two operators are treated as cooperative.
\gzdel{This literature is directly relevant because the target-strategy oracle in our framework solves precisely such a coordinated design problem.}

A first class of papers treats passenger flows as planner-controlled variables and optimizes them jointly with service design toward a system objective~\citep{salazar2019intermodal,luo2021multimodal,auad2022ridesharing}. 
These works are valuable for system planning and network redesign, but the induced passenger allocation is effectively chosen by the planner. 
In the present paper, this is too strong as a market description, because travelers remain free to choose between \gls{acr:pt} and \gls{acr:amod} in response to service quality and price.

A second class improves behavioral realism by embedding passenger response through user-equilibrium or economic-equilibrium formulations~\citep{liu2025integrating,zhang2021integrating,ma2023economic}. 
This is an important step when travelers determine their own routes. In the setting studied here, however, the salient decentralized decision on the demand side is mode choice, since routing and rebalancing are centrally controlled by the \gls{acr:amod} operator.
\gzdel{This is an important step when travelers determine their own routes and route interactions are the main source of decentralization. 
In the setting studied here, however, the salient decentralized decision on the demand side is mode choice. Conditional on selecting \gls{acr:amod}, routing and rebalancing are centrally controlled by the operator. 
For this reason, route-equilibrium models are less well aligned with the operator-level decisions that matter in our regulation problem.
}

A third class, which is closest to our target-oracle model, represents passenger response through explicit choice models while keeping routing or service design under operator control. 
\cite{liu2019framework} and \cite{pinto2020joint} integrate mode choice into on-demand system design, typically through simulation-based or problem-specific heuristics that return feasible designs without bounds on their optimality gap. 
\cite{bertsimas2020joint} show how to jointly optimize frequency and pricing on multimodal transit networks at scale. 
\cite{banerjee2025plan} develop fast algorithms for multimodal transit operations with fixed-line service and demand-responsive vehicles. 
\cite{guo2024design} formulate a transit-centric \gls{acr:pt}/\gls{acr:amod} design problem with explicit discrete choice. 
Among these papers, \cite{guo2024design} is especially close in spirit to our setting, but road congestion is exogenous in their model, whereas our target module endogenizes the congestion created by \gls{acr:amod} fleet operations.

Taken together, centralized design papers provide the building blocks for the municipal target model, but they stop at identifying a desirable coordinated operating point. They do not explain how a municipality can induce that point when the \gls{acr:pt} and \gls{acr:amod} operators are self-interested, and in our paper the coordinated optimum is therefore the target to be implemented.
\gzdel{Taken together, centralized design papers provide the natural building blocks for the municipal target model developed later. 
However, they do not by themselves answer the problem studied here, because they stop at identifying a desirable coordinated operating point. 
They do not explain how a municipality can induce that point when the \gls{acr:pt} and \gls{acr:amod} operators are self-interested and strategically independent. 
In our paper, the coordinated optimum is therefore the target to be implemented.
}

\subsection{Strategic Interaction and Regulation in Multimodal Mobility Systems}\label{sec:review_joint}
The second stream drops the aligned-objective assumption and models \gls{acr:pt}, \gls{acr:amod}, and sometimes the municipality as distinct strategic agents. 
This stream is the closest comparator for our paper because it studies the same underlying policy problem: how a public agency should regulate self-interested mobility providers operating on a shared urban system.

\cite{mo2021competition} are among the earliest to model competition between shared autonomous vehicles and public transit as a game, using simulation-based \gls{acr:ibr} to approximate equilibrium outcomes. 
\cite{lanzetti2023interplay} move toward a more structural treatment by characterizing operator responses through optimization models, which makes the interaction mechanism more transparent and more portable across settings. 
Building on this line, \cite{zardini2021game, zardini2023strategic} introduce a municipal regulatory layer and analyze how taxes and price restrictions affect the equilibrium of multimodal transportation systems. 
\cite{dandl2021regulating} and \cite{he2025hierarchical} also place a regulator above lower-level mobility interactions and search over taxes, subsidies, or operational constraints through Bayesian or feedback-based optimization.

This literature makes an essential conceptual move: it acknowledges that realistic multimodal markets are decentralized and that municipal and operator objectives are misaligned~\citep{tirachini2020ride,hall2018uber}. 
However, this stream keeps equilibrium as the solution concept. The regulator chooses a policy, anticipates the induced operator equilibrium, and evaluates its outcome. This choice creates three difficulties, which we formalize in \cref{sec:pre}. First, it relies on a strong hypothesis on how operators reason and coordinate. Second, the resulting bilevel or tri-level models are typically solved by discretizing the municipal action space, simulation-based search, Bayesian optimization, or feedback heuristics, without optimality guarantees~\citep{zardini2021game,dandl2021regulating,he2025hierarchical}. Third, when multiple equilibria arise, the evaluation depends on equilibrium selection rather than on the policy alone~\citep{milchtaich2005topological}.
\gzdel{However, the dominant methodological choice in this stream is to keep equilibrium as the solution concept. 
The regulator chooses a policy from an admissible action space, anticipates the equilibrium response of the operators, and then evaluates the induced equilibrium outcome. This creates three difficulties for the present problem.

First, the analysis depends on a strong behavioral hypothesis, namely that operators correctly reason through the strategic environment and settle on the equilibrium relevant for the regulator. 
Second, the resulting model is typically bilevel or tri-level and computationally difficult, especially once passenger choice and network congestion are embedded in the lower levels. 
In practice, this has led to discretization of the municipal action space, simulation-based search, Bayesian optimization, or feedback heuristics rather than constructive policy derivations with explicit optimality guarantees~\citep{zardini2021game,dandl2021regulating,he2025hierarchical}.
Third, equilibrium multiplicity can arise in networked games, making the regulator's evaluation sensitive to equilibrium selection rather than solely to the policy itself~\citep{milchtaich2005topological}.
}

Our paper changes the regulatory question asked in this literature. Instead of asking which municipal action optimizes the equilibrium outcome, we ask for the minimum transfer rule that implements a specified, socially preferred operator profile. The regulator then needs one target problem and two unilateral deviation problems instead of an outer policy search over an equilibrium correspondence.
\gzdel{Our paper departs from this literature in a substantive rather than merely algorithmic way. 
We do not improve equilibrium computation for a municipality--operator game. 
Instead, we change the regulatory question from ``which municipal action optimizes the equilibrium outcome?'' to ``what is the minimum transfer rule that implements a specified socially preferred operator profile?'' Once this question is adopted, the computational objects also change: the regulator no longer needs an outer search over municipal policies and an equilibrium correspondence, but instead one target problem and two unilateral deviation problems.
}

\subsection{Implementation and Incentive Design}
Beyond transportation-specific models, the paper is also connected to the implementation and incentive-design literature. 
\cite{monderer2003k} introduce $k$-implementation, in which a reliable interested party cannot redesign the game or prohibit actions, but can commit to nonnegative transfers contingent on the realized strategy profile. 
The objective is to implement a desired outcome while minimizing the realized payment. Unlike the equilibrium-based regulatory papers above, the central object is the minimum transfer that makes a target outcome robust to unilateral deviation when agents do not play dominated strategies.\gzdel{This perspective is conceptually different from the equilibrium-based regulatory papers above: the central object is not an equilibrium correspondence induced by a policy search, but the minimum transfer needed to make a target outcome robust to unilateral deviation under the assumption that agents do not play dominated strategies.}

$k$-implementation is originally formulated for abstract games. To the best of our knowledge, this perspective has not been operationalized for multimodal \gls{acr:amod}/\gls{acr:pt} regulation. 
Our paper does not contribute new implementation theory, but makes a classical implementation result computable in transportation network models.\gzdel{Instead, it shows how a classical implementation result can be embedded in transportation-specific network models with explicit passenger choice, endogenous congestion, and operator-specific service design.}

\gzdel{This is precisely why the modular architecture introduced later is important: the target oracle is rooted in the centralized design literature, the two operator-deviation oracles connect to the decentralized strategic literature, and the payment formula comes from $k$-implementation.
}

\paragraph{Positioning of this paper.}
Our approach differs from prior work in three ways. First, relative to centralized multimodal design, we retain explicit mode choice and endogenous congestion but treat the coordinated optimum as a target rather than as the realized market outcome. Second, relative to equilibrium-based regulation, we model the same conflict between municipality and operators but replace equilibrium search with target implementation under a weaker behavioral requirement. Third, relative to the implementation literature, we make the payment computable on a transportation network by developing formulations with exactness conditions or certified bounds for the target and for both best responses.
\section{Preliminaries} \label{sec:pre}
\subsection{System Setup} \label{sec:setup}
We consider a multimodal transportation network served by an \gls{acr:amod} operator and a \gls{acr:pt} operator. 
Their service designs shape passenger behavior and, through that behavior, determine the performance of the overall system (i.e., passenger travel time and operation cost).

The \gls{acr:amod} service is operated on a strongly connected directed road~$\agraph = \tup {\anodes, \aedges}$, where~$\anodes$ is the set of road nodes and~$\aedges \subseteq \anodes \times \anodes$ is the set of directed road edges. 
For each edge~$e\in \aedges$, let~$b_e$ denote exogenous background traffic (i.e., nominal traffic conditions without the \gls{acr:amod} fleet).
We consider endogenous congestion triggered by the \gls{acr:amod} fleet by assuming flow-dependent travel time for each edge.
Road travel times depend on total vehicle flow through a nondecreasing function~$T:\mathbb{R}_{\geq 0}\to\mathbb{R}_{\geq 0}$.

The \gls{acr:pt} service is operated on a directed transit network~$\ptgraph=\tup{\ptnodes,\ptedges}$, where~$\ptnodes$ is the set of transit stations and~$\ptedges$ is the set of directed transit links. 
Let~$\ptl=\{l_1,\dots,l_n\}$ denote the set of transit lines. Each line~$l\in\ptl$ is an ordered sequence of stations,~$l=\tup{v^l_1,\dots,v^l_{n_l}},$ such that~$\tup{v^l_k,v^l_{k+1}}\in\ptedges$ for all~$k=1,\dots,n_l-1$.

We make the following assumption on the relation between~$\agraph$ and~$\ptgraph$.

\begin{assumption}[Disjoint driving and transit network]\label{assump:disjoint}
The driving network $\agraph$ and transit network $\ptgraph$ are disjoint, i.e.,~$\anodes \cap \ptnodes = \emptyset$ and $\aedges \cap \ptedges = \emptyset$.
\end{assumption}

\cref{assump:disjoint} excludes direct physical interaction between road traffic and transit operations. 
This assumption is appropriate when transit is operated on dedicated infrastructure, such as metro systems or dedicated bus lanes~\citep{zardini2022analysis}.

The two service layers are connected by a walking network~$\wgraph=\tup{\wnodes,\wedges}$, which is assumed to be strongly connected. 
The walking layer intersects the road and transit layers at~$\mathcal{V}^{wa}:=\wnodes\cap\anodes,$ and~$\mathcal{V}^{wp}:=\wnodes\cap\ptnodes.$
We assume that all demand originates and ends at nodes in~$\wnodes$. 
\gzdel{Thus, a passenger traveling from origin~$o\in\wnodes$ to destination~$d\in\wnodes$ may access either service through the walking layer, traverse the corresponding service network, and then walk from the egress node to the final destination.}

\gzdel{With the above definitions, we provide a definition of the multimodal transportation network that we use throughout the rest of the manuscript.}

\begin{definition}[Multimodal transportation network]
A \emph{multimodal transportation network}~$\mathcal{G}=\tup{\mathcal{V},\mathcal{E}}$ consists of a driving network~$\agraph$, a transit network~$\ptgraph$, and a walking network~$\wgraph$. 
The walking network~$\wgraph$ shares nodes in common with~$\agraph$ and~$\ptgraph$, respectively, making~$\mathcal{G}$ a strongly connected graph routable between any two demand nodes.
\end{definition}

Travel demand is specified over \gls{acr:od} pairs on the walking layer.

\begin{definition}[Demand]
For any ordered pair of distinct nodes~$i,j \in \wnodes$, let~$d_{ij}:=\tup{i,j,\alpha_{ij}}$ denote the \emph{travel demand} from origin~$i$ to destination~$j$, where~$\alpha_{ij} \in \mathbb{R}_{>0}$ is the number of travelers wishing to move from~$i$ to~$j$. 
The total demand is the set~$D := \{\,d_{ij} \mid i,j \in \wnodes,\ i \neq j,\ \alpha_{ij} > 0\,\}$, and~$\od := \{\tup{i,j} \in \mathcal{V}^\mathrm{w} \times \mathcal{V}^\mathrm{w} \mid d_{ij} \in D\}$ is the corresponding set of demand \gls{acr:od} pairs.
\end{definition}
\revised{Here, $\alpha_{ij}$ is the aggregate \gls{acr:od} demand over the planning period. The optimization models developed in later sections are therefore a steady-state planning representation instead of a real-time control formulation.}

\subsection{Problem Definition} \label{sec:problem}
\begin{figure}[tb]
    \centering
    \begin{minipage}[t]{0.44\textwidth}
        \centering
        \includegraphics[width=\linewidth]{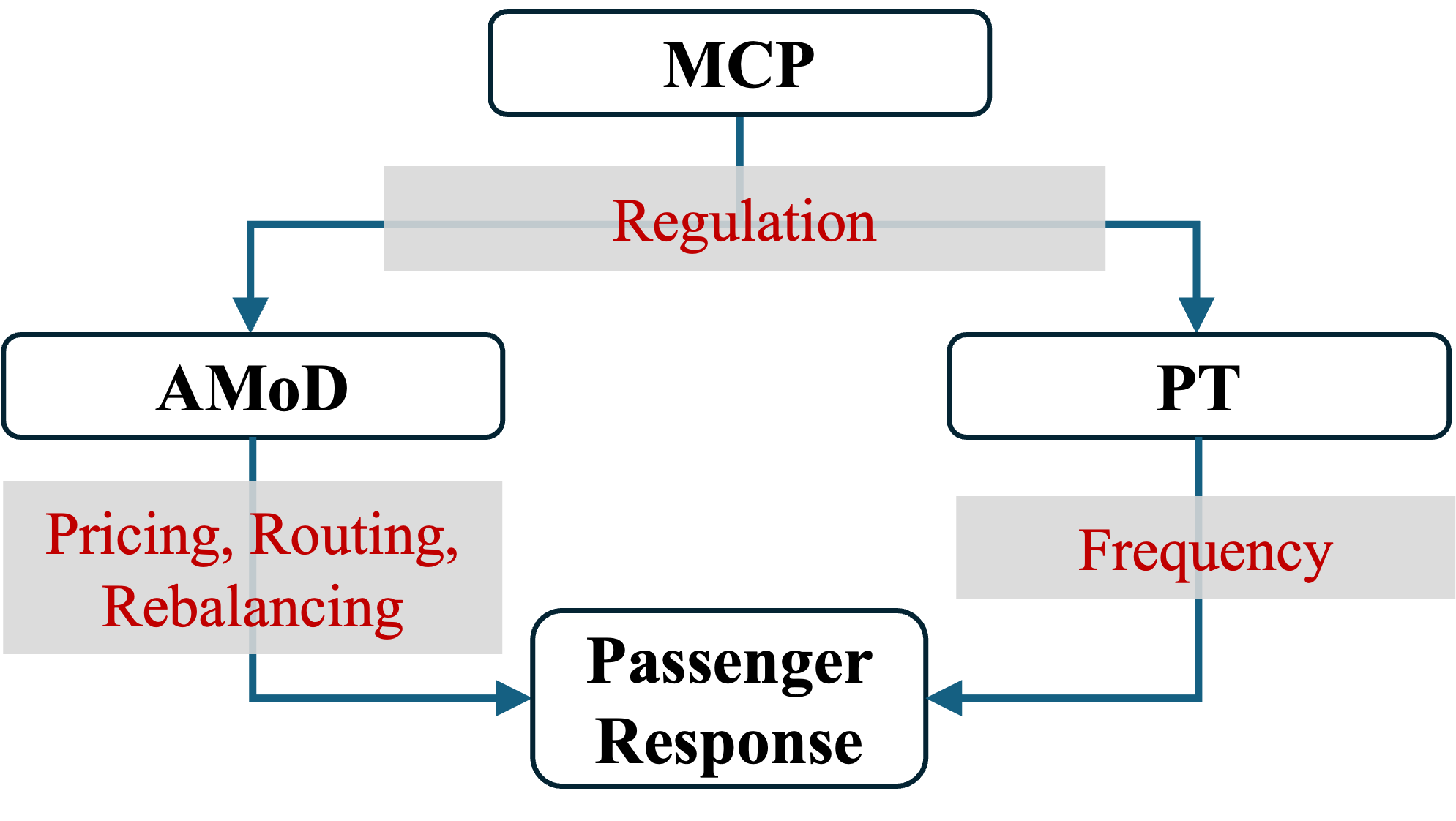}
        \caption{Structure of interaction among \gls{acr:mcp}, operators, and passengers.}
        \label{fig:interaction}
    \end{minipage}\hfill
    \begin{minipage}[t]{0.52\textwidth}
        \centering
        \includegraphics[width=\linewidth]{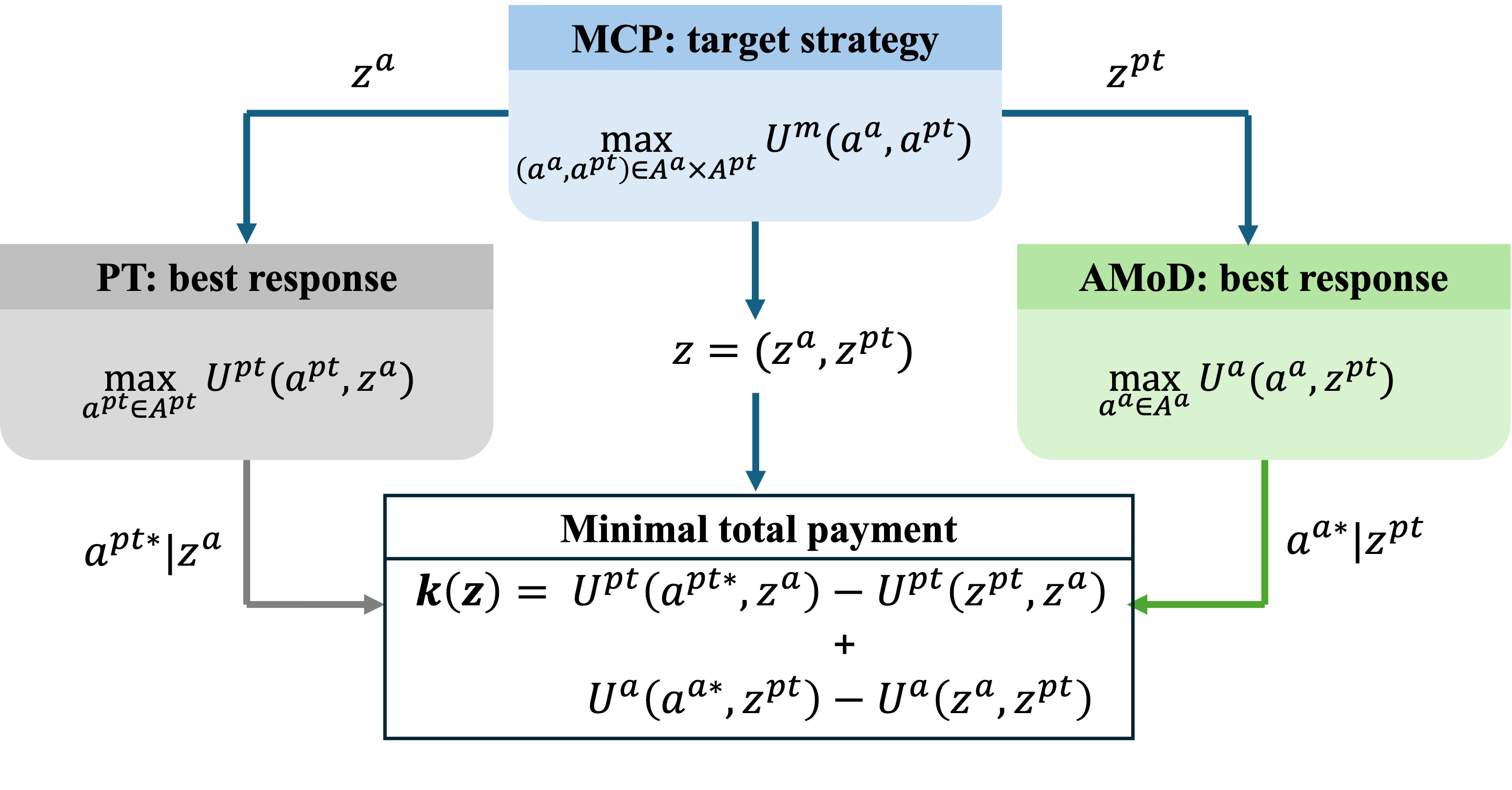}
        \caption{\gzadd{The $k$-implementation-based framework. The target oracle computes a socially preferred operator profile $z=(z^a,z^{pt})$, and the two operator oracles compute the deviation values from which the implementing transfer rule $a^m(z)\in A^m$ is derived.}}
        \label{fig:framework}
    \end{minipage}
\end{figure}

The system is shaped by three interacting entities, summarized in \cref{fig:interaction}. The operators design services, passengers respond to the resulting service attributes, and the \gls{acr:mcp} seeks to improve system performance through regulation. We define the decisions of each entity in turn.\gzdel{The multimodal transportation system is shaped by three interacting entities: the \gls{acr:mcp}, the two operators, and the passengers. 
Their interaction architecture is summarized in \cref{fig:interaction}.
The operators design services, passengers respond to the resulting service attributes, and the \gls{acr:mcp} seeks to improve overall system performance through regulation.
Before formally discussing the problem arising from this interaction, we first define the behavior and decisions of each entity involved.}

\paragraph{\gls{acr:amod}:} 
During operation, the \gls{acr:amod} operator chooses an \gls{acr:od}-specific price~$\pi_{od}$ for each~$\tup{o,d}\in\od$, assigns passengers to candidate road paths~$p\in P_a^{od}$, and determines rebalancing flows on the road network. 
We denote by~$x^a_{od,p}$ the number of passengers assigned to path~$p$, and by~$r_{ij}$ the rebalancing flow on road edge~$\tup{i,j}\in\aedges$. 
The path set~$P_a^{od}$ captures the routing menu available to the operator; for example, it may consist of the top-$k$ shortest paths between~$o$ and~$d$.

% During operation, the \gls{acr:amod} operator designs the price of the service and the route, both of which are \gls{acr:od}-dependent. Specifically, for each $\tup{o,d} \in \od$, the decisions that the \gls{acr:amod} operator needs to make include: 1) the price $\pi_{od}$ for the \gls{acr:amod} trip from $o$ to $d$ and 2) the distribution of the passengers who choose to use \gls{acr:amod} to different routes. This is represented by the number of passengers $x_{od,p}^a$ assigned to each of the possible path $p$ in an set of candidate routes $p \in P_a^{od}$. 
% The set $P_a^{od}$ is derived based on the customized routing preference of the \gls{acr:amod} operator, and it represents the different routing behavior that the \gls{acr:amod} can instruct on the vehicles when traveling between the same \gls{acr:od}. One example is to get the set $P_a^{od}$ as the top $k$ shortest paths from $o$ to $d$. Besides the price and route assignment, another important operation decision of \gls{acr:amod} is the rebalancing. Rebalancing is crucial to ensure efficient steady-state operation without vehicles accumulating at the destinations. The rebalancing decision is represented by the rebalancing flow $r_{ij}$ on each edge $\tup{i,j} \in \aedges$.

\paragraph{\gls{acr:pt}:} 
We focus on operational decisions and therefore treat the topology of the transit lines as fixed. 
The transit fare~$P^{pt}$ is fixed and exogenous during the operating horizon, because fare design is a medium- to long-term decision that is typically constrained by equity, accessibility, and political considerations rather than by short-run profit maximization~\citep{cats2014public,nuworsoo2009analyzing}. The operational decision of the transit operator is therefore the service frequency.
\gzdel{The transit fare~$P^{pt}$ is also taken as fixed and exogenous during the operating horizon. 
This is natural for two reasons. 
First, fare design is typically a medium- to long-term planning decision. 
Second, transit fares are often constrained by equity, accessibility, and political considerations rather than short-run profit maximization~\citep{cats2014public,nuworsoo2009analyzing}. 
The operational decision of the transit operator is therefore the service frequency. 
}
For each line~$l\in\ptl$, the operator chooses a frequency~$s_l$ from a finite set of increasingly ordered admissible levels~$S:=\{\overline{s}_1,\dots,\overline{s}_n\}$.
The frequency affects both passenger waiting time and service capacity. \revised{For each $(o,d) \in \od$, there is a finite candidate transit path set $P_{pt}^{od}$. Each $p \in P_{pt}^{od}$ has fixed in-vehicle, access, and egress time $\tau_p$, boarding or transfer line set $\Gamma_p \subseteq L$, and edge indicators $\eta_e^p$. We make the following assumption on transit route choice.
}

\begin{assumption}[Deterministic route choice]
    \revised{Given frequencies $\mathbf{s}$, transit passengers of $(o,d)$ travel on a path in minimum total travel time. Ties are broken in favor of the optimizing agent.}
\end{assumption}
\revised{In the models of this paper, transit path costs are flow-independent, so the user equilibrium of within-transit route choice is all-or-nothing on the cheapest path. Appendix~\ref{app:routechoice} explains how a stochastic transit path choice model would impact the tractability of the model developed in \Cref{sec:so}.} 

% \gzadd{Appendix~\ref{app:routechoice} explains why a stochastic route-choice model within transit would break the decomposition of \Cref{sec:so}.}

% Since we focus on the interaction during operation, we assume that the network topology of the \gls{acr:pt} lines is already fixed and given. We also assume \gls{acr:pt} charge a fixed unit price $P^{pt}$. Passengers pay $P^{pt}$ every time they use transit, and they do not need to repay when transferring. We assume a fixed and given transit fare for two reasons. First, transit fare design is a long-term planning task, and typically does not change during operation. Second, while the price design for \gls{acr:amod} is often used to improve the profit and utility of the operator, the transit fare design is more closely related to social equity and accessibility, and is subject to strict regulation~\cite{cats2014public,nuworsoo2009analyzing}. Therefore, the price design is not in full control of the \gls{acr:pt} operator for utility maximization during operation time. What the transit operator can design during operation is the frequency $s_l$ for each line $l \in \ptl$, where $s_l$ is picked from a set of possible frequencies $S:=\{\overline{s}_1,\dots,\overline{s}_n\}$. The frequency influences the expected waiting time of the passengers and the capacity of the transit service.

\paragraph{Passengers:} 
Given the service decisions of the two operators, passengers choose between \gls{acr:amod} and \gls{acr:pt}. 
Let~$t^a_{od}$ and~$t^{pt}_{od}$ denote the generalized travel times associated with the two modes for \gls{acr:od} pair $\tup{o,d}$. \revised{Note that the \gls{acr:pt} travel time is the lowest path time among all candidates, i.e., $t^{pt}_{od}=\min_{p \in P_{pt}^{od}}t_p$, where $t_p$ is the total travel time on transit path $p$ and depends on the frequency design.}
Let~$\mathrm{VOT}>0$ denote the value-of-time parameter used to convert time into generalized monetary cost. 
The utilities of the two alternatives are
\begin{equation*}
\begin{aligned}
V^a_{od}&=-\mathrm{VOT}\cdot t^a_{od}-\pi_{od},\\
V^{pt}_{od}&=-\mathrm{VOT}\cdot t^{pt}_{od}-P^{pt}.
\end{aligned}
\end{equation*}

Passenger mode choice follows a binary logit model~\citep{mcfadden1972conditional}:
\[
\theta^i_{od}
=
\frac{\exp(V^i_{od})}{\exp(V^a_{od})+\exp(V^{pt}_{od})},
\qquad
i\in\{\mathrm{a},\mathrm{pt}\}.
\]
\revised{Throughout the paper, we restrict attention to the binary choice between \gls{acr:amod} and \gls{acr:pt}. Appendix \ref{app:choice} discusses which richer choice models the framework accommodates and under what conditions.}

% With the design decisions of the \gls{acr:amod} and \gls{acr:pt} operators, the passengers make mode choice based on the utility of the two alternatives. Denote the travel time by \gls{acr:amod} and \gls{acr:pt} as $t^a_{od}$ and $t^{pt}_{od}$. The utility of each alternative is calculated based on the travel time and monetary cost, expressed in the following format:
% \[V_{od}^{a} = -\alpha\cdot t^a_{od}-\pi_{od};\qquad V_{od}^{pt} = -\alpha\cdot t^{pt}_{od}-P^{pt}\]
% where $\alpha$ is the value of time that transforms time to monetary value. Based on random utility theory, the choice probability of each alternative is then calculated by the \gls{acr:mnl} choice model~\cite{mcfadden1972conditional}:
% \[\theta_{od}^{i}=\frac{e^{V^i_{od}}}{e^{V^a_{od}}+e^{V^{pt}_{od}}}\]
% where $i \in \{\text{a},\text{pt}\}$.

\paragraph{\gls{acr:mcp}:} 
At the highest level, the \gls{acr:mcp} acts as a regulator that seeks to improve the overall efficiency of the multimodal transportation system. 
Social welfare in mobility can be quantified in various ways, such as maximizing consumer and producer surplus~\citep{small2024economics}, minimizing total system travel time~\citep{daganzo1997fundamentals}, or internalizing operational externalities~\citep{parry2007automobile}. 
At this stage, we keep the municipal objective abstract and introduce a concrete social-cost specification in \Cref{sec:so}.

As discussed in \cref{sec:review}, the standard approach to aligning self-interested operators with the \gls{acr:mcp} objective relies on equilibrium-based formulations, which we review next and clarify limitations.

\subsection{\revised{Equilibrium-Based Regulatory Model}}

A natural benchmark is a hierarchical game in which the municipality commits to a regulatory action, the two operators respond strategically, and passengers react to the resulting service attributes.

\begin{definition}[Multimodal transportation game]\label{def:multigame}
Let~$A^{m}$ denote the set of admissible \gls{acr:mcp} intervention rules.
In the setting considered here, an element~$a^m\in A^m$ may be interpreted as a subsidy or transfer scheme. 
Let~$A^a$ and~$A^{pt}$ denote the action spaces of the \gls{acr:amod} and \gls{acr:pt} operators, respectively, and let~$A^p$ denote the space of feasible passenger response profiles. 
Passenger behavior is captured by a reaction function~$f:A^a\times A^{pt}\to A^p,$ which maps operator actions to passenger responses. 
Substituting the induced passenger response into the payoffs yields reduced-form utility functions
\[
U^i:A^m\times A^a\times A^{pt}\to\mathbb{R},
\qquad
i\in\{m,a,pt\}.
\]

Given a \gls{acr:mcp} action~$a^m\in A^m$, the \gls{acr:amod} and \gls{acr:pt} operators choose actions~$a^a\in A^a$ and~$a^{pt}\in A^{pt}$ to maximize their respective utilities.
\end{definition}

Under the equilibrium-based perspective, the municipality chooses a regulatory action while anticipating the equilibrium response of the two operators. \revised{Given an MCP action $a^m\in A^m$, an operator profile $(a^{a*},a^{pt*})$ is a Nash equilibrium if neither operator can improve its utility through unilateral deviation:}

\begin{revisedblock}
\begin{equation}\label{eq:nash}
\begin{aligned}
    U^a(a^m,a^a,a^{pt*}) &\leq U^a(a^m,a^{a*},a^{pt*}) \forall a^a \in A^a, \\
    U^{pt}(a^m,a^{a*},a^{pt}) &\leq U^{pt}(a^m,a^{a*},a^{pt*}) \forall a^{pt} \in A^{pt}. 
\end{aligned}
\end{equation}
Let $g(a^m)$ denote the set of operator equilibria induced by $a^m$. The standard equilibrium-based regulatory model chooses an intervention while anticipating this equilibrium response:
    \begin{equation} \label{eq:equilibrium_regulation}
        \max_{a^m \in A^m} U^m(a^m,a^a,a^{pt}) \quad \text{s.t.}  \quad (a^a,a^{pt})\in g(a^m).
    \end{equation}
\end{revisedblock}

% \begin{definition}[Nash equilibrium of the operator game] \label{def:nash} 
% Fix a \gls{acr:mcp} action~$a^m \in A^m$.
% An action profile~$\tup{a^{a*},a^{pt*}}\in A^a\times A^{pt}$ is a \emph{Nash equilibrium} of the induced operator game if neither operator can improve its utility through unilateral deviation:
% \begin{equation}\label{eq:nash}
% \begin{aligned}
%     U^a(a^m,a^a,a^{pt*}) &\leq U^a(a^m,a^{a*},a^{pt*}) \forall a^a \in A^a, \\
%     U^{pt}(a^m,a^{a*},a^{pt}) &\leq U^{pt}(a^m,a^{a*},a^{pt*}) \forall a^{pt} \in A^{pt}. 
% \end{aligned}
% \end{equation}
% \end{definition}

% Let $g(a^m) := \{(a^{a*}, a^{pt*}) \mid (a^{a*}, a^{pt*}) \text{ satisfies condition } \eqref{eq:nash}\}$ denote the mapping from a \gls{acr:mcp} regulatory action to the set of induced Nash equilibria.
% The benchmark regulatory problem is then the bilevel program below.
% \begin{definition}[Equilibrium-based regulatory problem] \label{def:bilevel}
%     Under the equilibrium-based formulation, the \gls{acr:mcp} seeks a regulatory action that maximizes its utility subject to the lower-level operator game reaching a Nash equilibrium:
%     \begin{equation}
%         \max_{a^m \in A^m} U^m(a^m,a^a,a^{pt}) \quad \text{s.t.}  \quad (a^a,a^{pt})\in g(a^m).
%     \end{equation}
% \end{definition}

As anticipated in \cref{sec:review_joint}, formulation~\eqref{eq:equilibrium_regulation} is intuitive but raises three difficulties for the regulatory problem considered here. First, it requires operators to reason correctly about each other and to coordinate on an exact equilibrium, which is difficult to justify in complex transportation markets~\citep{simon1955behavioral,mahmassani1987boundedly} and to compute~\citep{daskalakis2009complexity}. Second, it is a bilevel program whose equilibrium constraint $g(a^m)$ inherits the nonconvexity of the passenger reaction function~$f$, which precludes standard reformulations through \gls{acr:kkt} conditions or variational inequalities. Prior work therefore discretizes~$A^m$~\citep{zardini2021game} or relies on Bayesian~\citep{dandl2021regulating} or feedback optimization~\citep{he2025hierarchical}, which do not guarantee global optimality or bounds on solution quality. Third, $g(a^m)$ may contain multiple equilibria on which operators need not coordinate~\citep{cooper1990selection}. A rigorous design would then have to evaluate the whole equilibrium set, whereas prior multimodal studies implicitly assume uniqueness. More fundamentally, the \gls{acr:mcp} is not concerned with one arbitrarily selected equilibrium, but seeks an operating point robust to unilateral deviations by self-interested operators.

These limitations motivate a different use of the municipal action space~$A^m$. 
Rather than optimizing over~$A^m$ through the bilevel benchmark in \cref{eq:equilibrium_regulation}, we instead fix a desired operator outcome and ask whether one can construct an intervention~$a^m\in A^m$ that makes unilateral deviation unattractive. 
This is precisely the perspective of $k$-implementation~\citep{monderer2003k}.

% These limitations motivate an implementation-based perspective to model the regulation design problem. Rather than solving for an equilibrium of the decentralized operator game, we ask what regulatory transfers are needed to make a target operating point robust to unilateral deviations by self-interested operators. To formalize this idea, we use k-implementation theory~\cite{monderer2003k}. 
%The next section introduces the key concepts and main results from the $k$-implementation theory, and explains why this framework is well-suited to the present multimodal transportation setting.

\subsection{\revised{Implementation Principle and Computational Requirements}}
To circumvent the limitations of equilibrium search, our framework replaces the equilibrium computation with a weaker behavioral requirement and a more direct design question: given a desired target outcome, what incentives must an interested party provide so that each agent finds deviation unattractive?

Consider a strategic-form game among a set of agents~$N = \{1,2,\dots,n\}$ and pure-strategy space $X = X_1 \times X_2 \times \dots \times X_n$, where~$X_i$ is the set of pure strategies of agent $i$. 
Each agent has a utility function~$U_i: X \rightarrow \mathbb{R}$. 
Denote a game~$G$ with payoff functions~$U = \tup{U_1, \dots, U_n}$ as $G(U)$. 
For each agent~$i$, let~$X_{-i}$ denote the strategy profiles of all other agents. 
The agent's strategies are comparable through the following dominance relationship.
\begin{definition}[Dominance]
In a game~$G(U)$, a strategy~$x_i \in X_i$ \emph{dominates} another strategy~$x_i'\in X_i$ for agent~$i$ if 
\[
U_i(x_i,x_{-i})\geq U_i(x_i',x_{-i})
\qquad
\forall x_{-i}\in X_{-i},
\]
and the inequality is strict for at least one~$x_{-i}\in X_{-i}$. 
A strategy is dominant if it dominates every other strategy in~$X_i$.
\end{definition}

In $k$-implementation, an interested party cannot redesign the game, forbid actions, or directly enforce behavior, but it can commit to transfers contingent on the realized strategy profile. 
These transfers are represented by a vector~$V=\tup{V_1,\dots,V_n}$, where each~$V_i:X\to\mathbb{R}_{\geq 0}$ is the payment promised to player~$i$. 
The transfers modify the game from~$G(U)$ to~$G(U+V)$. 
The framework relies on two assumptions. 
First, the interested party is credible, so the agents trust that promised payments will be realized. 
Second, the behavioral assumption on agents is minimal: agents do not play dominated strategies. 
Under this perspective, a target outcome is implemented if, in the modified game $G(U+V)$, each rational agent must select the strategy in the target outcome. \revised{For the singleton implementation considered in this paper, we construct transfers that make the target strategy dominant for every agent. This is the construction underlying \Cref{theorem:k} and it ensures that the target is the unique profile of undominated strategies.
}

\revised{Both assumptions have concrete interpretations in this paper's setting. Because the transfers make the target action dominant, each operator's best choice does not depend on what the other operator does. 
An operator therefore needs neither beliefs about the other operator nor coordination on an equilibrium, only its own objective and the contracted transfer schedule. 
Optimizing against these two objectives is what a profit-maximizing \gls{acr:amod} operator or a transit agency with a stated objective already does when it plans its own operations. 
Credibility requires the \gls{acr:mcp} to commit to the payments and to verify the contracted quantities, which \Cref{sec:payment_assemly} identifies for the present model. 
Municipalities already make such commitments in performance-based transit service contracts, in which payments to a contracted operator are conditioned on realized service~\citep{hensher2003performance,hensher2005competitive,vigren2020impact}. Because transfers are contingent on the realized profile, they are paid ex post, after the contracted quantities have been verified, and because they are nonnegative, no penalties need to be enforced.
An operator that does not comply simply forgoes the transfer.
}
% Finally, the framework does not remove the need to know the operators' objectives: the \gls{acr:mcp} must know, or estimate sufficiently well, each operator's utility in order to evaluate its best response.

Let $z=(z_1,\dots,z_n)\in X$ denote a desired target profile. The objective of the interested party is to design transfers that implement $z$ while minimizing the realized payment at the target. 
The following result, adapted from \cite{monderer2003k}, provides the characterization used in this paper.

\begin{theorem}[Singleton $k$-implementation~\cite{monderer2003k}]\label{theorem:k}
Let $G(U)$ be a finite game with at least two strategies for every player.
Then every strategy profile $z$ admits an optimal implementation, and the minimum realized payment required to implement $z$ is
    \begin{equation}\label{eq:kz}
        k(z) = \sum_{i=1}^n \max_{x_i \in X_i}\Big(U_i(x_i,z_{-i})-U_i({z_i,z_{-i}})\Big).
    \end{equation}
For each agent $i$, the corresponding payment committed at the target is
\[
\max_{x_i\in X_i}\Big(U_i(x_i,z_{-i})-U_i(z_i,z_{-i})\Big).
\]
\end{theorem}
\begin{revisedblock}
\Cref{theorem:k} is stated for finite games, i.e., games with finite strategy sets. Finiteness makes payoff boundedness automatic and\gzadd{} guarantees that the maxima in~\eqref{eq:kz}, as well as the off-target payments used in the construction, are finite and attained.
For arbitrary infinite games, these properties need not hold. 
In the transportation setting considered here, operator actions include continuous variables, such as prices and flows, and prices are not exogenously bounded.
The next result shows that the payment characterization continues to hold whenever every unilateral best-response value is finite and attained, and we then verify this condition for both operators.
%We establish below that the implementation formula used in \Cref{theorem:k} continues to hold in this paper's setting because each unilateral deviation problem has a finite and attained maximum.

\begin{proposition}[Implementation with continuous strategies]\label{prop:continuous}
Consider a game $G(U)$ with at least two players and at least two strategies for every player. Let $z$ be a target profile. Suppose that, for every player $i$ and every $x_{-i} \in X_{-i}$, the unilateral best-response value
\[\hat{U_i}(x_{-i}):=\max_{x_i\in X_i}U_i(x_i,x_{-i})\]
exists and is finite. Then $z$ can be implemented by nonnegative transfers that make $z_i$ a dominant strategy for every player, and the minimum realized payment is
\[k(z) = \sum_{i=1}^n \max_{x_i \in X_i}\Big(U_i(x_i,z_{-i})-U_i({z_i,z_{-i}})\Big).\]
\end{proposition}
\begin{proof}{Proof}
The proof follows the construction of \Cref{theorem:k} and is given in Appendix~\ref{app:imp}.
\end{proof}
For the \gls{acr:pt} operator and any fixed \gls{acr:amod} action, the maximum in \Cref{prop:continuous} exists because service frequencies are selected from the finite set $S$, the candidate path sets are finite, and total ridership is bounded by finite \gls{acr:od} demand. Hence, the \gls{acr:pt} payoff is finite when the best-response problem attains its maximum. For the \gls{acr:amod} operator and any fixed \gls{acr:pt} action, under the logit choice, demand converges exponentially to zero as price increases, so $\pi_{od}x^a_{od}(\pi_{od}) \rightarrow 0$ as $\pi_{od}\rightarrow +\infty$. Conversely, as $\pi_{od}\rightarrow - \infty$ \gls{acr:amod} demand approaches $\alpha_{od}$ and revenue diverges to $-\infty$. Together with finite network demand and operating costs, the \gls{acr:amod} profit-maximization problem therefore admits a finite objective and price. Therefore, in the transportation setting of this paper, \Cref{prop:continuous} applies and the implementation construction of \Cref{theorem:k} can be used.
\end{revisedblock}

\begin{revisedblock}
\Cref{prop:continuous} extends \Cref{theorem:k} to continuous strategies and reduces the regulatory problem to three objects: (i) a target profile $z$, (ii) the \gls{acr:amod} conditional best-response value $\max_{a^a} U^a(a^a,z^{pt})$, and (iii) the \gls{acr:pt} conditional best-response value $\max_{a^{pt}} U^{pt}(a^{pt},z^a)$. The abstract implementation result takes the target as prescribed and characterizes the payment in terms of these best-response values, but it does not provide methods for computing them. In a multimodal network, each object is a large-scale optimization problem in which operator decisions, passenger mode choice, and network conditions
are jointly endogenous, and their natural formulations contain distinct nonconvexities. The target problem couples a logit mode split to network flows and frequency binaries, the \gls{acr:amod} best response jointly optimizes price and routes with congestion-dependent endogenous travel time, and the \gls{acr:pt} best response embeds the logit equality in a mixed-integer frequency design. Computing these three objects to a certified accuracy therefore requires tailored formulation and algorithm design, which are the focus of \Cref{sec:Method}. \revised{Certification matters for the payment itself, as a payment based on an underestimate of $\widehat U^i$ leaves a profitable deviation uncompensated.} The \gls{acr:mcp} payments then follow from \Cref{prop:continuous} directly, without the outer search over $A^m$ in \cref{eq:equilibrium_regulation}.
\end{revisedblock}

We focus on pure strategies because prices, service frequencies, routing instructions, and rebalancing plans are announced and implemented deterministically over the planning horizon, which is consistent with the transportation service-design literature~\citep{bertsimas2020joint,banerjee2025plan}.
\section{Methodology}\label{sec:Method}
In this section, we instantiate the implementation perspective of \cref{sec:pre} for the multimodal setting. \Cref{sec:overview} assembles the framework, \Cref{sec:so,sec:amod_br,sec:pt_br} develop the three oracles, and \Cref{sec:payment_assemly} combines them into the payment and identifies the quantities that must be contracted.
\gzdel{In \cref{sec:pre}, we formalized the regulatory design problem and introduced the $k$-implementation perspective. 
We now instantiate that perspective for the multimodal transportation setting. 
The key modeling step is to separate the problem into two parts: First, the municipality identifies the operator profile it wishes to implement. Second, it computes a transfer rule that compensates each operator for its best unilateral deviation from that profile. This separation mirrors \cref{theorem:k} and avoids the outer search over $A^m$ that would arise in the bilevel equilibrium model.
}

\subsection{Overview of the Framework}\label{sec:overview}
\gzmoved{Figure 2 now shares a float with Figure 1 in \cref{sec:problem} (layout only).}

A summary of the proposed framework is presented in \cref{fig:framework}. 
The overall interaction is decomposed into three computationally independent components: a target-strategy oracle for the municipality, an \gls{acr:amod} best-response oracle, and a \gls{acr:pt} best-response oracle.
The target oracle computes a desired operator profile~$z=(z^a,z^{pt}),$ where~$z^a$ denotes the target \gls{acr:amod} action and~$z^{pt}$ denotes the target \gls{acr:pt} action.

Conditional on this target, the two operator oracles compute the unilateral deviation values required by \cref{theorem:k}. Specifically, define
\begin{align}
\widehat U^a(z) &:= \max_{a^a\in A^a} U^a(a^a,z^{pt}), \label{eq:amod_dev_value}\\
\widehat U^{pt}(z) &:= \max_{a^{pt}\in A^{pt}} U^{pt}(a^{pt},z^a). \label{eq:pt_dev_value}
\end{align}
These quantities are the utilities attained by the two operators when each is allowed to deviate unilaterally while the other remains at the target action. 
The associated deviation gains are
\begin{align}
\Delta^a(z) &:= \widehat U^a(z)-U^a(z^a,z^{pt}), \label{eq:delta_a}\\
\Delta^{pt}(z) &:= \widehat U^{pt}(z)-U^{pt}(z^{pt},z^a). \label{eq:delta_pt}
\end{align}
Because the present application has exactly two strategic operators, \cref{theorem:k} yields
\begin{equation}
k(z)=\Delta^a(z)+\Delta^{pt}(z). \label{eq:k_framework}
\end{equation}

\cref{eq:k_framework} has two immediate consequences. 
First, it gives the minimum realized payment required to implement the target profile $z$. 
Second, it induces a constructive municipal action in $A^m$: a target-contingent transfer rule that compensates each operator for its best unilateral deviation gain. 

\gzdel{In other words, the implementation perspective does not eliminate the municipal action; it replaces the outer optimization over $A^m$ with a direct derivation of the smallest implementing action associated with a prescribed target.
}

The implementation logic is fixed, whereas the three component models are replaceable, as discussed below. The quantity $k(z)$ also measures the incentive gap between the target and the operators' privately preferred actions, so the payment serves both as a policy instrument and as a diagnostic of incentive misalignment.
\gzdel{A major advantage of this architecture is modularity. The implementation logic is fixed, but the three component models are replaceable. 
The framework only requires: (i) a target oracle that computes a socially preferred operator profile, and (ii) two deviation oracles that compute unilateral best responses around that profile. 
This makes the methodology adaptable to alternative behavioral models, profit specifications, and network representations whenever those are more appropriate for a particular application.

The quantity $k(z)$ is also informative beyond subsidy design. 
In the present setting, it measures the incentive gap between the socially desired operating point and the operators' privately preferred actions. 
A small value of $k(z)$ indicates that the target is already close to being aligned with private incentives, whereas a large value indicates that substantial regulatory effort is required to implement the target. 
The implementation payment therefore serves both as a policy instrument and as a diagnostic of incentive misalignment in the multimodal transportation system.
}

\begin{revisedblock}
\paragraph{Oracle models versus implementation requirements.}
The explicit optimization models of the three oracles developed below should be distinguished from the information requirements of the implementation framework. The framework does not require the \gls{acr:mcp} or the operators to solve these particular formulations. Abstractly, given a target profile $z$, the implementation calculation only requires the target utilities $U^i(z)$ and an oracle capable of evaluating each conditional best response value
\[
    \widehat U^i(z^{-i})
    =
    \max_{a^i\in A^i} U^i(a^i,z^{-i}).
\]
Such an oracle may be an analytical optimization model as developed here, a validated simulation-based, empirical, or other black-box response model. Independent of its form, the oracle needs to return the best-response value or a valid upper bound on it. An estimate that falls short leaves a profitable deviation uncompensated. In the following sections, we formulate the three oracles explicitly as optimization problems because the corresponding joint \gls{acr:amod}--\gls{acr:pt} design, \gls{acr:amod} pricing-routing, and \gls{acr:pt} frequency-design problems are important transportation optimization problems in their own right, and the explicit formulations allow us to derive exactness results and certified solution bounds.

Likewise, the resolution of the operator action space is an application choice. The detailed flow, routing, rebalancing, pricing, and frequency variables used in the presented oracles provide a concrete network model for computing the target and deviation values and they should not be interpreted as requiring trip-level municipal dispatch. In an operational implementation, $A^a$ and $A^{pt}$ can instead be defined in terms of observable and contractible service commitments, with each operator retaining freedom over its internal decisions subject to those commitments. The corresponding best response oracle must then evaluate the operator's optimal internal response under that contractual representation. More details about the contractible quantities for our setup are discussed in \Cref{sec:payment_assemly}.
\end{revisedblock}

We now instantiate the three components in \cref{fig:framework}.
The target oracle is formulated as a joint multimodal system-design problem for the \revised{implementable} social optimum. \revised{Consistent with the steady-state representation of \Cref{sec:setup}, the target oracle models a planning-level problem which gives service agreement on aggregates. It does not specify trip-level instructions to individual vehicles.}
The two deviation oracles are then formulated as an \gls{acr:amod} profit-maximization problem and a \gls{acr:pt} frequency-design problem, respectively. 
\subsection{Target Strategy: Joint Optimization for the \revised{Implementable Social Target}} 
\label{sec:so}

This subsection instantiates the target-strategy oracle in \cref{fig:framework}. We define the \gls{acr:mcp} utility $U^m(a^a,a^{pt})$ through a social-cost model and exploit its structure to compute a target profile $z=(z^a,z^{pt})$ without carrying the logit constraint in the main optimization. No municipal transfer is computed at this stage. Relative to the coordinated design models of \cref{sec:review}, the model retains operator-specific service decisions, explicit passenger mode choice, and road congestion endogenous to \gls{acr:amod} fleet operations.

\subsubsection{Problem formulation} \label{sec:so_problem}
We use the decision variables discussed in \cref{sec:problem}. The \gls{acr:amod} operator chooses \gls{acr:od}-specific prices~$\pi_{od}$, path flows~$x_{od,p}^a$ on candidate routes~$p\in P_a^{od}$, and steady-state rebalancing flows~$r_{ij}$, and the \gls{acr:pt} operator chooses frequencies~$s_l\in S$ at the fixed fare~$P^{pt}$.
\gzdel{Consider a multimodal transportation system with \gls{acr:od} set~$\od$ and demand volume~$\alpha_{od}$ for each~$(o,d)\in\od$. 
For each~$(o,d)\in\od$, the \gls{acr:amod} operator chooses an \gls{acr:od}-specific price~$\pi_{od}$ and assigns passengers to candidate routes~$p\in P_a^{od}$. 
Let~$x_{od,p}^a$ denote the number of passengers assigned to route~$p$, and let~$r_{ij}$ denote the steady-state rebalancing flow on road edge~$(i,j)\in\aedges$. 
The \gls{acr:pt} operator charges a fixed fare~$P^{pt}$ and chooses a service frequency~$s_l$ for each line~$l\in\ptl$, where~$s_l$ is selected from a discrete set~$S$.
} \revised{To model transit passengers' deterministic route choice, let $y_{od,p}\in \{0,1\}$ denote whether transit passengers of $(o,d)$ use path $p$. $x_{od,p}^{pt}$ is the corresponding transit path flow.}

\paragraph{Road-flow and rebalancing constraints.}
Let~$f_{ij}$ denote the total \gls{acr:amod} vehicle flow on road edge~$(i,j)\in\aedges$. 
This flow consists of passenger-carrying vehicles and rebalancing vehicles:
\begin{equation}\label{constr:flow}
f_{ij}
=
r_{ij}
+
\sum_{(o,d)\in\od}\sum_{p\in P_a^{od}} x_{od,p}^a \delta_{ij}^p,
\qquad
\forall (i,j)\in\aedges,
\end{equation}
where~$\delta_{ij}^p=1$ if edge~$(i,j)$ belongs to path~$p$, and~$\delta_{ij}^p=0$ otherwise.

Let~$e^-(v)$ and~$e^+(v)$ denote the sets of incoming and outgoing road edges incident to node~$v\in\anodes$.
Steady-state rebalancing requires flow conservation at every road node:
\begin{equation} \label{constr:conservation}
    \begin{aligned}
    \sum_{(i,v)\in e^-(v)} r_{iv}&-\sum_{(v,j)\in e^+(v)} r_{vj}
    = \sum_{(v,d)\in\od}\sum_{p\in P_a^{vd}}x_{vd,p}^a\\
    -&
    \sum_{(o,v)\in\od}\sum_{p\in P_a^{ov}}x_{ov,p}^a,
    \qquad \forall v\in\anodes
    \end{aligned}
\end{equation}

\paragraph{Transit capacity constraints.}
Transit capacity is limited on each transit edge:
\begin{equation}
\revised{\sum_{(o,d)\in \od}\sum_{p\in P_{od}^{pt}} \eta^{p}_ex^{pt}_{od,p} \leq C_l \sum_{l \in L}\mu^l_es_l, \quad \forall e \in \ptedges
    \label{constr:capacity}}
\end{equation}
\revised{where $\eta_e^p=1$ if transit edge $e$ belongs to path $p$, $\mu_e^l=1$ if edge $e$ belongs to line $l$, and $C_l$ is the unit vehicle capacity of line $l$.}

\paragraph{Demand split and utilities.}
Let $x_{od}^a$ denote total \gls{acr:amod} demand for OD pair $(o,d)$:
\begin{equation}
x_{od}^a=\sum_{p\in P_a^{od}}x_{od,p}^a,
\qquad \forall (o,d)\in\od.
\label{constr:a_total}
\end{equation}
The corresponding transit demand is
\begin{equation}
x_{od}^{pt}=\alpha_{od}-x_{od}^a,
\qquad \forall (o,d)\in\od.
\label{constr:demand_split}
\end{equation}

The expected \gls{acr:amod} travel time is the route-assignment-weighted average of realized path travel times:
\begin{equation}\label{constr:a_time}
\begin{aligned}
    t_{od}^a=
    \sum_{p\in P_a^{od}}
    \frac{x_{od,p}^a}{x_{od}^a}
    \sum_{(i,j)\in\aedges}T(b_{ij}+f_{ij})\delta_{ij}^p, \\ \forall (o,d)\in\od,
\end{aligned}
\end{equation}
where $b_{ij}$ denotes exogenous background traffic and $T(\cdot)$ is the road travel-time function. We use the standard BPR form $T(q)=t_{ij}^0[1+\alpha(q/c_{ij})^\beta]$, where $t_{ij}^0$ and $c_{ij}$ denote free-flow travel time and capacity on edge $(i,j)$.

The corresponding \gls{acr:amod} utility is
\begin{equation}\label{constr:a_v}
V_{od}^a=-\mathrm{VOT}\cdot t_{od}^a-\pi_{od},
\qquad \forall (o,d)\in\od.
\end{equation}

\revised{One path is chosen from the candidate transit path set.}
\begin{equation}
    \revised{\sum_p y_{od,p}=1, \quad \forall (o,d)\in\od,} \label{eq:one_transit}
\end{equation}
\revised{and the transit flow needs to satisfy}

\begin{equation}
    \revised{x_{od,p}^{pt} \leq \alpha_{od}y_{od,p}, \quad\sum_p x_{od,p}^{pt}=x_{od}^{pt}, \quad \forall (o,d) \in \od.}
\end{equation}

The generalized travel time for transit path consists of fixed walking and in-vehicle time plus expected waiting time. Let $z_{ls}\in\{0,1\}$ indicate whether frequency level $s\in S$ is chosen for line $l\in\ptl$. Then
\begin{equation}
\revised{t_{p}=\tau_p+\frac{1}{2}\sum_{l\in\Gamma_p}\sum_{s\in S}\frac{1}{s}z_{ls},
\qquad \forall (o,d)\in\od, p\in P_{od}^{pt}.}
\label{constr:pt_time}
\end{equation}
\revised{We can also calculate an upper and lower bound for the travel time of each transit path based on the minimum and maximum frequency by}
\begin{equation}
    \revised{\underline{t}_p=\tau_p+\frac{1}{2}\sum_{l \in \Gamma_p}\frac{1}{\overline{s}_n}, \quad \overline {t}_p=\tau_p+\frac{1}{2}\sum_{l \in \Gamma_p}\frac{1}{\overline{s}_1},}
\end{equation}
\revised{which are constants. Write $M_{od,p}:= \overline{t}_p-\min_{p'}\underline{t}_{p'}$. $M_{od,p}$ is the maximum difference in travel time between the chosen transit path and the others in the candidate path set. The transit travel time can then be written as}
\begin{equation}
    \revised{t_{od}^{pt} \leq t_p, \quad t_{od}^{pt} \geq t_p-M_{od,p}(1-y_{od,p}), \quad \forall p \in P_{od}^{pt}.} \label{eq:transit_time}
\end{equation}

\revised{Constraints \eqref{eq:one_transit}-\eqref{eq:transit_time} together represent a transit passenger choosing a minimum-time path within the candidate set under the selected frequency design. The experiments of \Cref{sec:Results} use the special case $|P^{pt}_{od}|=1$. The justification and the candidate-path-set robustness analysis are given at the opening of \Cref{sec:Results} and Appendix \ref{app:multipath}.} 

The corresponding transit utility of taking transit is
\begin{equation}
V_{od}^{pt}=-\mathrm{VOT}\cdot t_{od}^{pt}-P^{pt},
\qquad \forall (o,d)\in\od.
\label{constr:pt_v}
\end{equation}

\paragraph{Frequency selection and passenger choice.}
Transit frequencies are selected from the discrete set $S$.
\begin{align}
\sum_{s\in S} z_{ls} &= 1,
\qquad \forall l\in\ptl,
\label{constr:selection}\\
s_l &= \sum_{s\in S} s\, z_{ls},
\qquad \forall l\in\ptl.
\label{constr:freq}
\end{align}

Passenger mode choice follows a binary logit model:
\begin{equation}
x_{od}^a=\alpha_{od}\frac{e^{V_{od}^a}}{e^{V_{od}^a}+e^{V_{od}^{pt}}},
\qquad \forall (o,d)\in\od.
\label{constr:choice}
\end{equation}

All flow variables are nonnegative:
\begin{equation}
x_{od,p}^a,\revised{x_{od,p}^{pt}},r_{ij}\ge 0,
\quad \forall (o,d)\in\od,\ (i,j)\in\aedges.
\label{constr:nonneg}
\end{equation}

\paragraph{Municipal utility.}
The \gls{acr:mcp}'s target is to minimize total social cost:
\begin{equation}\label{eq:so_obj}
    \begin{aligned}
        \min_{\substack{\boldsymbol{x}^a, \boldsymbol{x}^{pt}, \boldsymbol{f}, \\\boldsymbol{r}, \boldsymbol{\pi}, \boldsymbol{s}, \boldsymbol{z}}} & \quad \mathrm{VOT}_\mathrm{ENV}\sum_{(i,j)\in \aedges}(b_{ij} +f_{ij}) \cdot T(b_{ij} +f_{ij}) \\ 
        & + \mathrm{VOT} \sum_{(o,d) \in \od}x^{pt}_{od}t^{pt}_{od} \\
        & + \sum_{(i,j)\in \aedges}C_{ij}f_{ij} + \sum_{l \in \ptl}K_l s_l
    \end{aligned}
\end{equation}
subject to constraints \eqref{constr:flow}-\eqref{constr:nonneg}.

The first term in \eqref{eq:so_obj} measures road-system delay generated by total road traffic, including background traffic and \gls{acr:amod} fleet operations. The second term measures transit passenger time. The last two terms are \gls{acr:amod} and \gls{acr:pt} operating costs.

\gzmoved{Remark (Transit route choice and tractability) moved verbatim to Appendix~\ref{app:routechoice}.}

Formulation \eqref{constr:flow}-\eqref{eq:so_obj} presents a mixed-integer nonconvex optimization problem. The main sources of nonconvexity are the logit choice constraint \eqref{constr:choice} and the bilinear term $x_{od}^{pt}t_{od}^{pt}$ in the objective. 
We next show how the problem can be decomposed and reformulated into a \gls{acr:micp}.

\subsubsection{Decomposition, reformulation, and regularization}
A useful structural observation about the problem formulation presented in \cref{sec:so_problem} is that the price variables~$\pi_{od}$ affect the feasible mode split through \eqref{constr:choice}, but do not appear directly in the municipal objective \eqref{eq:so_obj}. 
This enables a two-step decomposition.

\paragraph{Step 1: Optimal flow assignment.}
We first remove the price variables and the logit choice constraint, and solve the reduced problem over the physical flow and service-design variables subject to \eqref{constr:flow}--\eqref{constr:demand_split}, \eqref{eq:one_transit}--\eqref{eq:transit_time}, \eqref{constr:selection}--\eqref{constr:freq}, and \eqref{constr:nonneg}.
% We first solve a reduced problem obtained by removing the choice constraint \eqref{constr:choice} and the price variables~$\pi_{od}$. 
% The reduced problem optimizes the socially optimal flow assignment subject to \eqref{constr:flow}--\eqref{constr:demand_split}, \eqref{constr:pt_time}, \eqref{constr:selection}-\eqref{constr:freq}, and \eqref{constr:nonneg}.

\paragraph{Step 2: Price recovery.}
Given an optimal reduced solution, we recover an \gls{acr:amod} price vector that rationalizes the induced mode shares under the binary logit model.
Specifically, after computing~$t_{od}^{a*}$ from \eqref{constr:a_time} and~$t_{od}^{pt*}$ from \eqref{eq:transit_time}, the logit relation implies
\begin{equation}
\begin{aligned}
\pi_{od}^*
=
\mathrm{VOT}\,(t_{od}^{pt*}-t_{od}^{a*})
+P^{pt}
-\ln\frac{x_{od}^{a*}}{x_{od}^{pt*}},\\ \forall (o,d)\in\od,    
\end{aligned}
\label{eq:price_recovery}
\end{equation}
where~$x_{od}^{a*},x_{od}^{pt*}$ are the optimal solutions from Step 1.

\begin{lemma}[Equivalence of Decomposition]\label{prop:dcomp}
An optimal solution of the reduced problem with interior solution $x^{a*}_{od}\in (0,\alpha_{od})$, together with prices recovered via \eqref{eq:price_recovery}, yields an optimal solution of the original formulation \eqref{constr:flow}--\eqref{eq:so_obj}.
\end{lemma}
\begin{proof}{Proof}
See Appendix~\ref{app:lemma1}. \qed
\end{proof}

The reduced target problem remains nonconvex because the objective contains the bilinear term $x_{od}^{pt}t_{od}^{pt}$. 
To linearize this term, we introduce an auxiliary variable 
\begin{equation} \label{eq:so_aux}
\revised{w_{od,p,l,s}=x^{pt}_{od,p}\cdot z_{ls}.}
\end{equation}

The product \eqref{eq:so_aux} can be represented exactly by a big-M linearization. See details about the linearization constraints and proof of exactness in Appendix \ref{app:bigm}.

Using \eqref{constr:pt_time} and \eqref{eq:so_aux}, we rewrite the bilinear transit term as
\begin{revisedblock}
    \begin{equation}\label{eq:so_linear}
    \begin{aligned}
      x^{pt}_{od}t^{pt}_{od} &= \sum_{p}x^{pt}_{od,p}(\tau_{p}+\frac{1}{2}\sum_{l \in \Gamma_p} \frac{1}{s_l}) \\
        &= \sum_{p} x^{pt}_{od,p}\tau_{p} + \frac{1}{2} \sum_{p} \sum_{l \in \Gamma_p} \sum_{s \in S} x^{pt}_{od,p} \cdot \frac{1}{s} \cdot z_{ls} \\
        & = \sum_p x^{pt}_{od,p}\tau_{p}+\frac{1}{2} \sum_p \sum_{l \in \Gamma_p} \sum_{s \in S} \frac{1}{s}w_{od,p,l,s}  
    \end{aligned}
\end{equation}
\end{revisedblock}
Substituting \eqref{eq:so_linear} into \eqref{eq:so_obj} and removing \eqref{constr:pt_time} yields a \gls{acr:micp}, since the remaining nonlinear road-travel term is convex under the BPR function, and all remaining constraints are linear.

\begin{revisedblock}
The reduced problem treats the mode split as a planner-chosen variable, and \Cref{prop:dcomp} shows that any interior split is price-implementable. What the reduced problem does not control is the scale of the price needed. By ~\eqref{eq:price_recovery}, the recovered price is the generalized-cost differential minus the log-odds $\ln \frac{x^a_{od}}{x_{od}^{pt}}$, and nothing in the reduced objective bounds the log-odds. A deterministic system-optimal assignment may place \gls{acr:od} demand at or near an all-or-nothing split, for which the log-odds diverge and no finite price exists, and nearby splits are implementable only through very large price spreads. Such spreads are not cosmetic. They depress the \gls{acr:amod}'s profit at the target and can inflate the gap between the target and the operator's best response. The target oracle should therefore control the log-odds directly, and we do so by adding to the objective the entropy term
\begin{equation}
    \frac{1}{\theta}\sum_{(o,d)\in \od}x_{od}^a\ln x_{od}^a + x_{od}^{pt}\ln x_{od}^{pt}, \quad \theta>0 \label{eq:reg_ent}
\end{equation}
with the convention $0\ln0:=0$. Because the derivative of \gzadd{$x\ln x$} is unbounded at zero, the regularized optimum is interior for every \gls{acr:od} pair, so the interiority hypothesis of \cref{prop:dcomp} is satisfied automatically and \eqref{eq:price_recovery} always returns a finite price. Without \eqref{eq:reg_ent} this is not guaranteed.

We choose this term rather than another penalty on the split because its derivative with respect to the split is exactly the log-odds. It delivers two properties. First, the optimal target split is a logit function of the marginal social costs of the two modes with scale $\theta$ (\Cref{lemma:logit}). The planner selects the target as if travelers responded to marginal social cost in the same way as they respond to generalized cost under the choice model. Second, the recovered price equals the generalized-cost differential minus $\theta$ times the marginal-social-cost difference, so the price spread is directly controlled by $\theta$ (\Cref{coro:price}).
\end{revisedblock}

The regularized target-oracle problem becomes
\begin{revisedblock}
\begin{equation}\label{eq:so_obj_reg}
\begin{aligned}
\min_{\substack{\boldsymbol{x}^a, \boldsymbol{x}^{pt}, \\\boldsymbol{f}, \boldsymbol{r}, \boldsymbol{s}, \boldsymbol{z}}} & \quad \mathrm{VOT}_\mathrm{ENV} \sum_{(i,j)\in \aedges}(b_{ij} +f_{ij}) \cdot T(b_{ij} +f_{ij}) +\\ 
& \mathrm{VOT} \sum_{(o,d)\in \od}(\sum_p x^{pt}_{od,p}\tau_{p}+\frac{1}{2} \sum_{p,l,s} \frac{1}{s}w_{od,p,l,s} ) \\
&+ \sum_{(i,j)\in \aedges}C_{ij}f_{ij} + \sum_{l \in \ptl}K_l s_l \\
&+ \frac{1}{\theta}\sum_{(o,d)\in \od}(x_{od}^a\ln x_{od}^a + x_{od}^{pt}\ln x_{od}^{pt})
\end{aligned}
\end{equation}
\end{revisedblock}
\begin{lemma}[Logit Structure Induced by Entropy Term]\label{lemma:logit}
Fix a feasible binary frequency-selection vector $\boldsymbol{z}$ and transit path selection vector $\boldsymbol{y}$. 
Let~$(\boldsymbol{x}^{a*},\boldsymbol{x}^{pt*},\boldsymbol{f}^*,\boldsymbol{r}^*,\boldsymbol{w}^*)$ solve the resulting continuous convex subproblem of \eqref{eq:so_obj_reg}. 
Then, for every~$(o,d)\in\od$, the optimal mode-share ratio satisfies
\[
\frac{x_{od}^{a*}}{x_{od}^{pt*}}
=
\exp\left(\theta\left(\mathrm{MSC}_{od}^{pt}-\mathrm{MSC}_{od}^{a}\right)\right),
\]
where $\mathrm{MSC}_{od}^{a}$ and $\mathrm{MSC}_{od}^{pt}$ denote the marginal system costs of assigning one additional traveler to \gls{acr:amod} and \gls{acr:pt}, respectively.
\end{lemma}
\begin{proof}{Proof}
See Appendix \ref{app:lemma2}. \qed
\end{proof}

\begin{revisedblock}
\begin{corollary}[Price envelope]\label{coro:price}
    At an optimal solution of \eqref{eq:so_obj_reg} with recovered price by \eqref{eq:price_recovery},
    \begin{equation}
    \begin{aligned}
    \pi_{od}^*
    =
    \mathrm{VOT}\,(t_{od}^{pt*}-t_{od}^{a*})
    +P^{pt}
    -\theta(\mathrm{MSC}_{od}^{pt}-\mathrm{MSC}_{od}^{a}),    
    \end{aligned}
    \end{equation}
    so the deviation of the recovered price from the generalized-cost differential is bounded by $\theta\cdot \max_{od}|\mathrm{MSC}_{od}^{pt}-\mathrm{MSC}_{od}^{a}|$.
\end{corollary}

The parameter $\theta$ thus sets how strongly the \gls{acr:mcp} weights price implementability against system cost, and it indexes a family of implementable targets $z(\theta)$ with varying $\theta$. As $\theta \rightarrow \infty$, the entropy term vanishes and the target approaches the deterministic system optimum. Let $\eta_d^*$ denote the optimal value of the decomposed problem without the entropy term, and let $\eta_r^*(\theta)$ denote the social cost part of the optimal value with it. $\eta_r^*(\theta)-\eta_d^*$ then measures the system cost sacrificed for implementability. Smaller $\theta$ produces more moderate splits, and by \Cref{coro:price} proportionally smaller price spreads. \Cref{sec:Results} reports the implementation payment under varying $\theta$ to show the tradeoff between the achieved welfare and the required payment. 
\end{revisedblock}

In the remainder of the paper, the target profile~$z=(z^a,z^{pt})$ is defined by the solution of the regularized target-oracle problem \eqref{eq:so_obj_reg}, with the corresponding \gls{acr:amod} prices recovered from \eqref{eq:price_recovery}. 
\gzdel{The next two subsections then compute the unilateral deviation values around this target profile.}
\subsection{\gls{acr:amod} Best Response} \label{sec:amod_br}
This subsection instantiates the \gls{acr:amod} deviation oracle in \cref{fig:framework}. We define the \gls{acr:amod} operator's utility under a fixed target transit action, reformulate the resulting nonconvex profit-maximization problem, and derive both certified bounds and a feasible sequential-convexification algorithm.

Fix a target transit action~$a^{pt}$. The \gls{acr:amod} best response~$\max_{a^a \in A^a}U^a(a^a,a^{pt})$ is a profit maximization problem. \gzdel{Fix a target transit action~$a^{pt}$. 
Consider an \gls{acr:amod} operator that maximizes profit. The \gls{acr:amod} best response~$\max_{a^a \in A^a}U^a(a^a,a^{pt})$ is a profit maximization problem under the fixed transit action.}
The decision variables are the \gls{acr:od}-specific prices $\pi_{od}$, passenger route assignments $x_{od,p}^a$, total \gls{acr:amod} demand $x_{od}^a$, and rebalancing flows $r_{ij}$. 
The problem inherits the \gls{acr:amod} flow constraints \eqref{constr:flow}-\eqref{constr:conservation}, \gls{acr:amod} utility constraints \eqref{constr:a_time}-\eqref{constr:a_v}, the choice constraint \eqref{constr:choice}, and variable domain constraint \eqref{constr:nonneg} introduced in \cref{sec:so}, while all transit-side quantities entering passenger utility are treated as fixed data under the target action $a^{pt}$.
In particular, the transit utility $V_{od}^{pt}$ is exogenous in this subsection.

Concretely, the \gls{acr:amod} objective is to maximize operating profit
\begin{equation}\label{eq:amod_obj}
    \max_{\boldsymbol{x}^a, \boldsymbol{f}, \boldsymbol{r}, \boldsymbol{\pi}} ~\sum_{od}\pi_{od}x^a_{od}-\sum_{(i,j)\in \mathcal{E}^\mathrm{a}}C_{ij}f_{ij}.  
\end{equation}

\cite{kim2025strategic} formulate a choice-aware \gls{acr:amod} profit maximization with routing, rebalancing, pricing, and endogenous congestion for a given transit operation, but their formulation retains unresolved nonconvexities and does not certify an optimality gap. Likewise, \gls{acr:sca} has been used to find feasible solutions for choice-based optimization in transportation~\citep{bertsimas2020joint,guo2024design} without bounding their quality. We instead reformulate the problem to isolate the nonconvexity in a single structural term, whose convex relaxation gives an upper bound on the profit that certifies the \gls{acr:sca} solution~\citep{razaviyayn2013unified}.
\gzdel{\cite{kim2025strategic} discuss an \gls{acr:amod} profit maximization under similar setting. By assuming a given transit operation, they formulate a choice-aware profit maximization problem in which the \gls{acr:amod} designs routing, rebalancing, and pricing, with endogenous congestion. However, their final formulation retains unresolved nonconvexities, precluding a certifiable optimality gap. In this section, we reformulate the problem to isolate the nonconvexity into a single, parameter-independent structural component. We then approach this reformulated problem using \gls{acr:sca}~\citep{razaviyayn2013unified}. While \gls{acr:sca} has been used to find feasible solutions for choice-based optimization in transportation~\citep{bertsimas2020joint,guo2024design}, these studies typically focus exclusively on primal feasibility without bounding the solution quality. By establishing a tight convex relaxation of our reformulated problem, we provide a rigorous upper bound to the profit, serving as a certificate \gzadd{of the quality of the \gls{acr:sca} solution}.
}

\subsubsection{Exact Reformulation}
The optimization problem \eqref{eq:amod_obj} is inherently non-convex due to the expected travel-time definition \eqref{constr:a_time}, the logit mode-choice relation \eqref{constr:choice}, and the bilinear revenue term $\pi_{od}x_{od}^a$. \gzdel{We condense these non-convexities into a single parameter-independent term through reformulation.}

\begin{lemma}[Exact Reformulation with Isolated Structural Nonconvexity] \label{prop:amod_reform}
Given a fixed transit utility~$V_{od}^{pt}$, the \gls{acr:amod} profit maximization problem \eqref{eq:amod_obj} is mathematically equivalent to the following minimization problem over the physical network flows:
\begin{equation}\label{obj:amod_re} 
\begin{aligned}
    \min_{\boldsymbol{x}^a, \boldsymbol{f}, \boldsymbol{r}} \quad & \sum_{(o,d)\in \od} \Big(x_{od}^a\ln x_{od}^a-x^a_{od}\ln (\alpha_{od}-x_{od}^a) \\
    & + x_{od}^aV_{od}^{pt}\Big) + \sum_{(i,j)\in \aedges}C_{ij}f_{ij}
    \\
    & + \mathrm{VOT}\sum_{(i,j)\in \mathcal{E}^a } (f_{ij}-r_{ij})T(b_{ij}+f_{ij}) 
\end{aligned}
\end{equation}
by eliminating the nonconvex travel time and choice constraints \eqref{constr:a_time} and \eqref{constr:choice}.
\end{lemma}
\begin{proof}{Proof}
By analytically inverting the model-choice relation, we can express the \gls{acr:amod} price as a function of passenger flows. Substituting this algebraic expression into the profit objective removes the price variables entirely and isolates the structural non-convexity into a single, parameter-independent edge-flow term. The formal proof is provided in Appendix \ref{app:lemma3}. \qed
\end{proof}

\subsubsection{Convex relaxation} \label{sec:amod_relax}
Define $v_{ij}:=f_{ij}-r_{ij}\ge 0$. Since $T(\cdot)$ is nondecreasing and $r_{ij}\ge 0$, we have
\begin{equation}\label{eq:amod_relax}
\begin{aligned}
    \phi_{ij}(f_{ij},r_{ij})&:=(f_{ij}-r_{ij})T(b_{ij}+f_{ij}) \\ &= v_{ij}T(b_{ij}+v_{ij}+r_{ij}) \\
                                    & \geq v_{ij}T(b_{ij}+v_{ij}).   
\end{aligned}
\end{equation}
The right-hand side of \eqref{eq:amod_relax} is convex in $v_{ij}$ under the BPR travel-time function. Replacing each nonconvex term $\phi_{ij}(f_{ij},r_{ij})$ by its convex underestimator $v_{ij}T(b_{ij}+v_{ij})$ yields a convex relaxation of \eqref{obj:amod_re}.

% \begin{remark}[Global Upper Bound on \gls{acr:amod} Profit]\label{lemma:convex_under}
% Replacing $\phi_{ij}(f_{ij},r_{ij})$ with its convex underestimator \eqref{eq:amod_relax} generates a strict convex relaxation of the minimization problem \eqref{obj:amod_re}. Evaluating this relaxed problem yields a global lower bound on the minimized objective, which corresponds to a rigorous global upper bound on the true optimal \gls{acr:amod} profit. We denote this bounding certificate as $\overline{\eta}^{\,a}$. The relaxation is physically tight when the optimal steady-state rebalancing flows $r_{ij}$ approach zero.
% \end{remark}

\begin{remark}[Global Upper Bound on \gls{acr:amod} Profit]\label{lemma:convex_under}
Let $\Phi_{\mathrm{rel}}^{a*}$ denote the optimal objective value of the convex relaxation of \eqref{obj:amod_re}. Then
\[
\overline{\eta}^{\,a}:=-\Phi_{\mathrm{rel}}^{a*}
\]
is a global upper bound on the true optimal \gls{acr:amod} objective. The relaxation is tight when rebalancing flows are small, because \eqref{eq:amod_relax} becomes exact as $r_{ij}\to 0$.
 \end{remark}

% This relaxation is tight when the rebalancing flow is $r_{ij}$ is small. Solve the reformulated problem with the relaxed nonconvex term \eqref{eq:amod_relax} gives the lower bound of the minimization problem \eqref{obj:amod_re}, which is the upper bound to the profit. Denote the optimal profit, which is the negative to the objective of the minimization problem, as $\eta^*_{r}$

\subsubsection{\gls{acr:sca}}
To obtain feasible primal solutions, we isolate the concave residual $-r_{ij}T(b_{ij}+f_{ij})$ of the exact term $\phi_{ij}$ and apply \gls{acr:sca}. At each iteration, the residual is replaced by its first-order Taylor expansion around the current iterate $(r^{(k)},f^{(k)})$, and the resulting convex subproblem is solved within trust-region bounds on $f$ and $r$. Every iterate is feasible for the original problem, so the exact objective evaluated at any iterate is a valid lower bound on the best-response profit. The linearization and the pseudocode of the complete algorithm are given in Appendix~\ref{app:sca}.
\gzmoved{the Taylor expansion and trust-region bounds were moved to Appendix~\ref{app:sca}.}

\begin{remark}[Optimality Gap Certificate]
The relaxation provides a global upper bound $\overline{\eta}^{\,a}$ on the \gls{acr:amod} deviation value, while the \gls{acr:sca} incumbent provides a feasible lower bound $\underline{\eta}^{\,a}$. Thus, the relative suboptimality $\varepsilon$ of the computed \gls{acr:amod} best response is bounded by
\begin{equation}\label{eq:amod_gap}
    \varepsilon \leq \frac{\overline{\eta}^{\,a}-\underline{\eta}^{\,a}}{\overline{\eta}^{\,a}}.
\end{equation}
This explicit gap eliminates the heuristic ambiguity typical of choice-based transportation optimizations,
providing a certifiable guarantee of solution quality for the non-convex operator deviation.
\end{remark}
\subsection{\gls{acr:pt} Best Response} \label{sec:pt_br}
This subsection instantiates the transit deviation oracle in \cref{fig:framework}, that is, the transit operator's best response $\max_{a^{pt}\in A^{pt}} U^{pt}(a^{pt}, a^a)$ to a fixed target \gls{acr:amod} action.\gzdel{We model the transit operator's utility under a fixed target \gls{acr:amod} action. We then solve $\max_{a^{pt}\in A^{pt}} U^{pt}(a^{pt}, a^a)$ by deriving a mixed-integer convex relaxation and establishing an exactness condition that determines when the relaxation recovers the original logit-equality model.}

Transit service design encompasses a variety of modeling objectives, including maximizing the total demand served~\citep{bertsimas2021data}, minimizing total passenger wait time~\citep{bertsimas2020joint}, maximizing profit~\citep{luo2021optimal}, and maximizing system welfare~\citep{banerjee2025plan}. In the transit best response, we adopt the profit-maximization perspective augmented by a social incentive parameter. We define the objective of transit operator as
\begin{equation} \label{eq:obj_pt}
    \max_{\boldsymbol{x}^{pt}, \boldsymbol{s}, \boldsymbol{y}} \sum_{od}P^{pt}\cdot x_{od}^{pt}-\sum_{l \in L}K_ls_l+\sum_{od}\omega x_{od}^{pt},
\end{equation}
\revised{where the \gls{acr:pt} chooses frequencies $\boldsymbol{s}$, and the route choice $y_{od,p}$ and ridership $x_{od}^{pt}$ are endogenous.} The first two terms calculate the profit, and $\omega$ represents the perceived social benefit per passenger. It can be seen as the social responsibility of the transit operator to accommodate a larger number of passengers on the transit network. As $\omega \to 0$, the operator reduces to a strict profit maximizer.

Fix a target \gls{acr:amod} action $a^a$, the transit best response solves \cref{eq:obj_pt} with the capacity constraint \eqref{constr:capacity}, \revised{deterministic transit route choice constraints \eqref{eq:one_transit}-\eqref{eq:transit_time},} the \gls{acr:pt} utility constraints \eqref{constr:pt_v}, the frequency selection constraints \eqref{constr:selection}-\eqref{constr:freq}, and the choice constraint \eqref{constr:choice}. The nonconvexity of the transit best-response problem arises from the waiting-time term in transit utility and the mode-choice relation. In this section, we propose a \gls{acr:micp} relaxation to this problem and characterize when it is exact. \Cref{sec:res_pt} later shows that, at the baseline calibration, the relaxation is exact or within about 3\% of optimality.\gzdel{Crucially, we demonstrate that under standard urban operational regimes this relaxation binds tightly, yielding an exact \gls{acr:micp} reformulation.}

\subsubsection{Convex reformulation}
\begin{revisedblock}
The transit travel time of each candidate path $p$ is given by \eqref{constr:pt_time}\gzdel{\begin{equation*}
t_{p}=\tau_p+\frac{1}{2}\sum_{l\in\Gamma_p}\sum_{s\in S}\frac{1}{s}z_{ls},
\qquad \forall (o,d)\in\od, p\in P_{od}^{pt}, 
\end{equation*}}
and the route choice logic is retained by linear constraints  \eqref{eq:one_transit}-\eqref{eq:transit_time}. With the \gls{acr:amod} action fixed and $V^a_{od}$ constant, the binary logit condition for each $(o,d) \in \od$ is written as
\begin{equation}
\ln\frac{x_{od}^{pt}}{\alpha_{od}-x_{od}^{pt}} = -\,\mathrm{VOT}\cdot t_{od}^{pt} - P^{pt} - V_{od}^a.
\label{eq:pt_br_choice_new}
\end{equation}
Multiplying both sides by $x_{od}^{pt}$ produces the bilinear product $x_{od}^{pt}t_{od}^{pt}$, which equals $\sum_{p} x^{pt}_{od,p}\tau_{p} + \frac{1}{2} \sum_{p} \sum_{l \in \Gamma_p} \sum_{s \in S} \frac{1}{s}x^{pt}_{od,p}  z_{ls}$ as in \eqref{eq:so_linear}. Linearizing each product $x_{od,p}^{pt}z_{ls}$ through auxiliary variables $u_{od,p,l,s}=x_{od,p}^{pt} z_{ls}\geq 0$ and the exact big-M envelope of Appendix~\ref{app:bigm} yields an equality constraint free of bilinear products:
\gzdel{Multiplying both sides by $x_{od}^{pt}$ and rearranging yields
\begin{equation*}
    \begin{aligned}
    x_{od}^{pt}\ln(x_{od}^{pt})-x_{od}^{pt}\ln(\alpha_{od}-x_{od}^{pt}) + \mathrm{VOT} \cdot (x_{od}^{pt}t_{od}^{pt})\\ + x_{od}^{pt}(\gzadd{P^{pt}}+V_{od}^a)=0.  
    \end{aligned}
\end{equation*}
The equality constraint (41) contains a bilinear product $x_{od}^{pt}t_{od}^{pt}$. We now linearize this product. Following \eqref{eq:so_linear}, this term is equivalent to
\[\sum_{p} x^{pt}_{od,p}\tau_{p} + \frac{1}{2} \sum_{p} \sum_{l \in \Gamma_p} \sum_{s \in S} \frac{1}{s}x^{pt}_{od,p}  z_{ls}.\]
Recognizing that the product $x_{od,p}^{pt}\cdot z_{ls}$ is identical to the bilinear term encountered in \cref{sec:so}, we can directly apply the same exact Big-M linearization as in \cref{sec:so}. Introduce the auxiliary variables $u_{od,p,l,s}=x_{od,p}^{pt}\cdot z_{ls}\geq 0$. (41) yields a strict equality constraint free of bilinear products:
}

\begin{equation}\label{constr:pt_br_choice_linearized}
    \begin{aligned}
    &x_{od}^{pt}\ln(x_{od}^{pt})-x_{od}^{pt}\ln(\alpha_{od}-x_{od}^{pt}) \\
    +& \mathrm{VOT} \bigg[\sum_p x^{pt}_{od,p}\tau_{p}+\frac{1}{2} \sum_{p,l,s} \frac{1}{s}u_{od,p,l,s}\bigg]\\ 
    +& x_{od}^{pt}(P^{pt}+V_{od}^a) = 0.     
    \end{aligned}
\end{equation}
\end{revisedblock}

Although \eqref{constr:pt_br_choice_linearized} is now linearized, the presence of the non-linear log-odds terms within a strict equality still violates convexity requirements. However, this structure is now amenable to a convex inequality relaxation.
\begin{lemma}[Convex Relaxation of Mode Choice] \label{lemma:pt_relax}
The strict equality constraint \eqref{constr:pt_br_choice_linearized} can be relaxed to the inequality:
\begin{revisedblock}
    \begin{equation}\label{constr:pt_br_choice_convex}
    \begin{aligned}
    &x_{od}^{pt}\ln(x_{od}^{pt})-x_{od}^{pt}\ln(\alpha_{od}-x_{od}^{pt}) \\
    +& \mathrm{VOT} \bigg[\sum_p x^{pt}_{od,p}\tau_{p}+\frac{1}{2} \sum_{p,l,s} \frac{1}{s}u_{od,p,l,s}\bigg]\\ 
    +& x_{od}^{pt}(P^{pt}+V_{od}^a)\leq 0,       
    \end{aligned}
\end{equation}
\end{revisedblock}
which defines a convex region.
\end{lemma}
\begin{proof}{Proof}
See Appendix \ref{app:lemma4}. \qed
\end{proof}

Therefore, by replacing the non-convex mode choice and travel time constraints with the big-M envelope and convex relaxation \eqref{constr:pt_br_choice_convex}, we establish a globally solvable \gls{acr:micp} relaxation. The exactness of is discussed below.

\begin{lemma}[Exactness of the \gls{acr:micp} Relaxation] \label{prop:pt_exactness}
Let \revised{$(\mathbf{x}^{pt*}, \mathbf{s}^*, \mathbf{y}^*)$} be a global optimal solution obtained by solving the relaxed \gls{acr:micp}. For the frequency vector $\mathbf{s}^*$, let $\mathbf{x}^{L}(\mathbf{s}^*)$ denote the unique logit flow defined by the strict equality \eqref{eq:pt_br_choice_new} \revised{with path selection $\mathbf{y}^*$}, and let $\Omega_{cap}(\mathbf{s}^*)$ denote the feasible polytope defined by capacity constraint \eqref{constr:capacity}. The relaxed solution \revised{$(\mathbf{x}^{pt*}, \mathbf{s}^*, \mathbf{y}^*)$} is strictly exact and globally optimal for the original non-convex \gls{acr:pt} best response if and only if $\mathbf{x}^{L}(\mathbf{s}^*) \in \Omega_{cap}(\mathbf{s}^*)$.
\end{lemma}

\begin{proof}{Proof}
The exactness of the relaxation holds because the relaxed objective is increasing in transit ridership and therefore pushes passenger flow up to the logit boundary, unless the capacity constraint binds first. The formal proof is provided in Appendix \ref{app:lemma5}. \qed
\end{proof}

\begin{remark}[Operational Equivalence of the Relaxation] \label{remark:pt_relax_int}
The relaxation of the mode choice constraint from an equality to an inequality ($\le 0$) is \revised{exact whenever the full logit demand induced by the selected frequencies and minimum-time routes is capacity-feasible. Slack arises only when physical capacity truncates the logit demand}. Therefore, the relaxation does not introduce artificial behavioral errors, but rather models the regime shift from an unconstrained choice equilibrium to a capacity-constrained system.
\end{remark}

Demand shedding in \cref{remark:pt_relax_int} also exposes a misalignment between the operator and the \gls{acr:mcp}. The target oracle requires all transit demand to be served through the capacity constraint, whereas a self-interested operator may rationally shed demand when the marginal cost of additional capacity exceeds the marginal value of the passengers served. In that regime, the implementation payment compensates the operator for serving the full flow.
\subsection{Minimal Payment for \revised{the Target Profile}} \label{sec:payment_assemly}
\gzadd{The three oracles combine into the implementation payment as follows. The target oracle (\cref{sec:so}) returns the profile $z=(z^a,z^{pt})$, with $z^a=(\boldsymbol{\pi},\boldsymbol{x}^{a},\boldsymbol{r})$ and $z^{pt}=\boldsymbol{s}$, at which we evaluate the operators' utilities $U^a(z^a,z^{pt})$ and $U^{pt}(z^a,z^{pt})$ through \eqref{eq:amod_obj} and \eqref{eq:obj_pt}. The \gls{acr:amod} oracle (\cref{sec:amod_br}) with transit fixed at $z^{pt}$ and the \gls{acr:pt} oracle (\cref{sec:pt_br}) with \gls{acr:amod} operations fixed at $z^a$ return the best-response values $\widehat U^a(z)$ and $\widehat U^{pt}(z)$, and \eqref{eq:k_framework} gives $k(z)=\Delta^a(z)+\Delta^{pt}(z)$. When an oracle is solved with a certified gap, evaluating $\Delta^i(z)$ at its upper bound yields a payment that is provably sufficient, while evaluating it at the incumbent yields a lower bound, so the certified gaps bracket $k(z)$.}

\begin{revisedblock}
\paragraph{Observability and contractual implementation.}
Exact singleton implementation assumes that the variables defining the contractual target can be observed and verified. This does not require every internal routing or dispatching decision to be contracted with the \gls{acr:mcp}. The social cost specification of \eqref{eq:so_obj} consists of four components, each of which can be contracted through quantities that are observable in practice. (1) Road-system travel time $\sum_{(i,j)\in \aedges}(b_{ij} +f_{ij}) \cdot T(b_{ij} +f_{ij})$ comprises the fleet's own vehicle-hours $\sum_{(i,j)\in \aedges} f_{ij} \cdot T(b_{ij} +f_{ij})$ and the travel time of background traffic $\sum_{(i,j)\in \aedges} b_{ij} \cdot T(b_{ij} +f_{ij})$. Both are determined by the edge-level fleet flows, which are contracted with the \gls{acr:amod} and can be checked through reported vehicle trajectories. (2) Transit passenger time $\sum_{(o,d) \in \od}x^{pt}_{od}t^{pt}_{od}$ is determined by the service frequencies, which can be contracted with \gls{acr:pt} and verified from published schedules, and the transit ridership, identifiable from transit gate data, or \gls{acr:amod} trip records, if transit gate data is not enough to identify \gls{acr:od}. (3) \gls{acr:amod} operation cost $\sum_{(i,j)\in \aedges}C_{ij}f_{ij}$ follows from the contracted fleet flows from (1) with an agreed unit-cost model. (4) \gls{acr:pt} operation cost $\sum_{l \in \ptl}K_l s_l$ follows from the contracted frequencies with an agreed cost model. In summary, the \gls{acr:pt} agreement concerns frequencies and a cost model, and the \gls{acr:amod} agreement concerns edge vehicle flows, trip records, and a cost model. \gls{acr:amod} prices follow from the \gls{acr:od} split through \eqref{eq:price_recovery}, so they need not be contracted separately. 

% \gzadd{If trajectories are unavailable and the contract must rely on aggregate measures, such as fleet vehicle-hours and vehicle-miles, several internal operating profiles become observationally equivalent. The contract then implements a set of profiles that all attain the social cost of $z$, for which the payment computed for $z$ remains sufficient but may exceed the minimum (Appendix~\ref{app:aggregate}).}
% \gzmoved{the paragraph on aggregate-measure contracts was moved verbatim to Appendix~\ref{app:aggregate}.}
We note that the appropriate contractual resolution is application-specific. If the observable quantities uniquely identify the contractual target, as in the case described above, the implementation construction applies directly at that resolution. If they do not, several internal operator strategies are observationally equivalent to the \gls{acr:mcp}, and an outcome-based contract effectively implements a set of operating profiles rather than the particular detailed profile used in the present formulation. This occurs, for example, when trajectories are unavailable for verifying fleet flows and the contract must rely on aggregate measures. Component (1) can then be contracted through fleet-level vehicle-hours together with the realized network travel time, and component (3) through fleet-level vehicle-miles under a distance-based unit cost, while (2) and (4) are unaffected. For operators, operating profiles other than $z$ may satisfy these aggregates. Because the aggregates still determine every component of \eqref{eq:so_obj}, all such profiles attain the social cost of $z$. The payment computed for $z$ is still sufficient to achieve the target social cost, but it may exceed the minimum required to implement the set. Formal treatment of such set-valued implementation is beyond the present singleton model and is identified in \Cref{sec:Conclusion} as an extension.

\end{revisedblock}
\section{Numerical Experiments}\label{sec:Results}
In this section, we evaluate the framework on a Manhattan instance with empirically grounded demand and infrastructure data. Our objective is comparative and diagnostic: we study solution quality, incentive misalignment, and implementation payments.
\gzdel{This section has two purposes. 
First, it validates the three oracle models introduced in \cref{sec:Method}: the target strategy oracle, the \gls{acr:amod} deviation oracle, and the \gls{acr:pt} deviation oracle.
Second, it benchmarks the system-wide efficiency of our proposed framework against uncoordinated operational baselines to quantify the value of regulatory intervention. It then uses the validated oracle outputs to quantify the implementation burden associated with the social target and to interpret the economic sources of the resulting payment. 
% The experiments are designed as a realistic case study rather than as a statistical forecast of all Manhattan travel behavior. 
Our objective is comparative and diagnostic, to study solution quality, incentive misalignment, and implementation payments on a large-scale urban network with empirically grounded demand and infrastructure data.
} We proceed in \revised{five} steps. First, \cref{sec:res_so,sec:res_amod,sec:res_pt} analyze the entropy regularization of the target oracle and the solution quality of the two deviation oracles. Second, \cref{sec:coord} compares the \revised{target profile} with a decentralized best-response state and a static GTFS-schedule state. Third, \cref{sec:payment} studies the implementation payment. \revised{Finally, we vary the regularization parameter across a family of target profiles to examine how target optimality and implementation cost trade off.}

The experiments are based on the NYC Manhattan network with demand derived from TLC trip data~\citep{nyc_tlc_trip_data}, the road network built from \gls{acr:osm}~\citep{OpenStreetMap}, and the transit network from MTA GTFS data~\citep{mta_subway_gtfs}. To analyze the impact of congestion, exogenous background traffic is represented through the ratio $b_{ij}/c_{ij}$, where $c_{ij}$ is the road capacity. Throughout the experiments, the term background traffic level (BL) refers to this ratio. More detailed descriptions of the data processing, spatial aggregation, parameter values, and computation hardware are provided in Appendix \ref{app:case}. 

\revised{Throughout this section, transit is represented by a single reference path per \gls{acr:od} pair, the special case $|P^{pt}_{od}|=1$ of the model in \Cref{sec:so_problem}. We treat this as a computational baseline rather than a structural restriction of the model. The formulations and associated theoretical guarantees in \Cref{sec:Method} are stated for arbitrary finite candidate path sets. With up to five candidate paths per \gls{acr:od}, the target oracle's solve time grows by more than an order of magnitude, and the \gls{acr:pt} oracle reaches its time limit in most multi-path instances with
certified optimality gaps of up to $7\%$, whereas with one path both mixed-integer oracles solve to within $1\%$, which allows the implementation payments of \Cref{sec:payment} to be tightly certified. Although the richer representation does change passenger-level outcomes, the incentive conclusions of \Cref{sec:payment} are nevertheless robust to it. Appendix \ref{app:multipath} reports the full analysis with different sizes of candidate path.}

%To demonstrate the practicality of the proposed framework, we conducted experiments using real-world demand and network data from Manhattan, New York City. 
%This section first details the data sources, network construction, and parameter settings. We then evaluate the computational performance and theoretical properties of the proposed solutions for the three core oracles presented in \cref{sec:Method}. 
%Specifically, 1) for the social optimum oracle (\cref{sec:so}), we illustrate the role of entropy regularization in the joint social-optimum model; 2) for \gls{acr:amod} best response oracle (\cref{sec:amod_br}), we analyze the convergence of the \gls{acr:sca} algorithm for the \gls{acr:amod} best response; 3) for \gls{acr:pt} best response oracle (\cref{sec:pt_br}), we verify the tightness of the \gls{acr:micp} relaxation for the \gls{acr:pt} frequency design. Finally, we synthesize these components to compute the minimum implementation subsidy required by the \gls{acr:mcp}. All experiments were conducted on a HPC cluster with CPU-only nodes (96 cores, ~520 GB RAM per node) with Gurobi 10.0.2. Each solve used 4 CPU threads and 32 GB RAM.

\subsection{Target oracle: regularization and implementability} \label{sec:res_so}
We first study the entropy regularizer of the target oracle, which trades social-cost optimality against the spread of the recovered \gls{acr:amod} prices.
\Cref{fig:so_theta} reports, for three representative congestion levels ($\mathrm{BL}=0.1, 0.5, 1.0$), the optimality loss of the regularized target problem relative to the unregularized decomposed target problem, together with the recovered \gls{acr:amod} unit-distance price. \revised{The underlying values are tabulated in Appendix \ref{app:theta_table}.}
The unit-distance price is computed as the recovered price $\pi_{od}$ divided by the length of the assigned route.

%As discussed in \cref{sec:so}, the entropy regularizer is crucial for recovering mathematically stable and practically implementable \gls{acr:amod} prices $\pi_{od}$. 
%\cref{fig:so_theta} illustrates the impact of the regularization parameter $\theta$ on both the optimality gap (relative to the unregularized problem) and the resulting \gls{acr:amod} unit distance prices, evaluated across three different background traffic levels (BL = 0.1, 0.5, and 1.0). The \gls{acr:amod} unit distance price is calculated by dividing the recovered price $\pi_{od}$ by the length of the assigned route.

Two patterns are immediate. First, the optimality gap decreases monotonically in $\theta$ at all traffic levels, consistent with the convergence of the regularized problem to the unregularized optimum as $\theta \rightarrow \infty$.
\revised{Second, this gain in target optimality is accompanied by increasingly dispersed recovered prices, and for
large $\theta$ a growing share of \gls{acr:od} pairs receives a negative price. A negative recovered price is feasible in the framework and has a clear economic reading: on the \gls{acr:od} pair, the target profile calls for more \gls{acr:amod} ridership than passengers would choose at any non-negative fare, so implementing it requires a passenger-side incentive. Two consequences follow. The operator forgoes fare revenue on those pairs and must be compensated through the \gls{acr:mcp} transfer, so the implementation payment will rise. \Cref{sec:tradeoff} quantifies this effect along the family of targets. In addition, a negative price can be read as part of the transfer being passed through the operator to passengers, so implementation is no longer carried by operator-side transfers alone.}

For the remainder of the experiments, we set $\theta=0.5$. \revised{This is the largest tested value at which every recovered price is non-negative at all three benchmark congestion levels, so that the baseline target is implementable
through operator-side transfers alone with non-negative fares. Note that we do not impose non-negativity as a constraint. Targets with larger $\theta$ remain admissible designs that entail passenger-side incentives on some \gls{acr:od} pairs, and \Cref{sec:tradeoff} reports what they cost to sustain.} We stress that the weight $\theta$ is a policy calibration parameter rather than a universal constant, and \cref{fig:so_theta} exposes the efficiency--implementability tradeoff that it controls rather than a ``correct'' value.

\begin{figure}[tb]
    \centering
    \begin{minipage}[t]{0.48\textwidth}
        \centering
        \includegraphics[width=\linewidth]{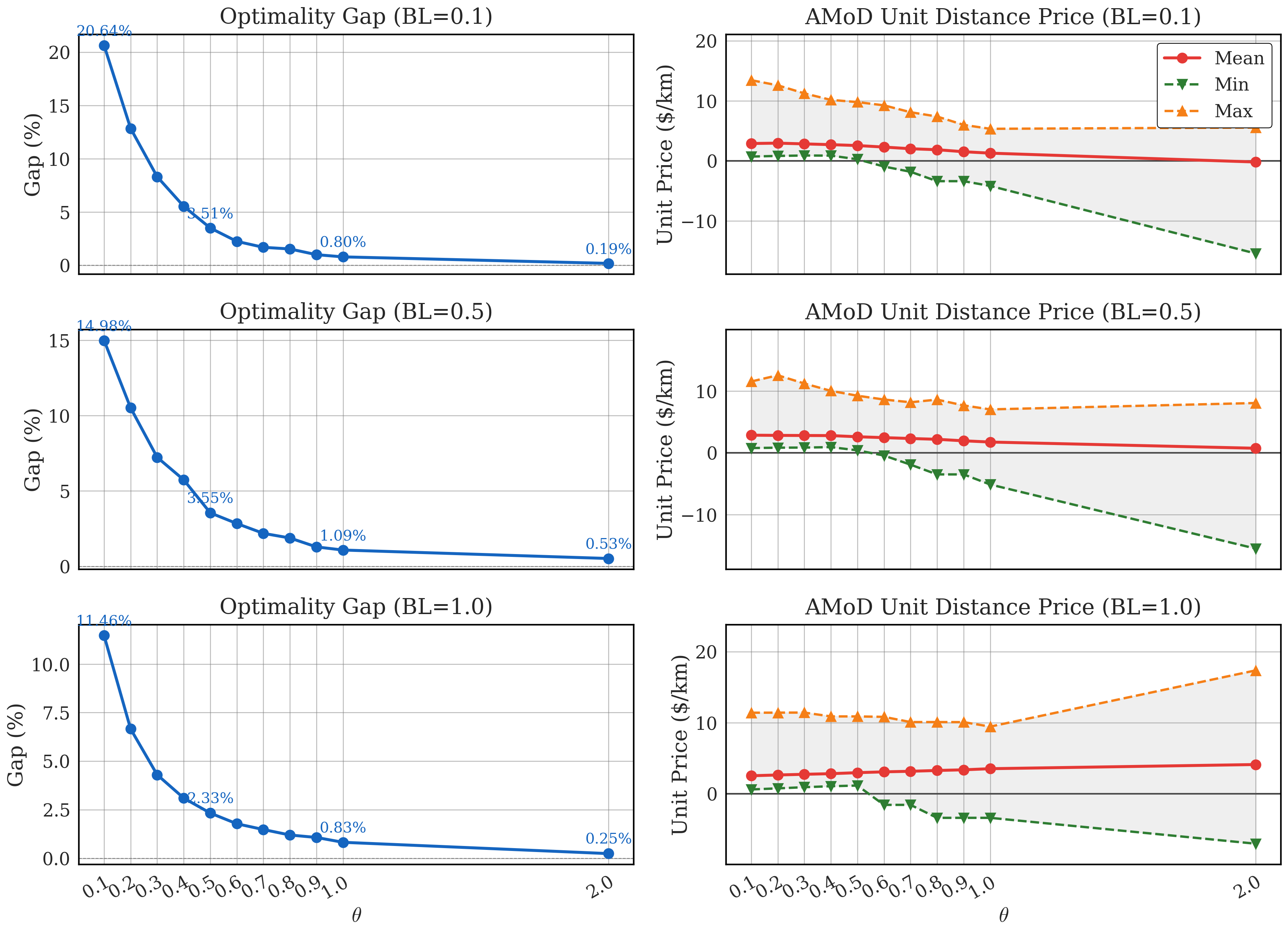}
        \caption{Tradeoff induced by entropy regularization: optimality loss relative to the unregularized decomposed target problem and recovered \gls{acr:amod} unit-distance price under different background traffic levels.}
        \label{fig:so_theta}
    \end{minipage}\hfill
    \begin{minipage}[t]{0.48\textwidth}
        \centering
        \includegraphics[width=\linewidth]{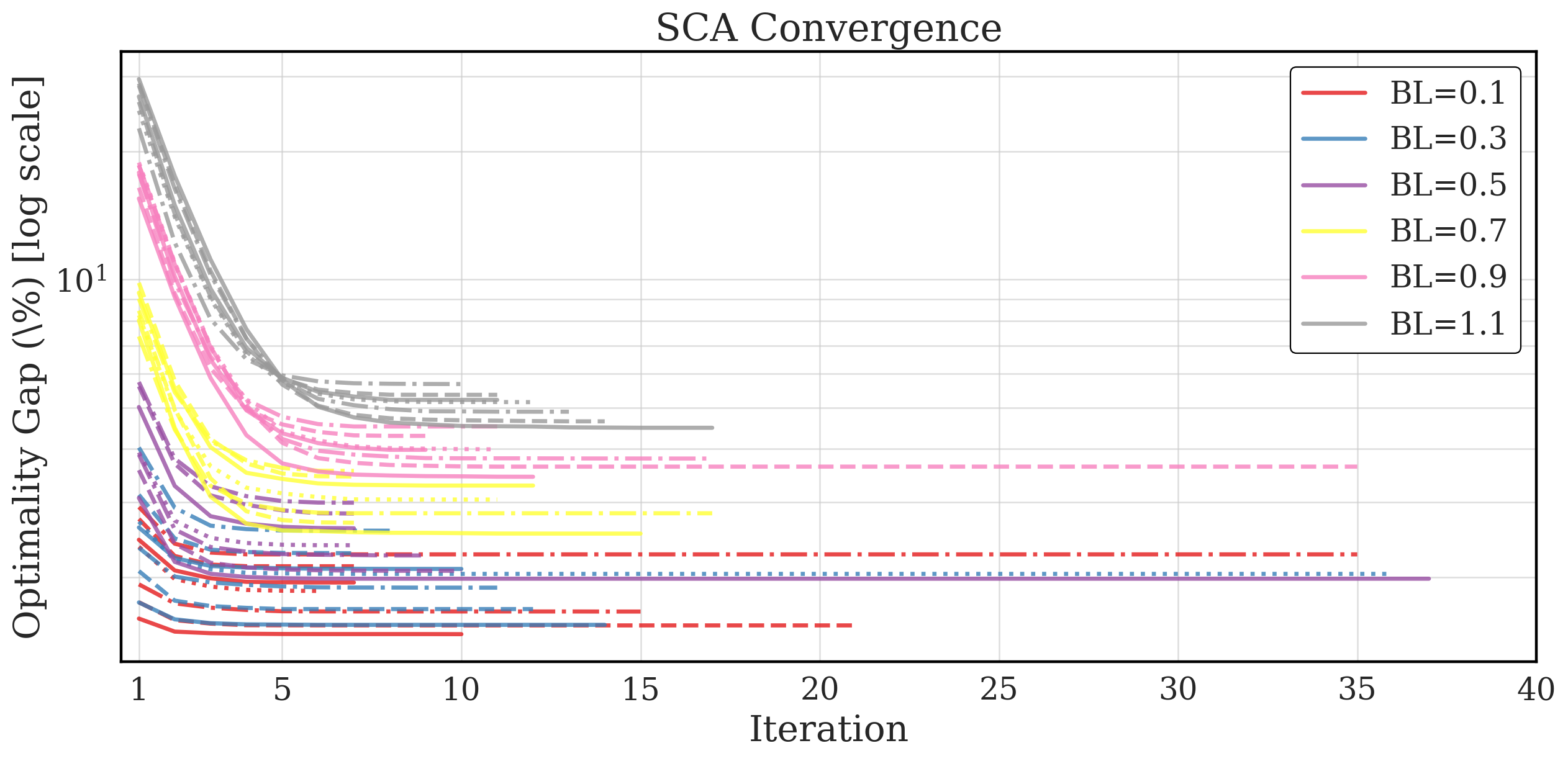}
        \caption{
        Certified optimality gap of the \gls{acr:amod} best-response oracle across \gls{acr:sca} iterations under different background traffic levels. Each trajectory corresponds to one \gls{acr:amod} operating-cost scenario under the corresponding traffic level indicated by color.}
        \label{fig:sca_convergence}
    \end{minipage}
\end{figure}

\subsection{\gls{acr:amod} oracle: solution quality and convergence} \label{sec:res_amod}
We next assess the \gls{acr:amod} deviation oracle, whose value enters the payment and must therefore be computed with certified accuracy. As established in \cref{sec:amod_br}, the convex relaxation gives a global upper bound on the best-response profit and \gls{acr:sca} returns feasible incumbents, so their gap certifies the deviation value.
\gzdel{We next assess the numerical quality of the \gls{acr:amod} deviation oracle. 
For the implementation analysis to be credible, the unilateral deviation value for the \gls{acr:amod} operator must be computed with controlled numerical error. 
As established in \cref{sec:amod_br}, the \gls{acr:amod} best-response problem admits a rigorous bounding architecture: a convex relaxation gives a global upper bound on the true profit, while the \gls{acr:sca} procedure returns feasible incumbents and therefore valid lower bounds. This yields a certified optimality gap for the original nonconvex deviation problem.
}

To improve numerical stability, \gls{acr:sca} is warm-started from the network flow assignment computed by the target oracle. 
\gzdel{Because this point is already feasible and physically meaningful, it provides a strong initialization for the local convexification steps.}
\Cref{fig:sca_convergence} plots the certified gap between the feasible incumbent and the global upper bound across successive \gls{acr:sca} iterations. 
Each curve corresponds to a different background traffic level and \gls{acr:amod} operating cost, ranging from $\$0.60$/km to $\$1.20$/km.

The convergence pattern is stable across all tested regimes. 
The gap decreases steadily with iteration count, and the final certified gap remains below $6\%$ in every tested instance. 
\Cref{tab:sca_summary} in Appendix~\ref{app:sca} reports the final gap and total wall-clock time as functions of background traffic. The mean final gap grows from $1.8\%$ at $\mathrm{BL}=0.1$ to $5.1\%$ at $\mathrm{BL}=1.1$. 
The gap grows with congestion because the interaction between routing, rebalancing, pricing, and endogenous congestion becomes sharper, while average runtimes stay at or below about ten minutes in all regimes.\gzdel{As congestion increases, the problem becomes more difficult because the interaction between routing, rebalancing, pricing, and endogenous congestion becomes sharper. The final certified gap increases moderately with congestion, while}\gzdel{runtimes remain within about ten minutes across all regimes. 
Even in the most difficult tested regime, however, the final gap remains small enough to support the downstream implementation analysis.}

Note that the quantity in \cref{fig:sca_convergence} is a certificate derived from the convex relaxation rather than a local convergence residual. The gap at the warm-start point measures how far the target incumbent lies from the upper bound and should not be read as the deviation gain, which should be computed from the final oracle output.

\subsection{\gls{acr:pt} oracle: exactness boundary and capacity-limited regimes} \label{sec:res_pt}
We now study the \gls{acr:pt} deviation oracle, asking whether the \gls{acr:micp} relaxation is exact and how informative it remains when it is not. At the baseline calibration ($K_l=320$\,\$/VRH and $C_l=900$ passengers per train), the relaxation is exact at $\mathrm{BL}=0.3$ and is certified within $2.6\%$ and $3.0\%$ at the more congested regimes $\mathrm{BL}=0.7$ and $\mathrm{BL}=0.9$, where the gap is measured between the relaxed upper bound and a feasible capacity-repaired lower bound. Thus, in the parameter region most relevant for the Manhattan case study, the \gls{acr:pt} deviation term is either exact or tightly bounded. We also examine the gap across a grid of transit capacities and operating costs, and the results are presented in Appendix~\ref{app:pt_exact}. These results show that exactness holds when capacity is abundant relative to induced demand, and the large gaps under bus-scale stress-test capacities signal a capacity-truncated, demand-shedding regime as discussed in \cref{remark:pt_relax_int}. From the municipality's viewpoint, this regime is precisely the kind of misalignment that generates a nontrivial implementation burden.
\gzmoved{the exactness-map discussion and Table 2 were moved to Appendix~\ref{app:pt_exact}.} \revised{Appendix~\ref{app:multipath_time} further evaluates the generalized multiple-candidate-path formulation, reporting its computation performance and certified bounds as the candidate set grows.}

\subsection{Value of coordination: Baseline Comparisons}\label{sec:coord}
Before analyzing the payment, we quantify the cost of uncoordinated operation by comparing the \revised{regularized target profile} with two baselines. (1) \textbf{Joint best response (BR)}: both operators play their unilateral best responses to the \revised{target profile}. This profile is not an equilibrium of the operator game, but it indicates the efficiency loss when both operators act on their private incentives without regulation. (2) \textbf{Static baseline}: transit frequencies are fixed at the historical GTFS schedule and the \gls{acr:amod} operator best-responds to them, which represents an uncoordinated status quo. \cref{fig:gap} reports the relative increase in social cost. The joint BR deviation is modest under light traffic, about 6 to 7\%, and rises to 34\% under heavy traffic, reflecting the externalities of decentralized best responses. The static baseline exceeds the social cost of the target by more than 50\% at every traffic level. \revised{In this calibration,} the magnitude of these gaps motivates regulatory intervention by the \gls{acr:mcp}.
\begin{figure}[tb]
    \centering
    \includegraphics[width=0.45\linewidth]{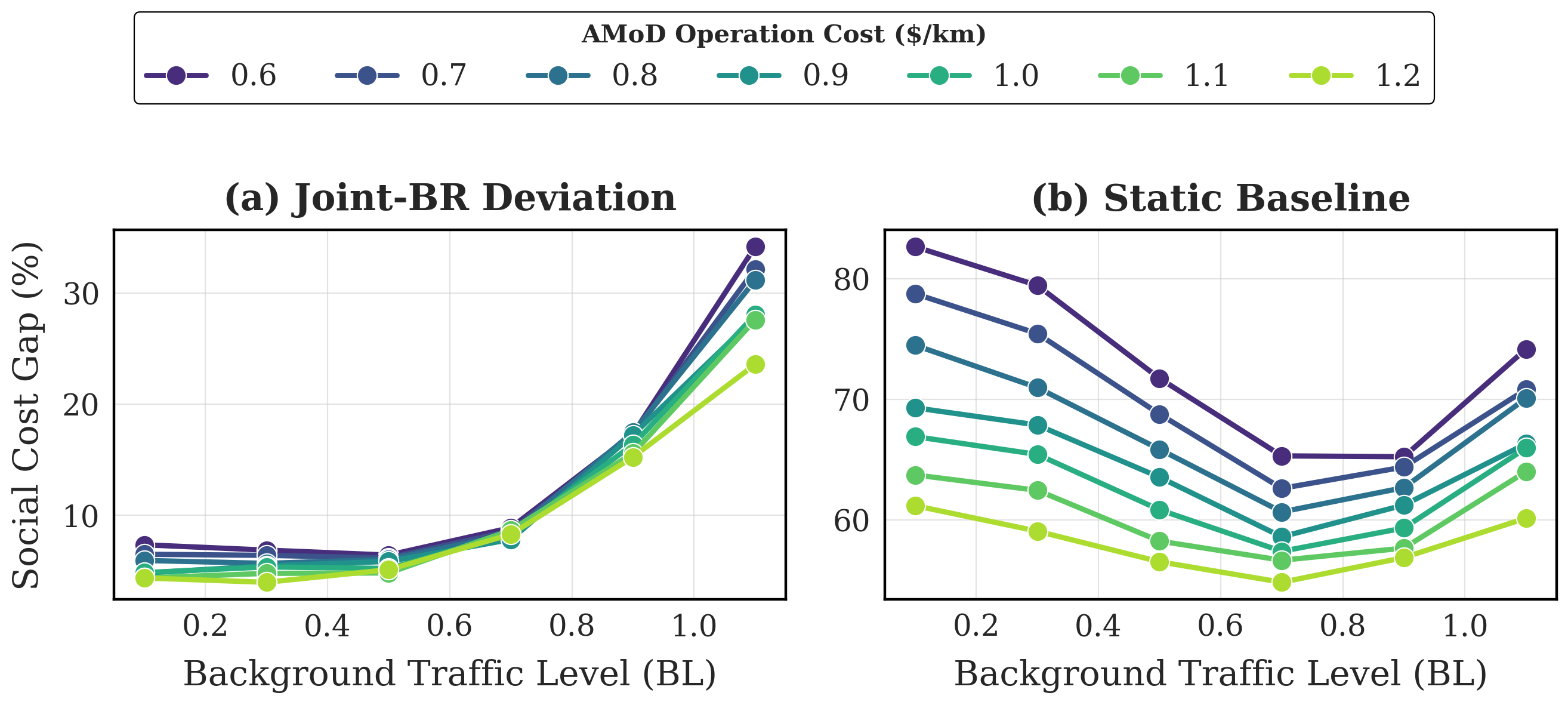}
    \caption{Excess social cost due to deviation and uncoordinated operator behavior.}
    \label{fig:gap}
\end{figure}
\subsection{Implementation payment for the \revised{target profile}} \label{sec:payment}
We now combine the three oracle outputs into the implementation payment $k(z)$ for the \revised{target profile}. \cref{fig:k_res} reports the payment to each operator, the total payment, and the social cost of the target as functions of the background traffic level, with one curve per \gls{acr:amod} operating cost between \$0.60/km and \$1.20/km. \revised{In the Manhattan calibration,} the total payment is non-monotone (U-shaped) in background traffic. It decreases as congestion rises from low levels and increases again under heavy traffic. The operating cost shifts the level of the payment, whereas the non-monotone profile persists across the tested range. The incentive gap between the \revised{target profile} and the operators' private objectives is therefore largest at the two ends of the congestion range.

\begin{figure}[tb]
    \centering
    \begin{minipage}[t]{0.48\textwidth}
        \centering
        \includegraphics[width=\linewidth]{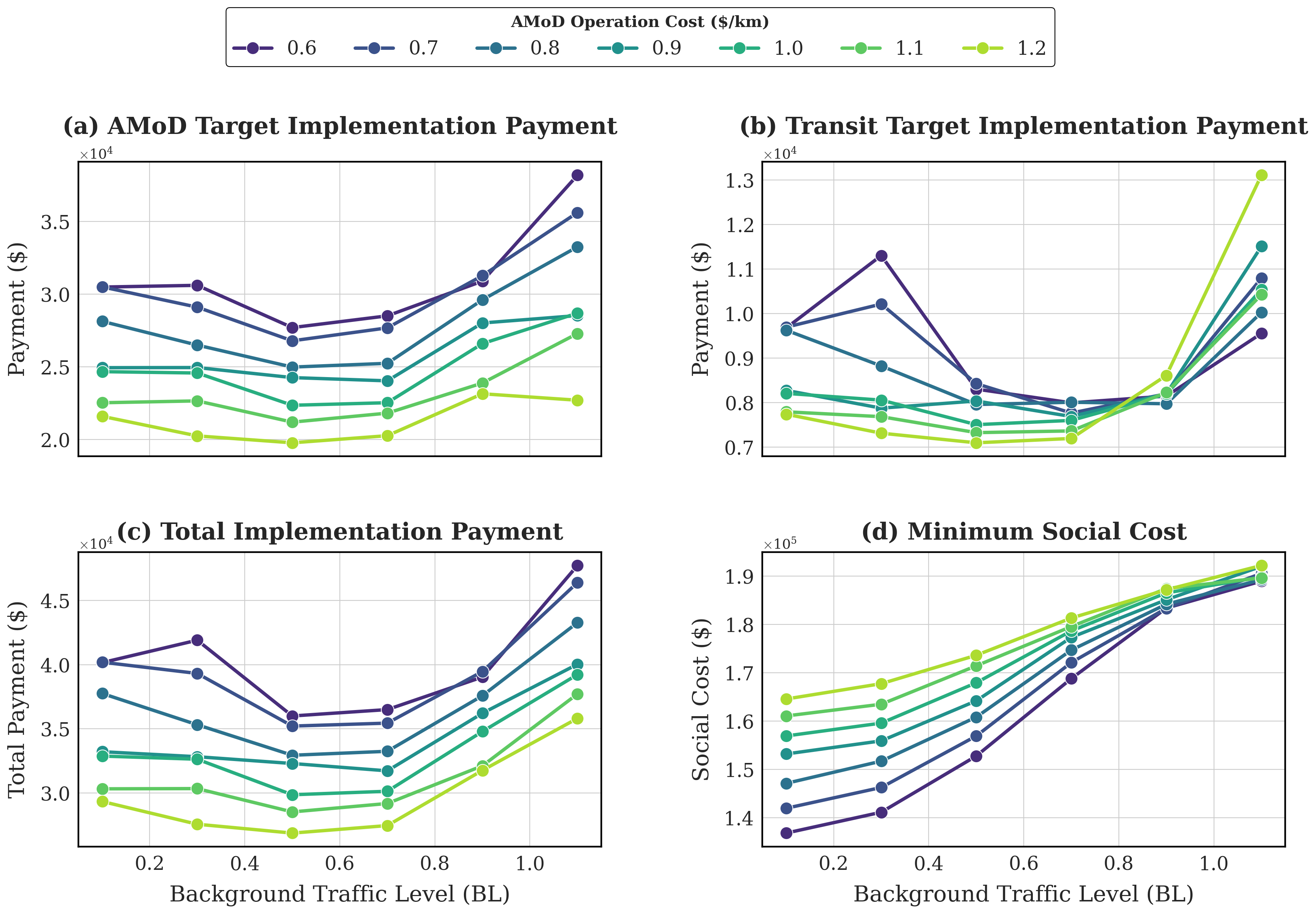}
        \caption{Implementation-payment profile for each operator and social cost of the target profile as background traffic varies, with one curve per \gls{acr:amod} operating cost}.
        \label{fig:k_res}
    \end{minipage}\hfill
    \begin{minipage}[t]{0.48\textwidth}
        \centering
        \includegraphics[width=\linewidth]{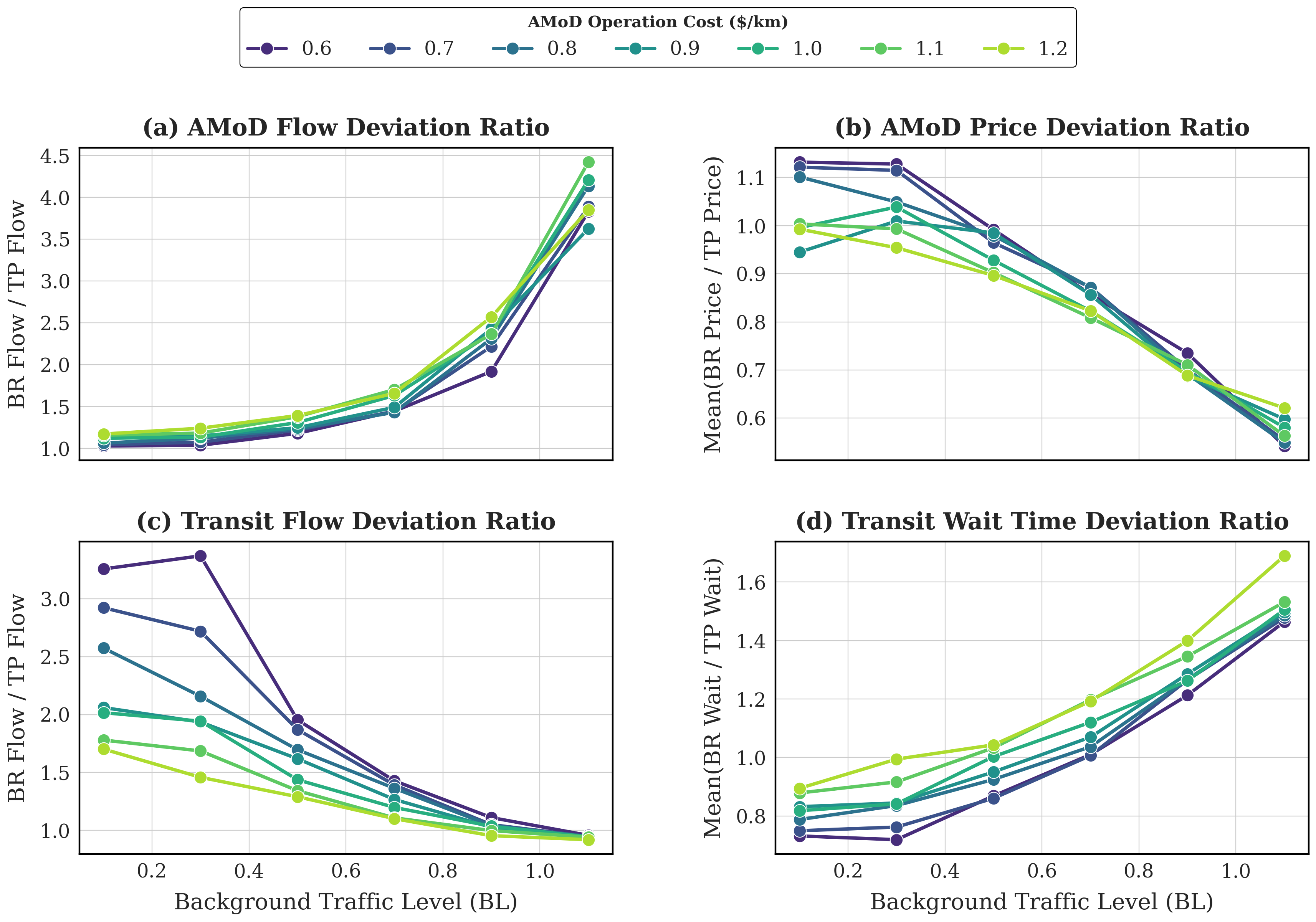}
        \caption{Deviation mechanisms underlying the implementation payment: ratios of \gls{acr:amod} flow, \gls{acr:amod} price, and \gls{acr:pt} waiting time between operator best response (BR) and \revised{target profile (TP)}.}
        \label{fig:k_analysis}
    \end{minipage}
\end{figure}

\Cref{fig:k_analysis} explains the mechanisms behind this U-shape. 
Panels (a) and (b) compare the \gls{acr:amod} best response with the target profile. At low BL, roads are uncongested and \gls{acr:amod} offers a strong time advantage over transit, so the target assigns substantial demand to \gls{acr:amod}, and the self-interested operator exploits this market position by charging higher prices than is socially desirable. The payment to \gls{acr:amod} then mainly offsets market-power incentives. At high BL, the mechanism reverses. Additional \gls{acr:amod} vehicles now create substantial congestion externalities, so the target restricts \gls{acr:amod} usage, whereas the operator internalizes fare revenue but not the congestion it imposes on the rest of the road network and seeks more riders than is socially optimal. The payment then compensates the operator for internalizing an externality it would otherwise ignore.
\gzdel{Panels (a) and (b) compare the \gls{acr:amod} best response with the target profile. 
At low BL, roads are uncongested and \gls{acr:amod} offers a strong time advantage over transit.
The municipality therefore assigns substantial demand to \gls{acr:amod} in the social target, but the self-interested \gls{acr:amod} operator exploits this favorable market position by charging higher prices than is socially desirable. 
In this regime, the implementation payment to \gls{acr:amod} mainly offsets market-power incentives. 
At high BL, the mechanism reverses. 
Because additional \gls{acr:amod} vehicles now create substantial congestion externalities, the social target restricts \gls{acr:amod} usage. 
The operator, however, internalizes private revenue but not the congestion imposed on the rest of the road network and therefore seeks to attract more riders than is socially optimal. 
In this regime, the payment to \gls{acr:amod} compensates the operator for internalizing an externality it would otherwise ignore.
}

Panels (c) and (d) reveal the corresponding \gls{acr:pt} mechanism. 
At low BL, the target relies more heavily on \gls{acr:amod} and economizes on transit operating cost through lower frequencies, whereas the self-interested \gls{acr:pt} operator over-supplies frequency to retain ridership. At high BL, transit becomes the socially preferred high-capacity mode and the target calls for aggressive frequency provision to absorb demand diverted from the congested roads, whereas the operator stops adding service once its marginal operating cost exceeds the privately captured ridership benefit. Under heavy congestion, the payment to \gls{acr:pt} therefore compensates the operator for providing more service than it would choose on its own.
\gzdel{At low BL, the social target can rely more heavily on \gls{acr:amod}, so it economizes on transit operating cost by selecting a lower transit frequency. 
The self-interested \gls{acr:pt} operator instead tends to over-supply frequency in order to retain ridership, resulting in shorter waiting times than the target. 
At high BL, transit becomes the socially preferred high-capacity mode, and the target profile calls for aggressive frequency provision to absorb demand shifted away from the congested road network.
The \gls{acr:pt} operator, by contrast, behaves more conservatively because the marginal operating cost of additional service eventually exceeds the privately captured ridership benefit. Hence, under heavy congestion, the implementation payment to \gls{acr:pt} compensates the operator for providing more service than it would choose based on its own objective.
} \revised{The high-congestion side of this mechanism is the Mohring effect \citep{mohring1972optimization} in transit systems. Under high congestion, more travelers switch to transit, which induces higher frequencies. Higher frequency lowers every rider's waiting
time, which the social objective values and the operator values only through retained ridership. Therefore, the operator under-supplies frequency and the transit component of the payment is the compensation for closing that gap. Density economies on the \gls{acr:amod} side, through which additional demand shortens pickup times and empty travel, are not considered in the model. Representing them would require a different endogenous service-quality model, for which the oracle formulations and convexity results developed here would need to be re-established.
}

\revised{These conclusions do not depend on the single-path representation of transit. Appendix \ref{app:multipath} repeats the implementation analysis with up to five candidate transit paths per \gls{acr:od} pair. Although the richer route choice
changes passenger-level outcomes, shifting about a third of transit ridership onto alternative paths and shortening waiting times, the incentive picture is stable: the U-shaped profile of the payment in background traffic persists, the
deviation ratios of \Cref{fig:k_analysis} are nearly coincident across candidate sets, and the total payment stays within about 12\% of the single-path value over the whole certified solver interval.
}

\begin{revisedblock}
\subsection{Trade-offs between the social cost and implementation payment} \label{sec:tradeoff}
\Cref{sec:payment} evaluated the implementation framework with one target profile. In practice, a \gls{acr:mcp} might want to choose the target profile based on the implementation payment needed. \Cref{sec:so} and \Cref{sec:res_so} showed that the regularization weight $\theta$ indexes a family of targets, with small $\theta$ yielding a target farther from the unregularized social optimum, and larger $\theta$ approaches it. In this subsection, we run the implementation analysis on a family of target profiles with varying $\theta$ and show how the implementation payment depends on the target that the \gls{acr:mcp} chooses. \Cref{fig:tradeoff}(a) plots
the total payment against the optimality loss of the target, and \Cref{fig:tradeoff}(b) its decomposition between the two operators. We refer to the results as the target--payment trade-off curve. The values behind those panels are listed in Appendix \ref{app:theta_table}.

\begin{figure}[tb]
    \centering
    \includegraphics[width=0.55\linewidth]{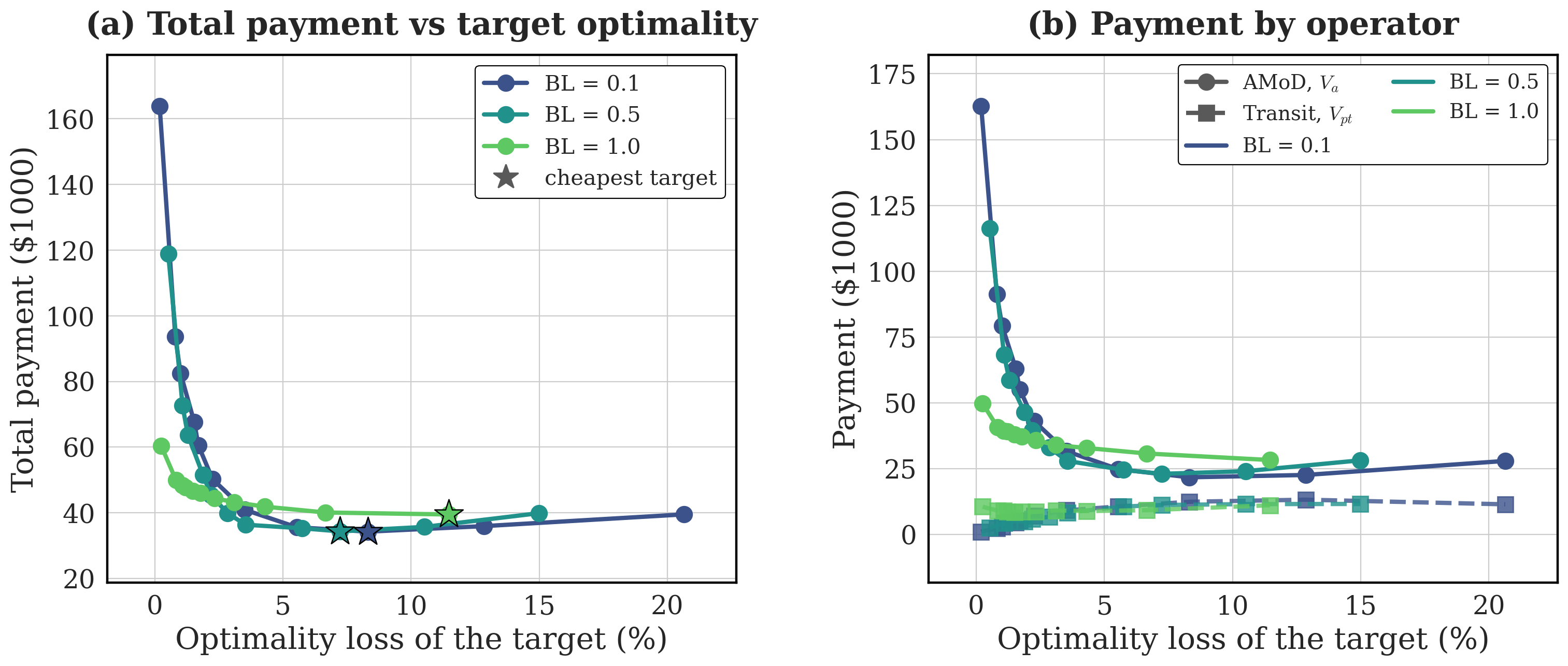}
    \caption{\revised{Implementation payment of target profiles with different optimality loss compared to the unregularized social optimum, under three representative BL levels.}}
    \label{fig:tradeoff}
\end{figure}

The trade-off is steep once the optimality loss falls below about $5\%$, and its steepness depends on the background congestion. At low and intermediate background traffic, moving the target from about a $5\%$ loss to the least-regularized tested target raises the implementation payment by a factor of about $4.6$ at $\mathrm{BL}=0.1$ and $3.4$ at $\mathrm{BL}=0.5$. Under heavy congestion, the same movement raises it by a factor of about $1.5$. The reason follows from \Cref{sec:payment}: under heavy traffic the target already restricts \gls{acr:amod} usage, so sharpening it moves the operator's operating point comparatively little. The trade-off curves are also not monotone. At low and intermediate congestion, the payment is minimized at an interior target, with an optimality loss of about $7$--$8\%$, and rises again for less optimal targets. This indicates that a less optimal target is not automatically a cheaper one. Beyond the minimum, a target that lies farther from the unregularized social optimum also lies, on some \gls{acr:od} pairs, farther from the operating point each operator would privately choose; this raises the operators' unilateral deviation incentives and with them the payment.

\Cref{fig:tradeoff}(b) then shows that the increase in payment along the curve in \Cref{fig:tradeoff}(a) is primarily driven by the \gls{acr:amod} operator. Sharpening the target moves demand toward the social optimum split, and the payment needs to compensate the \gls{acr:amod} operator for forgoing the market power it would otherwise exercise through both the pricing and routing strategies.

Taken together, the numerical studies support three conclusions. The first is computational: in the benchmark Manhattan regimes, the three oracles are either exact or tightly bounded, so the resulting implementation analysis is numerically credible. The second is economic: the implementation payment is interpretable and decomposable. It identifies not only how large the \gls{acr:mcp}'s intervention must be, but also the source of the underlying friction. The third concerns the target itself: the framework tells the municipality not only what implementing the target costs, but how steeply that cost rises with the last few percent of optimality to the social optimum. This is precisely the diagnostic role that motivates the proposed implementation-based framework.
\end{revisedblock}
\section{Conclusion}\label{sec:Conclusion}
In this paper, we studied how a \gls{acr:mcp} can induce a socially preferred operating point when the \gls{acr:amod} and \gls{acr:pt} operators are self-interested. Building on $k$-implementation, we replaced the outer equilibrium search of bilevel regulation models with one target problem and two unilateral best-response problems, and the Nash-equilibrium requirements with the weaker assumption that operators do not play dominated strategies. The target oracle combines explicit mode choice, endogenous road congestion, operator-specific service design, and transit capacity in an entropy-regularized mixed-integer convex program obtained through price recovery. For the \gls{acr:amod} operator, an exact reformulation yields a convex relaxation with a global upper bound and a sequential convexification with feasible incumbents. For the \gls{acr:pt} operator, a mixed-integer convex relaxation is exact precisely when the induced logit demand is capacity-feasible, and its slack otherwise identifies a capacity-constrained regime. Consequently, the implementation payment is exact when both oracles are exact and is bracketed by certified bounds otherwise.

\revised{The Manhattan case study provides an empirically grounded proof of concept of the framework at urban scale under the behavioral and network specification considered here.} Because the payment decomposes by operator, it identifies both the size of the required intervention and its operational source, which in this calibration shifts from \gls{acr:amod} market power under light traffic to road externalities and transit under-provision under heavy traffic. The broader lesson is that incentive design for multimodal systems need not be an intractable outer layer on top of system design. A target model combined with unilateral-deviation models yields a constructive and interpretable measure of the incentive gap between decentralized operation and the \revised{target}.
\gzdel{As demonstrated across the experiments, the framework acts as an economic diagnostic tool rather than simply returning a monolithic subsidy number. Because the implementation payment is decomposable, it successfully isolates exactly why operators deviate under different conditions—identifying not just the magnitude of the required subsidies, but also why those subsidies are needed. The main message of the paper is therefore both conceptual and computational. 
Incentive design for multimodal transportation systems need not be treated as an intractable outer layer added on top of an already difficult system-design problem.
By exploiting the structure of $k$-implementation, the \gls{acr:mcp} can combine a target model with unilateral-deviation models and obtain a constructive, interpretable measure of the incentive gap between decentralized operation and the \revised{target}.
}
% This positions the framework not only as a regulatory design tool, but also as a diagnostic tool for understanding misalignment in multimodal transportation systems.

Several directions remain open. 
First, a natural extension of the singleton targets studied here is set-valued implementation, in which the municipality is willing to accept any operating point in a socially desirable set and seeks the least costly one to implement. 
Second, \revised{\Cref{sec:tradeoff} traces the target-payment trade-off along a family of targets generated by the entropy regularizer. A natural extension is to optimize jointly over the target itself and the implementation payment, which would reintroduce a bilevel target-selection problem and is left for future work.}
Third, our formulation assumes that operator objectives are known and that transfers can be conditioned on the relevant realized actions.
Extending the framework to robust or partially observed implementation, for example, when the \gls{acr:mcp} has only noisy estimates of operator utilities or can contract only on aggregate, observable outcomes, would move the approach closer to practical deployment. Finally, incorporating stochastic demand, rolling-horizon operations, density economies in \gls{acr:amod} matching and pooling, and richer passenger choice structures would broaden the empirical scope of the framework while preserving its core implementation logic.

\bibliographystyle{naturemag}
\bibliography{references}

\begin{APPENDICES}
% \section{Proof of \cref{prop:dcomp}}\label{app:lemma1}
% \begin{proof}{Proof}
% The price variables enter the model only through the \gls{acr:amod} utility \eqref{constr:a_v} and the mode-choice constraint \eqref{constr:choice}; they do not appear in the objective \eqref{eq:so_obj}. 
% Therefore, every feasible solution of the original problem projects to a feasible solution of the reduced problem with the same objective value. 
% Conversely, if~$(x^{a*},x^{pt*},f^*,r^*,s^*,\ldots)$ solves the reduced problem, then for each OD pair the binary logit equation can be inverted to recover a price~$\pi_{od}^*$ satisfying \eqref{constr:choice}.
% The recovered solution satisfies \eqref{constr:choice} and therefore is feasible for the original problem with unchanged objective value. Hence the decomposition preserves optimality. \qed
% \end{proof}
\begin{revisedblock}
\section{Implementation for continuous strategies}\label{app:imp}
The following proof closely follows the structure of the proof of \Cref{theorem:k} in the original paper by \cite{monderer2003k}.
\begin{proof}{Proof}
Consider any nonnegative transfer vector $V=(V_1,\dots,V_n)$ that makes the target strategy $z_i$ dominant for every player $i$. For any $x_i \in X_i$, dominance at the opponents' target strategy $z_{-i}$ requires
\[U_i(z_i,z_{-i})+V_i(z_i,z_{-i})\geq U_i(x_i,z_{-i})+V_i(x_i,z_{-i}).\]
Since transfers are nonnegative, we can get
\[V_i(z) \geq U_i(x_i,z_{-i})-U_i(z_i,z_{-i})\]
for every $x_i \in X_i$. Taking the maximum over $x_i$ yields
\[V_i(z) \geq \max_{x_i \in X_i}\bigg[U_i(x_i,z_{-i})-U_i(z_i,z_{-i})\bigg].\]
Summing over players gives
\begin{equation}\label{eq:lower_v}
    \sum_iV_i(z) \geq \sum_i \max_{x_i \in X_i}\bigg[U_i(x_i,z_{-i})-U_i(z_i,z_{-i})\bigg].
\end{equation}
We next construct a transfer vector attaining this bound. Define
\[D_i(x_{-i}):=\max_{x_i \in X_i}\bigg[U_i(x_i,x_{-i})-U_i(z_i,x_{-i})\bigg].\]
By definition, $D_i(x_{-i})$ is finite for every $x_{-i}$. Set
\[V_i(x_i,x_{-i})=0 \text{ for } x_i \neq z_i,\]
and 
\[V_i(z_i,x_{-i})=D_i(x_{-i})+\delta_i(x_{-i}),\]
where $\delta_i(x_{-i}) \geq 0, \delta_i(z_{-i})=0,$ and $\delta_i(\overline{x}_{-i})>0$ for at least one $\overline{x}_{-i} \neq z_{-i}$. Such $\overline{x}_{-i}$ exists because there are at least two players and every player has at least two strategies.
For every $x_i \neq z_i$ and every $x_{-i}$, by construction we have
\[D_i(x_{-i}) \geq U_i(x_i,x_{-i})-U_i(z_i,x_{-i}).\]
Therefore
\[U_i(z_i, x_{-i})+V_i(z_i,x_{-i})\geq U_i(x_i,x_{-i})+V_i(x_i,x_{-i}).\]
Moreover, at $x_{-i}=\overline{x}_{-i}$ the additional term $\delta_i(\overline{x}_{-i})>0$ makes this inequality strict. Therefore $z_i$ dominates every $x_i \neq z_i$ under this transfer design. 

Finally, because $\delta_i(z_{-i})=0$, we get the transfer payment at target profile as
\[V_i(z)=D_i(z_{-i})=\max_{x_i \in X_i}\bigg[U_i(x_i,z_{-i})-U_i(z_i,z_{-i})\bigg].\]
Summing over $i$ gives the transfers that attains the lower bound in \eqref{eq:lower_v}, which closes the proof for \Cref{prop:continuous}.
\end{proof}    
\end{revisedblock}

\section{Proof of \cref{prop:dcomp} and exactness of the big-M linearization}
\subsection{Proof of \cref{prop:dcomp}}\label{app:lemma1}
\gzmoved{moved from the main text (\cref{sec:so}).}
\begin{proof}{Proof}
The price variables enter the model only through the \gls{acr:amod} utility \eqref{constr:a_v} and the mode-choice constraint \eqref{constr:choice}; they do not appear in the objective \eqref{eq:so_obj}. 
Therefore, every feasible solution of the original problem projects to a feasible solution of the reduced problem with the same objective value. 
Conversely, if~$(x^{a*},x^{pt*},f^*,r^*,s^*,\ldots)$ solves the reduced problem, then for each OD pair the binary logit equation can be inverted to recover a price~$\pi_{od}^*$ satisfying \eqref{constr:choice}.
The recovered solution satisfies \eqref{constr:choice} and therefore is feasible for the original problem with unchanged objective value. Hence the decomposition preserves optimality. \qed
\end{proof}
\subsection{Exactness of the big-M linearization}\label{app:bigm}
\begin{revisedblock}
Suppose $z_{ls}\in\{0,1\}$ and~$0\le x_{od,p}^{pt}\le \alpha_{od}$. By introducing the continuous auxiliary variable $w_{od,p,l,s} = x_{od,p}^{pt}\cdot z_{ls} \geq 0$, the bilinear term $x^{pt}_{od,p}\cdot z_{ls}$ is exactly linearized via the following linear envelope constraints:
\begin{equation} \label{constr:envelope}
\begin{aligned}
    w_{od,p,l,s} & \leq x^{pt}_{od,p} \\
    w_{od,p,l,s} & \leq \alpha_{od}\cdot z_{ls} \\
    w_{od,p,l,s} & \geq x^{pt}_{od,p} - \alpha_{od}(1-z_{ls}) \\
    w_{od,p,l,s} & \geq 0
\end{aligned}
\end{equation}
Because $z_{ls} \in \{0,1\}$ and $x_{od,p}^{pt}$ is bounded between 0 and~$\alpha_{od}$, the four inequalities in \eqref{constr:envelope} enforce~$w_{od,p,l,s}=x_{od,p}^{pt}$ when~$z_{ls}=1$ and~$w_{od,p,l,s}=0$ when $z_{ls}=0$. 
Hence \eqref{constr:envelope} is exactly equivalent to \eqref{eq:so_aux}.
\end{revisedblock}

\section{Proof of \cref{lemma:logit}}\label{app:lemma2}
\begin{proof}{Proof}
For fixed $\boldsymbol{z}$ and $\boldsymbol{y}$, the subproblem is convex. Let $J$ denote the non-entropic part of the objective, and let $\{g_k\}_{k\in K}$ denote the remaining constraints with multipliers $\{\mu_k\}_{k\in K}$. The Lagrangian is
\begin{align}         
\mathcal{L} = J + \frac{1}{\theta}
\sum_{(o,d)\in\od}\sum_{m\in\{a,pt\}} x_{od}^m\ln x_{od}^m
+
\sum_{k\in K}\mu_k g_k, \nonumber
\end{align}
where $m\in\{a,pt\}$. 
The KKT stationarity condition with respect to $x_{od}^m$ gives
\[
\frac{\partial J}{\partial x_{od}^m}
+\sum_k \mu_k\frac{\partial g_k}{\partial x_{od}^m}
+\frac{1}{\theta}\left(\ln x_{od}^{m*}+1\right)=0.
\]
Define
\[
\mathrm{MSC}_{od}^m
=
\frac{\partial J}{\partial x_{od}^m}
+\sum_k \mu_k\frac{\partial g_k}{\partial x_{od}^m}.
\]
Applying stationarity for both modes and subtracting the two equations yields
\[
\ln\frac{x_{od}^{a*}}{x_{od}^{pt*}}
=
\theta\left(\mathrm{MSC}_{od}^{pt}-\mathrm{MSC}_{od}^{a}\right),
\]
which proves the claim. \qed
\end{proof}

\section{Proof of \cref{prop:amod_reform}}\label{app:lemma3}
\begin{proof}{Proof}
By taking the natural logarithm of the binary logit choice relation \eqref{constr:choice}, the unique price $\pi_{od}$ supporting a given flow can be analytically written as:
\begin{equation}\label{eq:choice_price}
    \begin{aligned}
        \pi_{od} &= \\
        &-\big(\ln x_{od}^a-\ln (\alpha_{od} - x_{od}^a)+\mathrm{VOT} \cdot t_{od}^a+V_{od}^{pt}\big). 
    \end{aligned} 
\end{equation}
Substituting \eqref{eq:choice_price} directly into the objective function \eqref{eq:amod_obj} removes the choice constraint \eqref{constr:choice} and transforms the objective into the equivalent minimization
\begin{equation}\label{obj:amod_eq} 
\begin{aligned}
    \min_{\boldsymbol{x}^a, \boldsymbol{f}, \boldsymbol{r}} ~& \sum_{od} \big(x_{od}^a\ln x_{od}^a-x^a_{od}\ln (\alpha_{od}-x_{od}^a)+\\ &\mathrm{VOT} \cdot x_{od}^at_{od}^a+ x_{od}^aV_{od}^{pt}\big)
    + \sum_{(i,j)\in \mathcal{E}^\mathrm{a}}C_{ij}f_{ij}    
\end{aligned}
\end{equation}
Building upon \eqref{obj:amod_eq}, we can further eliminate the explicit expected-travel-time constraint \eqref{constr:a_time} by projecting the path-based bilinear term~$x_{od}^at_{od}^a$ in \cref{obj:amod_eq} onto the edges
\begin{equation}\label{eq:amod_bi}
    \begin{aligned}
        \sum_{od} x_{od}^at_{od}^a &= \sum_{od,p} x_{od,p}^a\sum_{(i,j)\in \mathcal{E}^a }T(b_{ij}+f_{ij})\delta^p_{ij} \\
        &= \sum_{(i,j)\in \mathcal{E}^a } (f_{ij}-r_{ij})T(b_{ij}+f_{ij}).
    \end{aligned}
\end{equation}
Substituting \eqref{eq:amod_bi} into \eqref{obj:amod_eq} removes the final non-convex constraint \eqref{constr:a_time}. The log term in the objective \eqref{obj:amod_eq} is convex, since for function
\[
H(x):=x\ln x - x\ln(\alpha-x),
\qquad x\in(0,\alpha),
\]
we have
\[
H''(x)=\frac{\alpha^2}{x(\alpha-x)^2}>0.
\]
Therefore, the problem is now subject exclusively to linear physical network constraints, with all structural non-convexity strictly isolated within the edge-based term $\phi_{ij}(f_{ij},r_{ij}) := (f_{ij}-r_{ij})T(b_{ij}+f_{ij})$. \qed
\end{proof}

\section{\gls{acr:sca}}
\label{app:sca}
\gzmoved{the following linearization was moved from \cref{sec:amod_br}.}
To obtain a high-quality feasible primal solution, we address the exact bilinear term by isolating the concave residual $-r_{ij}T(b_{ij}+f_{ij})$ and applying \gls{acr:sca}. At iteration $k$, let
\begin{equation*}
    \begin{aligned}
        T_{ij}^{(k)} &:= T(b_{ij}+f_{ij}^{(k)}),\\
        \nabla T_{ij}^{(k)}
        &=
        t_{ij}^0 \alpha \frac{\beta}{c_{ij}}
        \left(\frac{b_{ij}+f_{ij}^{(k)}}{c_{ij}}\right)^{\beta-1}.
    \end{aligned}
\end{equation*}
The term $G_{ij}(r_{ij},f_{ij}) := -r_{ij}T(b_{ij}+f_{ij})$ is approximated by its first-order Taylor expansion around the state $(r_{ij}^{(k)},f_{ij}^{(k)})$:
\begin{equation}
\begin{aligned}
G_{ij}(r_{ij},f_{ij})
\approx\;
&-r_{ij}^{(k)}T_{ij}^{(k)}
-
T_{ij}^{(k)}(r_{ij}-r_{ij}^{(k)})
\\
&-
r_{ij}^{(k)}\nabla T_{ij}^{(k)}(f_{ij}-f_{ij}^{(k)}).
\end{aligned}
\label{eq:amod_taylor}
\end{equation}
Substituting \eqref{eq:amod_taylor} into the objective yields a convex subproblem, which we solve within trust-region bounds:
\[
f_{ij}^{(k)}-\delta_f \le f_{ij} \le f_{ij}^{(k)}+\delta_f,
\qquad \forall (i,j)\in\aedges,
\]
\[
r_{ij}^{(k)}-\delta_r \le r_{ij} \le r_{ij}^{(k)}+\delta_r,
\qquad \forall (i,j)\in\aedges.
\]
The resulting procedure is reported in \cref{algo:sca}.

\begin{algorithm}
\small
\caption{\gls{acr:sca} for the \gls{acr:amod} best response}
\label{algo:sca}
\begin{algorithmic}[1]
\Require Initial point $(f^{(0)},r^{(0)})$, trust-region radii $(\delta_f,\delta_r)$, tolerances $(\epsilon_f,\epsilon_r, \epsilon)$
\State Set $k\gets 0$
\State $\eta^{(0)} \gets J(f^{(0)}, r^{(0)})$ \Comment{Evaluate initial objective}
\Repeat
    \State Compute $T_{ij}^{(k)}$ and $\nabla T_{ij}^{(k)}$ for all $(i,j)\in\aedges$
    \State Form the convex subproblem by replacing $-r_{ij}T(b_{ij}+f_{ij})$ with its first-order approximation \eqref{eq:amod_taylor}
    \State Solve the convex subproblem subject to
    \[
    f_{ij}^{(k)}-\delta_f \le f_{ij} \le f_{ij}^{(k)}+\delta_f,
    \qquad
    r_{ij}^{(k)}-\delta_r \le r_{ij} \le r_{ij}^{(k)}+\delta_r,
    \]
    for all $(i,j)\in\aedges$
    \State Let the solution be $(f^{(k+1)},r^{(k+1)})$
    \State $\eta^{(k+1)} \gets J(f^{(k+1)}, r^{(k+1)})$ \Comment{Track improvement}
    \State $k\gets k+1$
\Until{($|f_{ij}^{(k)}-f_{ij}^{(k-1)}|\le\epsilon_f$ and $|r_{ij}^{(k)}-r_{ij}^{(k-1)}|\le\epsilon_r$ $\forall (i,j)$) \textbf{or} $|\eta^{(k)}-\eta^{(k-1)}| \le \epsilon$}
\end{algorithmic}
\end{algorithm}

\cref{tab:sca_summary} reports the final certified gap and the computation time of \cref{algo:sca} in the experiments of \cref{sec:res_amod}.
\gzmoved{the following table was moved from \cref{sec:res_amod}.}
\begin{table}[htbp]
\centering
\small
\caption{Final certified optimality gap and total computation time for the \gls{acr:amod} best-response \gls{acr:sca} algorithm under different background traffic levels.}
\label{tab:sca_summary}
\begin{tabular}{l cccccc}
\toprule
& \multicolumn{6}{c}{\textbf{Background Traffic Level (BL)}} \\
\cmidrule(lr){2-7}
\textbf{Metric} & \textbf{0.1} & \textbf{0.3} & \textbf{0.5} & \textbf{0.7} & \textbf{0.9} & \textbf{1.1} \\
\midrule
\textbf{Final Gap (\%)} & $1.8 \pm 0.3$ & $2.0 \pm 0.3$ & $2.4 \pm 0.4$ & $3.1 \pm 0.4$ & $4.0 \pm 0.3$ & $5.1 \pm 0.4$ \\
\textbf{Time (min)}     & $6.6 \pm 2.8$ & $6.2 \pm 0.9$ & $5.8 \pm 2.0$ & $7.7 \pm 2.4$ & $10.2 \pm 4.2$ & $9.5 \pm 3.1$ \\
\bottomrule
\addlinespace
\multicolumn{7}{l}{\footnotesize \textit{Note:} \gls{acr:sca} is run with $\epsilon_r=\epsilon_f=10^{-4}$ and $\epsilon=10^{-3}$.}
\end{tabular}
\end{table}

\section{Proof of \cref{lemma:pt_relax}}\label{app:lemma4}
\begin{proof}{Proof}
As shown in Appendix \ref{app:lemma3}, the non-linear component $H(x):=x\ln x - x\ln(d_{od}-x)$ is convex for all valid flows $x \in (0, d_{od})$. The rest of the terms are all linear, so the left-hand side of \eqref{constr:pt_br_choice_convex} is a convex function. \eqref{constr:pt_br_choice_convex} defines the level set of a convex function, which is also convex. \qed
\end{proof}

\section{Proof of \cref{prop:pt_exactness}} \label{app:lemma5}
\begin{proof}{Proof}
Let $\mathcal{P}_{rel}$ denote the feasible region of the relaxed \gls{acr:micp}, and let $\mathcal{P}_{orig} \subset \mathcal{P}_{rel}$ denote the feasible region of the original problem with the strict equality constraint \eqref{constr:pt_br_choice_linearized}. Rewrite the objective as $F(\mathbf{x}^{pt}, \mathbf{s}) = \mathbf{c}^\top \mathbf{x}^{pt} - \mathbf{K}^\top \mathbf{s}$, where the marginal revenue vector $\mathbf{c} > \mathbf{0}$ since $P^{pt} + \omega > 0$. By definition, $(\mathbf{x}^{pt*}, \mathbf{s}^*)$ globally maximizes $F$ over $\mathcal{P}_{rel}$.

\textit{Necessity ($\Rightarrow$):} Assume $(\mathbf{x}^{pt*}, \mathbf{s}^*)$ is exactly optimal for the original problem. This requires the solution to be feasible in $\mathcal{P}_{orig}$, meaning the passenger flow must exactly satisfy the logit equilibrium: $\mathbf{x}^{pt*} = \mathbf{x}^{L}(\mathbf{s}^*)$. Because it is a physically feasible solution, it must not violate capacity, meaning $\mathbf{x}^{pt*} \in \Omega_{cap}(\mathbf{s}^*)$. 

\textit{Sufficiency ($\Leftarrow$):} Assume $\mathbf{x}^{L}(\mathbf{s}^*) \in \Omega_{cap}(\mathbf{s}^*)$. We first prove that this solution is feasible for the original problem. In the relaxed \gls{acr:micp}, for the given optimal frequency $\mathbf{s}^*$, the optimizer maximizes $\mathbf{c}^\top \mathbf{x}^{pt}$ subject to the relaxation $\mathbf{x}^{pt} \le \mathbf{x}^{L}(\mathbf{s}^*)$ and the capacity bound $\mathbf{x}^{pt} \in \Omega_{cap}(\mathbf{s}^*)$. Because the objective gradient $\nabla_{\mathbf{x}} F = \mathbf{c} > \mathbf{0}$, the function is strictly monotonically increasing. Since $\mathbf{x}^{L}(\mathbf{s}^*)$ is within the capacity polytope, the optimizer naturally drives the flow to this maximum, selecting exactly $\mathbf{x}^{pt*} = \mathbf{x}^{L}(\mathbf{s}^*)$. Therefore, $(\mathbf{x}^{pt*}, \mathbf{s}^*)$ is a feasible solution for the original problem with the equality constraint, which provides a lower bound to the original problem. Since $(\mathbf{x}^{pt*}, \mathbf{s}^*)$ also provides an upper bound as it is the optimal solution of the relaxed problem, the upper and lower bound are equal at $(\mathbf{x}^{pt*}, \mathbf{s}^*)$. Therefore, $(\mathbf{x}^{pt*}, \mathbf{s}^*)$ is an optimal solution to the original problem, and the relaxation is exact. \qed
\end{proof}

\begin{revisedblock}
\section{Extensions of the binary choice model} \label{app:choice}
\subsection{Stochastic route choice within transit}\label{app:routechoice}
\gzmoved{formerly Remark 1 in \cref{sec:so}, moved here verbatim to meet the page limit.}
\begin{revisedblock}
    Constraints \eqref{eq:one_transit}-\eqref{eq:transit_time} are built on the assumption that a passenger taking transit chooses the minimum-time path within the candidate set, and they encode this assumption exactly. An alternative is to model the route choice within the transit route set also as a stochastic process. However, stochastic route choice within transit would break the decomposition algorithm in two ways. First, path shares would be a logit function of the frequencies of every line on the \gls{acr:od} pairs' candidate paths. Denote this logit function as $\sigma_p(\boldsymbol{z})$. The consistency would require a constraint in the form of $x^{pt}_{od,p}=x_{od}^{pt}\sigma(\boldsymbol{z})$, which introduces the product of a continuous variable and a nonlinear choice function of the binaries, and this destroys the convex structure. Second, if the decomposition instead treated transit path flows as free assignment variables as it does for \gls{acr:amod} flows on different routes, nothing would rationalize the resulting split in the second step. The transit fare is the same across paths, so no instrument plays for transit routes the role that it plays for the mode split, and price recovery is no longer guaranteed. Deterministic choice is also the appropriate model here, since transit path costs are flow-independent and the user equilibrium of route choice is therefore all-or-nothing on the minimum-time path.
\end{revisedblock}

\subsection{Richer mode-choice models}
The models and experiments in the main text use a binary  \gls{acr:amod}--\gls{acr:pt} choice model. For different applications, other types of choices might exist, and this appendix states which parts of the framework carry over to the different choice models. Before we list relevant extensions about choice models, it is useful to distinguish the implementation architecture from the particular oracle formulations developed in \Cref{sec:Method}. At the implementation level, passenger behavior enters only through the reaction mapping from operator actions to passenger responses. Once this response is substituted into the operators' utilities, the implementation payment is still determined by their unilateral deviation gains. Therefore, replacing the binary logit by a different choice model does not by itself change the $k$-implementation construction, provided that the corresponding unilateral best response values remain finite and attained, as required by~\Cref{prop:continuous}.

The binary-logit specification is more important for the tractability of the three oracles. The governing property is the one behind the decomposition of \Cref{sec:so_problem}. In the target oracle, the \gls{acr:amod} price is a contractible instrument that rationalizes the choice ratio $\frac{x^a_{od}}{x^{pt}_{od}}$, so the target oracle need not carry the logit constraint, and the price is recovered afterward through \Cref{prop:dcomp}. The best response oracles, by contrast, take the other operator's action as fixed and carry the logit explicitly. In the \gls{acr:amod} best response oracle, the binary logit relation is inverted to eliminate the price and mode-choice constraint in \Cref{prop:amod_reform}. In the \gls{acr:pt} best response oracle, the logit equality is instead transformed into the convex relaxation of \Cref{lemma:pt_relax}. Richer choice models therefore preserve the implementation logic but need not preserve these particular reformulations.

\paragraph{Sufficient condition for price-recovery decomposition.} With $J$ alternatives, a multinomial logit model contains $J-1$ independent share ratios. A sufficient condition is that each free utility difference can be adjusted independently through a contractible alternative-specific instrument (\gls{acr:amod} price in our case) and that the resulting mapping from instruments to utility differences is invertible. If this condition is not satisfied, the price recovery of the second step is unaffected, but the share ratios that are not controlled by an instrument remain as choice-consistency constraints in the first-step problem, which therefore makes the first step problem nonlinear. In the following, we discuss several relevant extensions regarding the choice model.

\begin{enumerate}
    \item \emph{Additional alternatives with exogenous utility.} Consider a multinomial-logit model containing \gls{acr:amod}, \gls{acr:pt}, and a set $\mathcal O$ of outside alternatives whose utilities are fixed with respect to operators' decisions. Examples may include walking or cycling represented by fixed generalized costs.

    For the \gls{acr:amod} best response, define the inclusive value of all non-\gls{acr:amod} alternatives as
    \begin{equation}\label{eq:inclusive}
    \bar V^{-a}_{od}
    :=
    \log\left(
        \exp(V^{pt}_{od})
        +
        \sum_{m\in\mathcal O}\exp(V^m_{od})
    \right).
    \end{equation}
    The \gls{acr:amod} share can then be written as
    \[\frac{\exp(V^a_{od})}
    {\exp(V^a_{od})+\exp(\bar V^{-a}_{od})}.\]
    Thus, from the \gls{acr:amod}'s perspective, all fixed alternatives collapse into one composite outside alternative. The inversion used in \Cref{prop:amod_reform} consequently retains the same structure after replacing $V^{pt}_{od}$ by $\bar V^{-a}_{od}$. Likewise, for the \gls{acr:pt} best response,
    \[\bar V^{-pt}_{od} := \log\left(\exp(V^a_{od})+ \sum_{m\in\mathcal O}\exp(V^m_{od})\right),\]
    and the transformation underlying \eqref{eq:pt_br_choice_new}--\eqref{constr:pt_br_choice_linearized} retains the same form with $V^a_{od}$ replaced by $\bar V^{-pt}_{od}$.

    The target oracle requires an additional consistency condition. For example, with one outside alternative $0$, multinomial logit implies
    \[\frac{x^{pt}_{od}}{x^0_{od}} = \exp\left(V^{pt}_{od}-V^0_{od}\right).\]
    
    The \gls{acr:amod} price can control the relative attractiveness of \gls{acr:amod}, but it cannot independently choose the split between \gls{acr:pt} and the outside alternative. Therefore, the price recovery still works, but the first-step problem must carry the constraint $x^{pt}_{od} = (\alpha_{od} - x^a_{od})\,\sigma_{od}(\boldsymbol z)$ with $\sigma_{od}(\boldsymbol z) = \exp(V^{pt}_{od}(\boldsymbol z)) / (\exp(V^{pt}_{od}(\boldsymbol z)) + \exp(V^0_{od}))$, a product of a continuous flow and a nonlinear function of the frequency binaries. This nonlinearity requires extra treatment to remove. If the expanded model contains enough independent alternative-specific instruments to rationalize all free share ratios, the first step decouples again and the original decomposition is recovered without change. This requires a different policy specification and is not assumed in the present paper.

    \item \emph{Alternatives with endogenous utility.} A more substantial extension arises when the utility of an outside alternative depends on operator decisions. One example is a private-car alternative sharing the congested road network with the \gls{acr:amod} fleet. In this case, the inclusive value for each operator's best response oracle becomes decision-dependent, so neither the inversion of \Cref{prop:amod_reform} nor the relaxation of \Cref{lemma:pt_relax} applies. For the target oracle, the first-step consistency constraint described in case 1 becomes flow-dependent and nonconvex. This case requires a different treatment. Another example is an intermodal alternative that uses AMoD to access transit. Its utility depends on the access-leg price and on the frequencies, so the consistency constraint retains price variables, and the first step of the target oracle can no longer be solved without the price variables.% GZ-CHECK: original read 'cannot be solved without time'; please confirm intended meaning

    \item \emph{Passenger heterogeneity.} Another extension of the choice model relates to passengers with different classes, e.g., classes $h \in \mathcal H$ with class-specific values of time $\mathrm{VOT}_h$. There are two cases here that lead to different applicability of the current models. If both the demand and price can be differentiated by class, all three oracles apply class by class, coupled only through aggregate congestion and capacity. However, with only one price per \gls{acr:od} pair, each class implies a supporting price $\pi_{od} = \mathrm{VOT}_h (t^{pt}_{od}-t^a_{od}) + P^{pt} - \ln(x^a_{od,h}/x^{pt}_{od,h})$ and these must agree, so the first step of the target oracle and the \gls{acr:amod} oracle retain cross-class consistency constraints. Therefore, the first step of the target oracle becomes nonlinear, and \Cref{prop:amod_reform} does not carry. The \gls{acr:pt} oracle takes prices as given and is unaffected.

    \item \emph{Crowding and reliability of transit.} Exogenous reliability terms can be incorporated directly as additive utility components if they depend only on frequencies. Provided its functional form of utility preserves convexity, the reliability term can be incorporated with the same finite-choice and big-$M$ machinery used for waiting time. Crowding is more consequential because it makes transit utility depend on passenger flows. The path-specific crowding makes transit route costs flow-dependent and therefore invalidates the minimum-path-time in the target oracle formulation (\eqref{eq:one_transit}-\eqref{eq:transit_time}). The \gls{acr:amod} oracle's inversion becomes implicit because its price moves the transit load. The \gls{acr:pt} oracle's log-odds constraint gains a load term whose exactness must be re-established.

    \item \emph{A third operator providing another alternative.} A final extension we discuss involves a third operator providing service in the system. If the operator is also under the \gls{acr:mcp}'s regulation, the implementation framework extends directly as \Cref{prop:continuous} holds for any finite number of players if the stated conditions about actions and payoffs are satisfied. However, each of the oracles needs to be reformulated, which can lead to completely new problems that the algorithms proposed in this paper do not address. If the operator is outside the regulation design, its response to the regulated operators' actions must be predicted inside each best response oracle, which is identical to case 1 or 2 discussed above.
\end{enumerate}

\end{revisedblock}

\section{Case study setup} \label{app:case}
Experiments require travel demand and a structured multimodal transportation network as inputs.
\subsection{Demand}
Travel demand is derived from the New York City Taxi and Limousine Commission (TLC) trip record dataset~\citep{nyc_tlc_trip_data}. 
We extract trips originating and ending within Manhattan during the morning peak period (9:00-11:00 AM) on May 15, 2024. 
The TLC dataset aggregates locations into 64 distinct taxi zones. \revised{Within the 64 Manhattan zones this window contains 23,730 trips. We first exclude intra-zone trips (4.2\%), which have no network distance at zone resolution. We then retain \gls{acr:od} pairs with at least four trips and a feasible subway itinerary within the walking and transfer limits of \Cref{app:network}. The first condition removes 1,257 sparse pairs carrying 9.8\% of trips;% GZ-CHECK: 9.8% and 22.7% reconcile with 15,364 retained trips only as shares of the 22,737 inter-zone trips, not of all 23,730; please state the base
 the second removes 460 pairs carrying 22.7\% of trips, for which the subway is not a usable alternative and the \gls{acr:amod}--\gls{acr:pt} choice studied here does not arise. The resulting instance has 1,199 \gls{acr:od} pairs and 15,364 trips, 65\% of all trips in the window.
}
To spatially align this zone-level demand with the node-based walking network, we project the centroid of each TLC zone to its nearest walking network node. This mapping yields a consistent, routable representation in which each demand pair $(o,d) \in \tilde{D}$ strictly corresponds to nodes within the walking graph.

\subsection{Networks} \label{app:network}
The multimodal networks consist of a driving network, a transit network, and a walking network. 
The driving and walking networks are constructed from \gls{acr:osm}~\citep{OpenStreetMap}. 
For the driving network, we retain road type, segment length, and number of lanes. 
Road capacity~$c_{ij}$ in the BPR function is set to $700$ veh/h per lane. 
% The driving network includes the type of road, length, and number of lanes. 
% The number of lanes is used to calculate the capacity $c_{ij}$ in the BPR function, with $700 veh/h$ for each lane. Exogenous background traffic $b_{ij}$ is scaled relative to this capacity (e.g., a background traffic level of 10\% indicates that $b_{ij}/c_{ij} = 0.1$).

Because the raw \gls{acr:osm} representation is designed for geographic completeness rather than optimization tractability, it contains many intermediate degree-2 nodes and minor local segments that are not essential for the present operator-level routing problem. 
We therefore contract degree-2 nodes along continuous segments and prune minor local roads. \revised{Contraction of degree-2 nodes preserves segment lengths and free-flow travel times exactly.}
This preprocessing therefore simplifies the instance while preserving the corridor structure relevant for pricing, routing, rebalancing, and congestion. 
For each $(o,d)\in\od$, the candidate road-path set $P_a^{od}$ is constructed as the top-5 shortest-distance paths on the simplified driving network. \revised{We analyzed the results and found that this cap is not binding. At the target profile, 97 to 99.7\% of \gls{acr:od} pairs use a single \gls{acr:amod} route and none uses more than two of the five candidates. Re-solving the implementation problems of \Cref{sec:payment} with the candidate set truncated to one, two or three routes per \gls{acr:od} pair changes the social cost of the target by at most 1.7\% and preserves the payment profile of \Cref{fig:k_res}. The total payment moves by at most 7\% at all traffic levels.
}

The \gls{acr:pt} network is built from the MTA subway GTFS feed~\citep{mta_subway_gtfs}. The resulting Manhattan transit layer contains 58 directed lines and 152 stations. The fixed transit fare is set to $P^{pt}=\$3.0$ per trip. For each \gls{acr:od} pair, the reference transit path $p_{od}^{pt}$ is computed with the R5 routing engine~\citep{r5py_2022} for a weekday $09:00$ departure. \revised{Itineraries are de-duplicated on their ordered sequence of boarded lines, filtered to at most 1,200 m of total walking and at most three transit legs, and ranked by in-vehicle plus access time; the top five are retained. Passengers are assigned to the minimum-time path within this candidate set given the frequency design. The fastest itinerary is the reference path used in the single-path experiments of \Cref{sec:Results}. The sets contain 1, 2, 3, 4 and 5 candidates for 297, 351, 247, 155 and 150 OD pairs respectively.% GZ-CHECK: these counts sum to 1,200 but App. I.1 reports 1,199 OD pairs
}

\subsection{Calibration and computational settings} Unless otherwise specified, the numerical experiments utilize the parameters summarized below. The \gls{acr:amod} operational cost is set to $\$0.60$/km~\citep{bosch2018cost}. The \gls{acr:pt} system operates with a vehicle capacity of $900$ passengers per train~\citep{fta_mta_2004_feis_ch5b} and incurs an operational cost of $\$320$ per Vehicle Revenue Hour (VRH)~\citep{fta_ntst_2024}. The set of admissible transit frequencies is defined as $S := \{2, 3, 4, 5, 6, 12, 20\}$ trains per hour.

In the sensitivity analyses in the experiments, we vary the regularization parameter $\theta$, the background traffic level BL, the \gls{acr:amod} operating cost, and the \gls{acr:pt} parameters $(C_l,K_l)$ to assess the robustness of the framework across materially different operating regimes. Unless otherwise specified, the implementation-payment study uses the baseline calibration above together with the regularization choice selected in \cref{sec:res_so}.

All experiments were run on an HPC cluster with CPU-only nodes (96 cores and approximately 520 GB RAM per node) using Gurobi 10.0.2. Each solve used 16 CPU threads and 64 GB RAM.

\begin{revisedblock}
\section{Robustness to the transit candidate-path set} \label{app:multipath}
The main results of \Cref{sec:Results} represent transit routes by one reference path per \gls{acr:od} pair. This appendix repeats the implementation analysis with candidate transit path sets of up to $K \in \{2,3,5\}$ paths per pair, and the passengers take the minimum-time path as in the formulation of \Cref{sec:so_problem}. All runs use the baseline calibration of \Cref{app:case} and the candidate sets are generated as described in \Cref{app:case}. Throughout, we distinguish quantities that move materially with the candidate set from conclusions that are robust to it.

\subsection{Candidate sets and usage}
We start by showing the distribution of transit flow on the different candidate paths. \Cref{fig:usage}(a) shows the percentage of transit passenger flow carried by each of the candidate paths under varying $K$ level in the target profile. At $K=5$, around $61-69\%$ of the transit flow is still routed on the reference path when other candidates are available. This shows that enlarging the candidate path set indeed changes the solutions of the target profile.

\begin{figure}[t]
    \centering
    \includegraphics[width=0.7\linewidth]{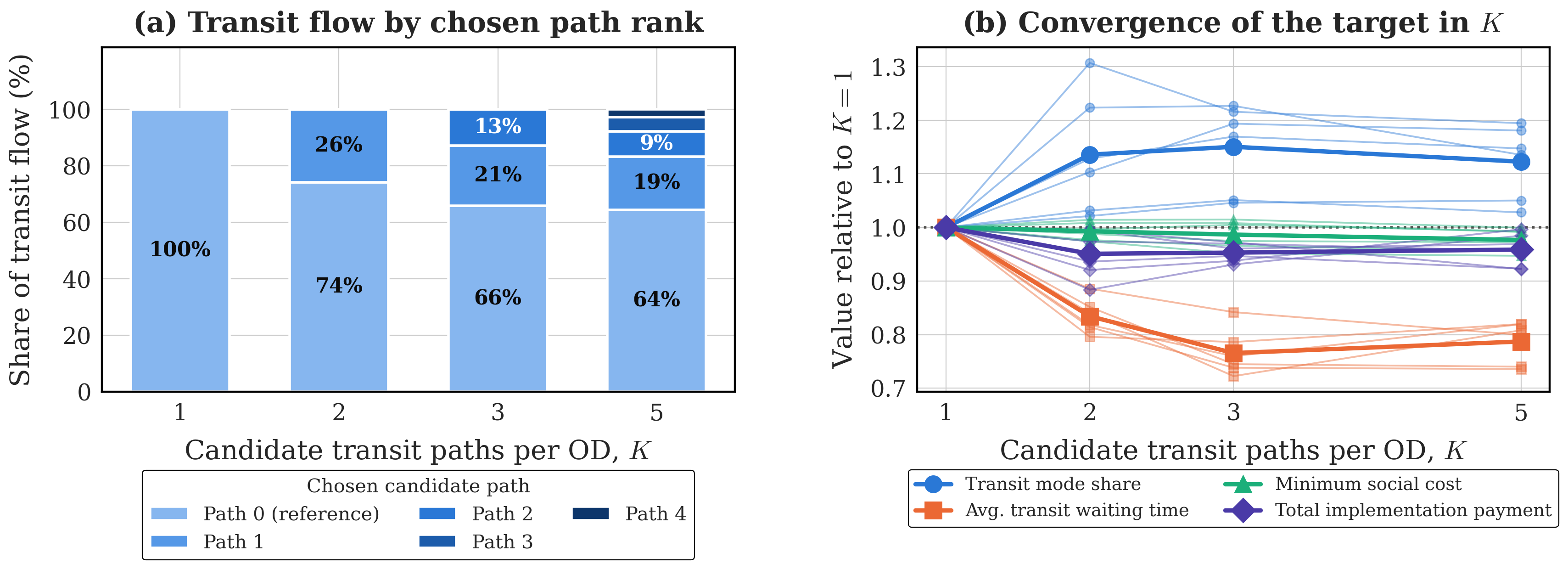}
    \caption{Candidate path usage and effect on passenger service and implementation payment. (a) Share of transit passenger flow on each candidate path with different number of candidate paths $K$, averaged over BL. Path 0 is the reference path used in the single-path baseline. (b) Passenger service and implementation payment under different $K$ and BL.}
    \label{fig:usage}
\end{figure}

\subsection{Effect on passenger service and total implementation payment}
As shown in \Cref{fig:usage}(b), passenger-level outcomes move materially under different $K$. This is because, with multiple candidate paths, waiting time enters route choice, so the single-path set overstates transit travel time. Specifically, at $K=5$, the average waiting time in the target profile falls by $18-26\%$ relative to $K=1$. Transit mode share rises by a few percentage points. These are the effects the richer route representation is expected to have, and they are the reason the single-path setting is treated in \Cref{sec:Results} as a computational baseline rather than a behavioral assumption. Compared to the average transit waiting time and mode share, the social cost of the target and the total implementation payment stay closer to the single-path case. The next section presents the payment analysis in detail to show the robustness of the incentive conclusions of \Cref{sec:payment}.

\subsection{Effect on detailed implementation payment and deviation mechanisms}
We repeat the experiments in \Cref{sec:payment} with one series per $K$, and the results are presented in \Cref{fig:k_res_robust} and \ref{fig:k_analysis_robust}. As in \Cref{fig:k_res_robust}(c), the U-shape in background traffic remains at every $K$. The total payment stays within 12\% of its single-path value over
the whole grid, and the \gls{acr:amod} share of the payment remains between 70 and 81\% throughout. The one component that moves materially is the transit payment, which changes by $-14\%$ to $+30\%$ at $K = 5$, most under light traffic where alternative paths redistribute ridership across lines. This effect is diluted in the total because transit accounts for roughly a quarter of it. \Cref{fig:k_analysis_robust} shows that the best-response-to-target deviation ratios are nearly coincident across $K$, with the same monotonicity in traffic, the same magnitudes, and the same crossings of unity, so the mechanisms identified in \Cref{sec:payment} do not depend on the route representation.

\begin{figure}[tb]
    \centering
    \begin{minipage}[t]{0.48\textwidth}
        \centering
        \includegraphics[width=\linewidth]{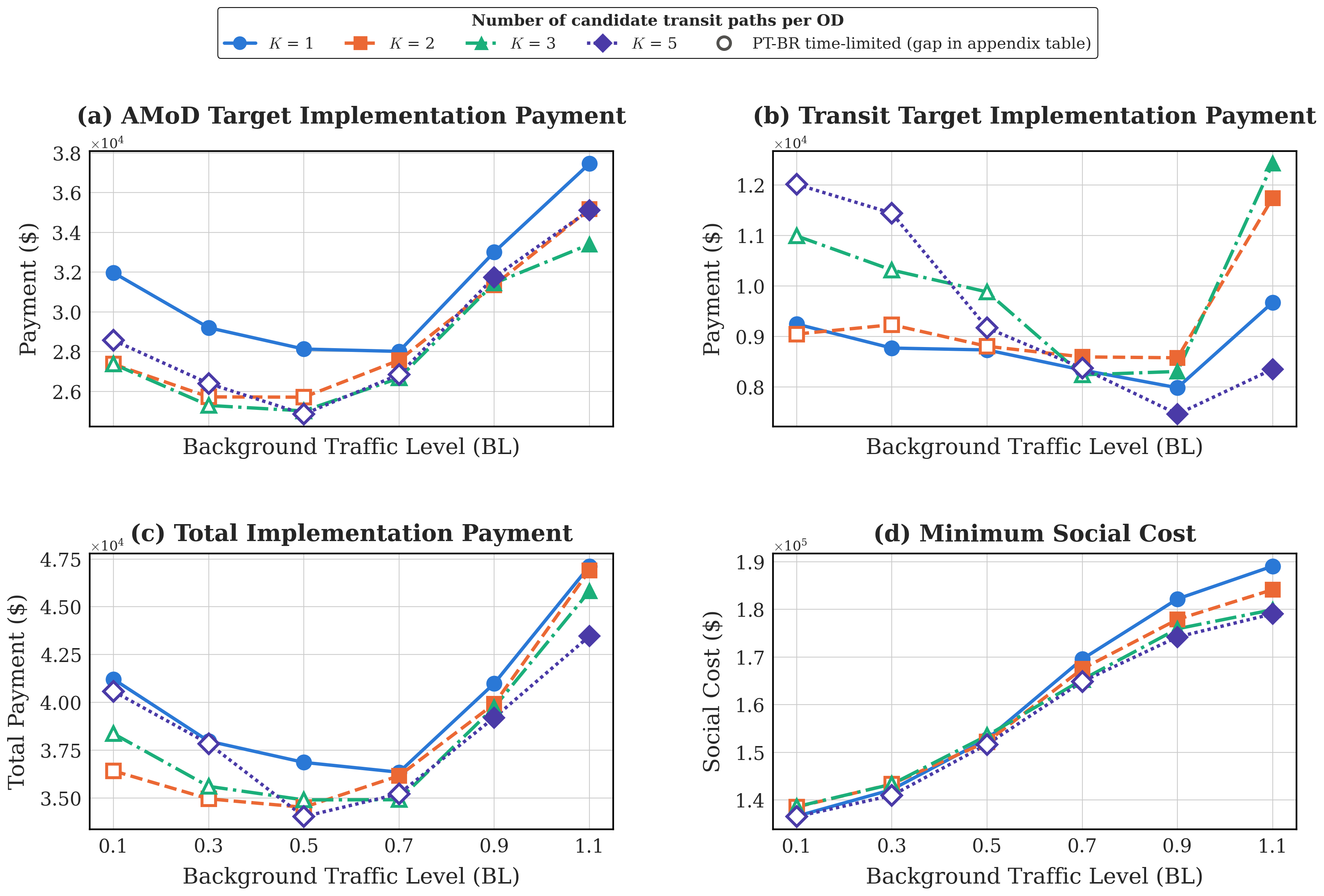}
        \caption{\revised{\Cref{fig:k_res} with different $K$ levels.}}
        \label{fig:k_res_robust}
    \end{minipage}\hfill
    \begin{minipage}[t]{0.48\textwidth}
        \centering
        \includegraphics[width=\linewidth]{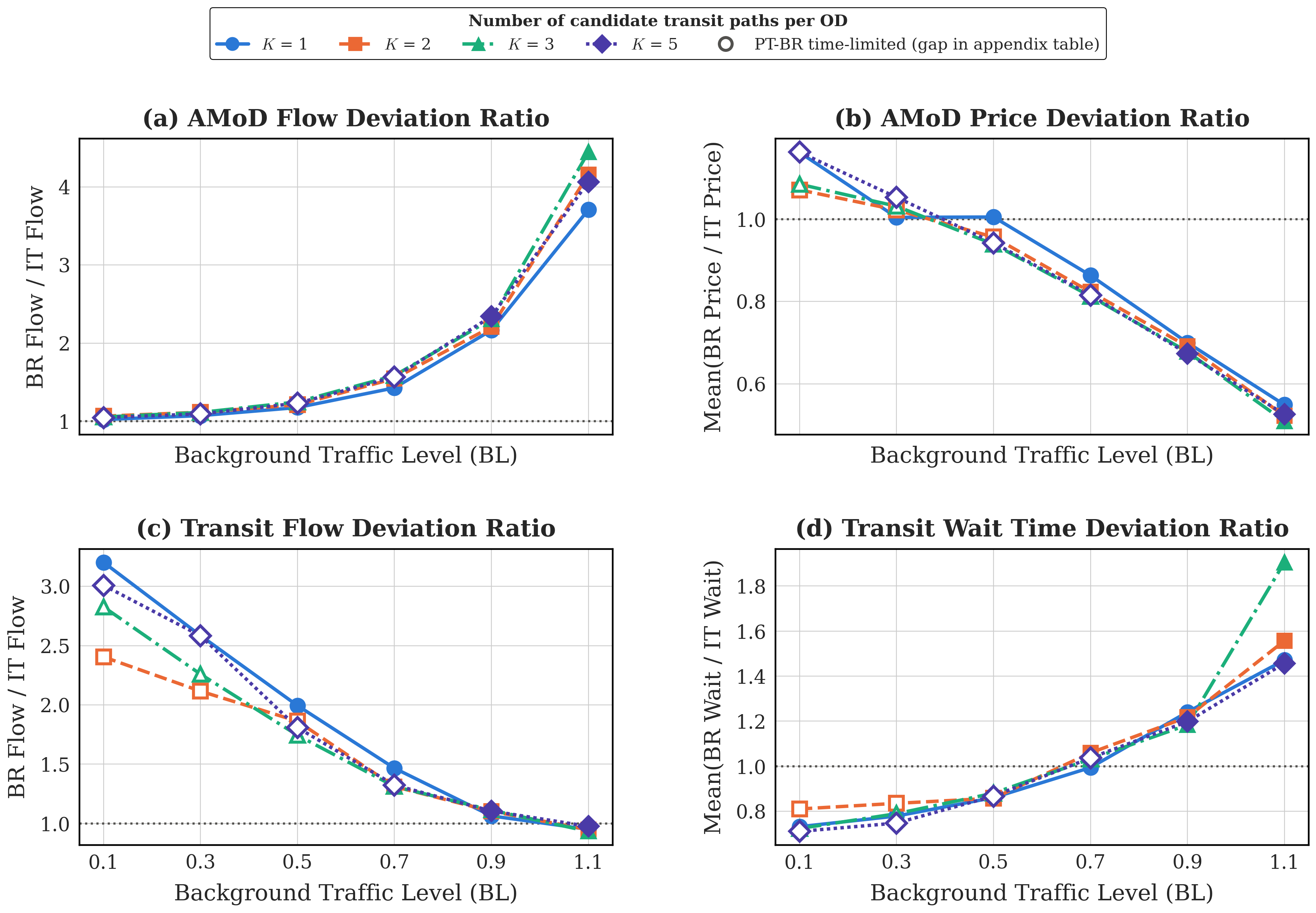}
        \caption{\revised{\Cref{fig:k_analysis} with different $K$ levels.}}
        \label{fig:k_analysis_robust}
    \end{minipage}
\end{figure}

\subsection{Computational cost and bounds}\label{app:multipath_time}
Although different $K$ does not change the incentive conclusions, enlarging the candidate set is expensive. \Cref{fig:time}(a) shows that the target profile oracle solved to the same 1\% tolerance takes 1--32 minutes at $K=1$ and 65--575 minutes at $K=5$, hence increasing $K$ from 1 to 5 increases the computation time between one and two orders of magnitude on each instance. The transit oracle is affected more severely regarding the computational complexity. As shown in \Cref{fig:time}(b) and \Cref{tab:pathgaps}, at $K=1$, the oracle solves to within about a 1\% gap on every instance in 5--47 minutes, while at $K \ge 2$ it reaches the 11-hour time limit on 11 of the 18 instances and terminates with certified gaps of up to 6.8\%. The limit-bound instances are those with light and intermediate background traffic. Under heavy traffic the oracle still closes within the limit in 52--497 minutes. A plausible reading is that under heavy traffic transit demand is high and capacity binds, so few frequency profiles are competitive and the relaxation bound is tight, whereas under light traffic many frequency profiles yield similar profit. The runtimes of the limit-bound instances are lower bounds and are reported as such. The \gls{acr:amod} best response oracle is unaffected, since the candidate set enters it only through a fixed transit utility.

\begin{figure}[tb]
    \centering
    \includegraphics[width=0.6\linewidth]{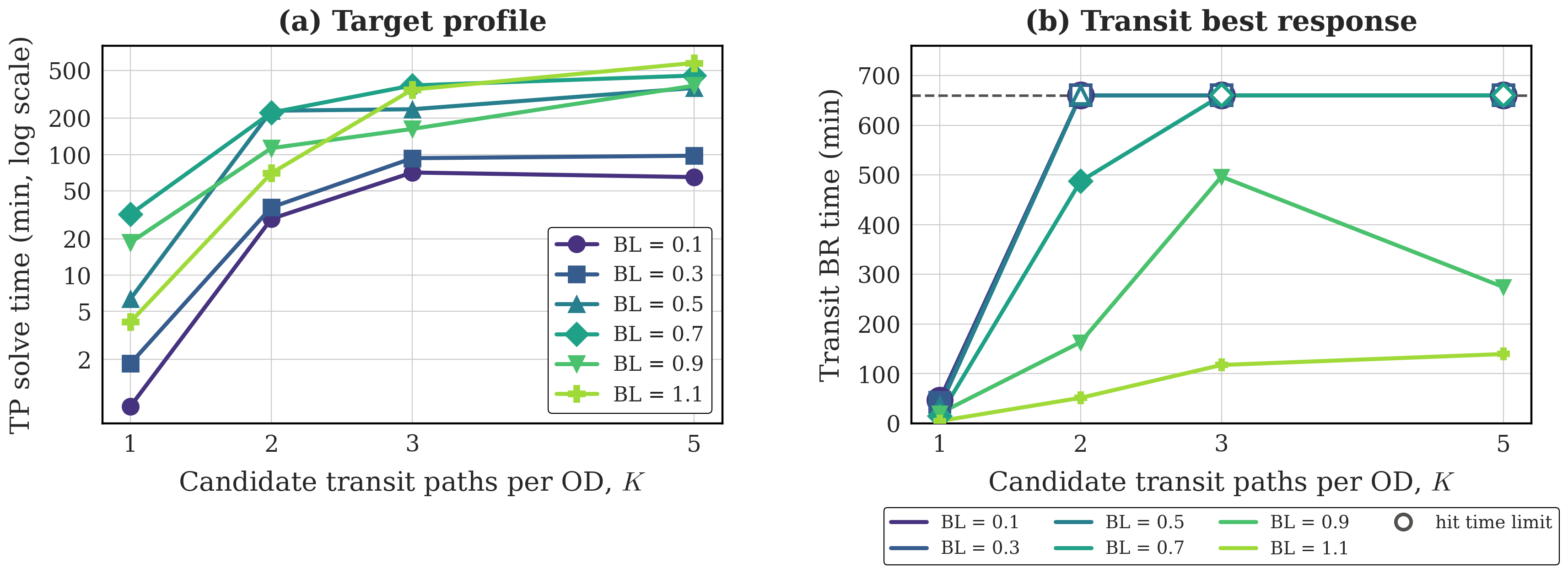}
    \caption{\revised{Target profile and transit best response solution time under different $K$. All experiments are run with the same thread and memory.}}
    \label{fig:time}
\end{figure}

\begin{table}[htbp]
\centering\small
\caption{\revised{Solver statistics of candidate-set size $K$
and background traffic level sweep. The PT oracle gap is Gurobi's
certified relative MIP gap at termination. The target oracle closed to its 1\%
tolerance in every cell and the \gls{acr:amod} oracle is unaffected by $K$. PT-oracle time limit 11 h. $^\dagger$ time limit reached.}}
\label{tab:pathgaps}
\begin{tabular}{l c rrrrrr}
\toprule
 & $K$ & BL = 0.1 & BL = 0.3 & BL = 0.5 & BL = 0.7 & BL = 0.9 & BL = 1.1 \\
\midrule
\multirow{4}{*}{Target oracle time (min)} & 1 & 0.8 & 1.8 & 6.3 & 32.0 & 18.7 & 4.1 \\
 & 2 & 29.3 & 36.5 & 231.0 & 223.0 & 113.0 & 70.1 \\
 & 3 & 71.0 & 93.3 & 237.6 & 374.3 & 163.1 & 345.9 \\
 & 5 & 65.0 & 97.8 & 356.6 & 452.3 & 372.5 & 575.3 \\
\midrule
\multirow{4}{*}{AMoD oracle time (min)} & 1 & 3.3 & 6.6 & 5.8 & 4.8 & 9.1 & 8.9 \\
 & 2 & 2.3 & 6.5 & 4.7 & 5.6 & 12.0 & 9.1 \\
 & 3 & 6.4 & 3.8 & 7.3 & 5.5 & 10.6 & 10.9 \\
 & 5 & 5.8 & 5.0 & 5.2 & 6.5 & 9.1 & 9.2 \\
\midrule
\multirow{4}{*}{PT oracle time (min)} & 1 & 47 & 42 & 33 & 15 & 20 & 5 \\
 & 2 & 660$^\dagger$ & 660$^\dagger$ & 660$^\dagger$ & 488 & 164 & 52 \\
 & 3 & 660$^\dagger$ & 660$^\dagger$ & 660$^\dagger$ & 660$^\dagger$ & 497 & 118 \\
 & 5 & 660$^\dagger$ & 660$^\dagger$ & 660$^\dagger$ & 660$^\dagger$ & 275 & 140 \\
\midrule
\multirow{4}{*}{PT oracle certified gap (\%)} & 1 & 0.99 & 0.97 & 0.99 & 0.97 & 0.97 & 0.95 \\
 & 2 & 3.58 & 2.25 & 1.61 & 0.99 & 0.94 & 0.83 \\
 & 3 & 6.80 & 6.58 & 5.88 & 2.44 & 0.99 & 0.90 \\
 & 5 & 6.05 & 3.66 & 3.57 & 2.86 & 0.94 & 0.97 \\
\bottomrule
\end{tabular}
\end{table}

\subsection{Summary}
This study indicates that the principal incentive conclusions are robust to the size of the transit candidate set under minimum-time route choice within that set at the baseline \gls{acr:amod} operating cost. \Cref{fig:k_res} and \Cref{fig:k_analysis} show that varying operation cost shifts the level of the implementation payment but not its profile across congestion levels or its decomposition between operators, so we vary only the background traffic level here. At the same time, enlarging the candidate set increases the computational burden dramatically, which is demonstrated by the one to two orders of magnitude growth in solution time of the target oracle, and the time limit bound on most instances of the transit oracle. The extensive sweeps of \Cref{sec:Results} therefore use the single reference path, based on which the conclusions this appendix shows to be robust.

\end{revisedblock}

\begin{revisedblock}
\section{Detailed metrics of the target-payment trade-off}
\label{app:theta_table}

\Cref{tab:theta} lists, for every regularization weight $\theta$ in the family
of targets studied in \Cref{sec:res_so} and \Cref{sec:tradeoff}, the optimality gap
of the target relative to the unregularized social optimum, the implementation
payment and its decomposition between the two operators, and the share of
\gls{acr:od} pairs with a negative recovered price. These are the values plotted in \Cref{fig:so_theta}
and \Cref{fig:tradeoff}. All targets were solved to a $0.1\%$ MIP tolerance
with the baseline calibration of \Cref{app:case}.

\begin{table}[htbp]
\centering\footnotesize
\setlength{\tabcolsep}{4pt}
\caption{\revised{Optimality gap, implementation payment, and recovered-price sign along
the family of regularized targets. Gap: social cost of the target relative to
the unregularized social optimum (\%). $V$, $V_a$, $V_{pt}$: total, \gls{acr:amod}
and \gls{acr:pt} implementation payment (\$1000). Neg.: share of \gls{acr:od}
pairs with a negative recovered \gls{acr:amod} price (\%).}}
\label{tab:theta}
\begin{tabular}{c rrrrr rrrrr rrrrr}
\toprule
 & \multicolumn{5}{c}{BL = 0.1} & \multicolumn{5}{c}{BL = 0.5} & \multicolumn{5}{c}{BL = 1.0} \\
\cmidrule(lr){2-6} \cmidrule(lr){7-11} \cmidrule(lr){12-16}
$\theta$ & Gap & $V$ & $V_a$ & $V_{pt}$ & Neg. & Gap & $V$ & $V_a$ & $V_{pt}$ & Neg. & Gap & $V$ & $V_a$ & $V_{pt}$ & Neg. \\
\midrule
0.1 & 20.6 & 39.5 & 28.0 & 11.5 & 0.0 & 15.0 & 39.9 & 28.2 & 11.7 & 0.0 & 11.5 & 39.5 & 28.3 & 11.2 & 0.0 \\
0.2 & 12.8 & 35.9 & 22.7 & 13.3 & 0.0 & 10.5 & 35.7 & 24.1 & 11.6 & 0.0 & 6.7 & 40.1 & 30.8 & 9.3 & 0.0 \\
0.3 & 8.3 & 34.2 & 21.8 & 12.5 & 0.0 & 7.2 & 34.3 & 23.1 & 11.3 & 0.0 & 4.3 & 41.9 & 32.9 & 9.0 & 0.0 \\
0.4 & 5.5 & 35.6 & 24.9 & 10.7 & 0.0 & 5.7 & 35.3 & 24.6 & 10.7 & 0.0 & 3.1 & 43.2 & 34.0 & 9.2 & 0.0 \\
0.5 & 3.5 & 40.9 & 31.7 & 9.2 & 0.0 & 3.6 & 36.3 & 28.0 & 8.4 & 0.0 & 2.3 & 44.5 & 35.8 & 8.6 & 0.0 \\
0.6 & 2.3 & 50.2 & 43.1 & 7.1 & 1.7 & 2.8 & 39.8 & 33.1 & 6.8 & 0.6 & 1.8 & 46.0 & 37.2 & 8.7 & 0.1 \\
0.7 & 1.7 & 60.6 & 55.2 & 5.4 & 5.5 & 2.2 & 45.4 & 39.5 & 5.9 & 2.3 & 1.5 & 46.6 & 38.1 & 8.5 & 0.1 \\
0.8 & 1.5 & 67.6 & 63.0 & 4.5 & 7.8 & 1.9 & 51.5 & 46.4 & 5.1 & 3.8 & 1.2 & 47.8 & 39.1 & 8.7 & 0.3 \\
0.9 & 1.0 & 82.5 & 79.4 & 3.1 & 13.0 & 1.3 & 63.7 & 58.7 & 5.0 & 8.3 & 1.1 & 48.4 & 39.4 & 9.1 & 0.3 \\
1.0 & 0.8 & 93.6 & 91.2 & 2.4 & 17.5 & 1.1 & 72.6 & 68.3 & 4.4 & 11.3 & 0.8 & 49.9 & 40.8 & 9.1 & 0.3 \\
2.0 & 0.2 & 163.9 & 162.8 & 1.1 & 49.3 & 0.5 & 118.9 & 116.3 & 2.6 & 32.9 & 0.2 & 60.4 & 49.8 & 10.6 & 2.6 \\
\bottomrule
\end{tabular}
\end{table}

\end{revisedblock}

\section{Exactness map of the \gls{acr:pt} relaxation}\label{app:pt_exact}
\gzmoved{the following text and table were moved verbatim from \cref{sec:res_pt} to meet the page limit.}
This appendix details the exactness analysis of the \gls{acr:pt} deviation oracle summarized in \cref{sec:res_pt}.
The key question is not only whether the \gls{acr:micp} relaxation is exact, but also how informative it remains when exactness fails. 
To this end, we test the relaxation across a grid of transit capacities $C_l$ and operating costs $K_l$. 
The tested parameter bounds are derived from standard transit specifications published by federal agencies~\citep{fta_mta_2004_feis_ch5b,fta_ntst_2024,tcqsm_bus_capacity_2013}, ensuring they reflect plausible, real-world operational constraints.
These experiments are designed to map the exactness boundary and to distinguish benign numerical inexactness from a genuine change in the underlying operating regime.
\Cref{tab:pt_exact} reports, for each scenario, whether the relaxation is exact and, when it is not, the relative gap between the relaxed upper bound and a feasible capacity-repaired lower bound. 
The recovery procedure starts from the relaxed solution and incrementally increases bottleneck frequencies until the resulting induced ridership satisfies all capacity constraints. 
Hence, when exactness fails, the table still reports a valid interval between a relaxed upper bound and a feasible lower bound for the original equality-constrained problem.

The low-capacity rows in \cref{tab:pt_exact}, namely $C_l\in\{60,80,120,140\}$, should be read as stress tests used to expose the exactness boundary; the baseline subway-scale calibration corresponds to $C_l=900$. 
The results follow a clear pattern. 
Exactness holds when capacity is abundant relative to induced demand. As background traffic increases, more travelers are pushed toward transit, capacity becomes tighter, and the relaxation is more likely to be inexact. 
Importantly, at the baseline operating cost $K_l=320$ and baseline capacity $C_l=900$, the relaxation is exact at $\mathrm{BL}=0.3$ and remains tightly bounded at the more congested tested regimes $\mathrm{BL}=0.7$ and $\mathrm{BL}=0.9$, with gaps of only $2.6\%$ and $3.0\%$, respectively. 
Thus, in the parameter region most relevant for the main Manhattan case study, the \gls{acr:pt} deviation term is either exact or very tightly bounded.

The large gaps that appear in the low-capacity rows should not be read as numerical instability. 
Rather, they signal a different operating regime: a capacity-truncated or demand-shedding state. 
In such cases, the strict logit-equality model implies more transit demand than the physical system can accommodate, and the relaxed solution should be interpreted as the profit-maximizing service level under limited capacity rather than as a full-demand operating point. This interpretation is consistent with \cref{remark:pt_relax_int}: once the capacity constraint binds, the gap reflects a meaningful divergence between unconstrained choice demand and physically serviceable demand. 
From the municipality's viewpoint, this is precisely the kind of misalignment that can generate a nontrivial implementation burden. 

\begin{table}[htbp]
\centering
\small
\caption{Exactness and relaxation-gap sensitivity analysis for the \gls{acr:pt} best-response oracle across transit capacities and operating costs.}
\label{tab:pt_exact}
\begin{tabular}{cc c S[table-format=2.1] c S[table-format=3.1] c S[table-format=3.1]}
\toprule
& & \multicolumn{2}{c}{\textbf{BL = 0.3}} & \multicolumn{2}{c}{\textbf{BL = 0.7}} & \multicolumn{2}{c}{\textbf{BL = 0.9}} \\
\cmidrule(lr){3-4} \cmidrule(lr){5-6} \cmidrule(lr){7-8}
\textbf{$K_l$ (\$/VRH)} & \textbf{$C_l$ (pax/veh)} & \textbf{Exact} & \textbf{Gap (\%)} & \textbf{Exact} & \textbf{Gap (\%)} & \textbf{Exact} & \textbf{Gap (\%)} \\ 
\midrule
\multirow{2}{*}{260} & 60   & N & 71.6 & N & 99.6 & N & 130.0 \\
                     & 80   & N & 50.0 & N & 82.5  & N & 126.9 \\ \addlinespace
\multirow{2}{*}{300} & 120  & N & 3.4  & N & 3.7  & N & 23.2  \\
                     & 140  & E & 0.0  & N & 3.4   & N & 4.6   \\ \addlinespace
\multirow{2}{*}{320} & 900  & E & 0.0  & N & 2.6   & N & 3.0   \\
                     & 1400 & E & 0.0  & N & 2.4   & N & 2.9   \\ 
\bottomrule
\addlinespace
\multicolumn{8}{l}{\footnotesize \textit{Note:} All instances solved with MIP gap = 0.01. E: Exact; N: Non-exact.}
\end{tabular}
\end{table}
\end{APPENDICES}

\end{document}